\documentclass[11pt,leqno]{amsart}
\usepackage{amsfonts}
\usepackage{color,xcolor}
\usepackage{graphicx}
\usepackage{float}
\usepackage{subfigure}
\usepackage{epsfig}
\usepackage{manfnt}
\usepackage{enumerate}
\newtheorem{thm}{Theorem}[section]
\newtheorem{cor}[thm]{Corollary}
\newtheorem{lem}[thm]{Lemma}

\newtheorem{prop}[thm]{Proposition}
\newtheorem{rem}[thm]{Remark}

\newcommand{\Z}{\mathcal{Z}}

 \numberwithin{equation}{section}

\newcommand{\be}{\begin{equation}}
\newcommand{\ee}{\end{equation}}
\newcommand\bes{\begin{eqnarray}}
\newcommand\ees{\end{eqnarray}}
\newcommand{\bess}{\begin{eqnarray*}}
\newcommand{\eess}{\end{eqnarray*}}

\newcommand{\R}{{\mathbb{R}}}

\begin{document}

\title[Convergence of the Solutions of RPME]
{Convergence of Solutions of the Porous Medium Equation with Reactions}
\author[B. Lou and M. Zhou]{Bendong Lou \ and\ Maolin Zhou}

\thanks{{\bf B. Lou}: Mathematics and Science College, Shanghai Normal University,
Shanghai, 200234, China. Email: {\tt lou@shnu.edu.cn}}

\thanks{{\bf M. Zhou}:
Chern Institute of Mathematics, Nankai University, Tianjin, 300071, China.
Email: {\tt zhouml123@nankai.edu.cn}}

\thanks{This research was partially supported by the National Key Research and Development Program of China (2021YFA1002400) and the NSF of China (No. 12071299).}
\maketitle

\noindent
{\bf Abstract.}
Consider the Cauchy problem of one dimensional porous medium equation (PME) with reactions. We first prove a general convergence result, that is, any bounded global solution starting at a nonnegative compactly supported initial data converges as $t\to \infty$ to a nonnegative zero of the reaction term or a ground state stationary solution. Based on it, we give out a complete classification on the asymptotic behaviors of the solutions for PME with monostable, bistable and combustion types of nonlinearities.

\smallskip


\smallskip


\setlength{\baselineskip}{16pt}{\setlength\arraycolsep{2pt}

\tableofcontents

\section{Introduction}

Consider the following reaction diffusion equation (abbreviated as RDE in the following)
\begin{equation}\label{eq}
 u_t = u_{xx} + f(u), \quad x\in \R,\ t>0.
\end{equation}
In the last decades, many authors gave systemic qualitative study for this equation. The reaction terms include the following three typical types:
\begin{center}
(f$_M$)\ \ monostable case, \ \ \ \ \ \ (f$_B$)\ \ bistable case,\ \ \ \ \ \
(f$_C$)\ \ combustion case.
\end{center}
In the monostable case, we assume $f=f(u)\in C^1([0,\infty))$ and
\begin{equation}\label{mono}
f(0)=f(1)=0, \quad f'(0)>0, \quad  f'(1)<0,\quad (1-u)f(u) >0 \ \mbox{for } u>0, u\not= 1.
\end{equation}
The well known example is the logistic term, that is, the Verhulst law: $f(u) =u(1-u)$.
In the bistable case, we assume $f=f(u)\in C^1 ([0,\infty))$ and
\begin{equation}\label{bi}
f(0)=f(\theta)= f(1)=0, \quad f(u) \left\{
\begin{array}{l}
<0 \ \ \mbox{in } (0,\theta),\\
>0\ \  \mbox{in } (\theta, 1),\\
< 0\ \ \mbox{in } (1,\infty),
\end{array} \right.
\quad \int_0^1 f(s) ds >0,
\end{equation}
for some $\theta\in (0,1)$,  $f'(0)<0$, $f'(1)<0$. A typical example is $f(u) = u(u-\theta)(1-u)$ with $\theta \in (0, \frac{1}{2})$.
In the combustion case, we assume there exists $\theta\in (0,1)$ such that $f \in C^1([\theta,\infty])$ and
\begin{equation}\label{combus}
\left\{
 \begin{array}{l}
 f(u)=0 \ \ \mbox{in } [0,\theta], \quad f(u) >0 \ \mbox{in }
(\theta,1), \quad f'(1)<0,\quad f(u) < 0 \ \mbox{in } (1, \infty),\\
 \mbox{there exists a small }\delta>0 \mbox{ such that } f'(u)>0 \mbox{ for } u\in (\theta, \theta +\delta].
 \end{array}
\right.
\end{equation}

In 1937, Fisher \cite{F} and Kolmogorov et al. \cite{KPP} studied the RDE with logistic reaction in their pioneer works. In particular, they specified the {\it traveling wave solutions}. In 1975 and 1978, Aronson and Weinberger \cite{AW1, AW2} studied the asymptotic behavior for the solutions of the Cauchy problem of RDE.
Among others, in the monostable case, they proved the {\it hair-trigger effect}, which says that {\it spreading happens} (that is, $u\to 1$ as $t\to \infty$, which is also called {\it persistence} or {\it propagation phenomenon}) for any positive solution of the monostable equation; in the bistable case, they gave  sufficient conditions for {\it spreading} and that for {\it vanishing} (that is, $u\to 0$ as $t\to \infty$, which is also called {\it extinction phenomenon}). In 1977, Fife and McLeod \cite{FM} also studied
the bistable equation, proved the existence and stability of the traveling wave solution. In 2006, Zlato\v{s} \cite{Zla} gave further systemic study on the asymptotic behavior for the solutions of bistable and
combustion equations. More precisely, he considered the Cauchy problem of \eqref{eq} with initial data $u(x,0)=\chi_{[-L,L]}(x)$ (which denotes the characteristic function over $[-L,L]$), and proved a trichotomy result: there exists a sharp value $L_*>0$ such that when $L>L_*$ (resp. $L<L_*$), spreading (resp. vanishing) happens for the solution. In addition, the solution develops to a transition one  if and only if $L=L_*$, which is a ground state solution
in the bistable case and the largest zero $\theta\in (0,1)$ of $f$ in the combustion case. In 2010, Du and Matano \cite{DM}
extended these results to the Cauchy problem with general initial data like $u(x,0) =\sigma \phi(x)$ for positive number $\sigma $ and continuous, nonnegative, compactly supported function $\phi$. For bistable and combustion equations, they also proved the trichotomy and sharp transition results as in \cite{Zla}. In 2015, Du and Lou \cite{DL} proved similar results for the nonlinear Stefan problem.

In this paper, we study the porous medium equation with reactions:
$$
\mbox{(CP)} \hskip 40mm
 \left\{
 \begin{array}{ll}
 u_t = (u^m)_{xx} + f(u), &\ \  x\in \R,\ t>0,\\
 u(x,0)=u_0(x), &\ \  x\in \R,
 \end{array}
\right.
 \hskip 50mm
$$
where $m>1$ is a constant. This problem can be used to model population dynamics with diffusion flux depending on the population density, the combustion, propagation of interfaces and nerve impulse propagation phenomena in porous media, as well as the propagation of intergalactic civilizations in the field of astronomy (cf. \cite{GN, GurMac, NS, SGM} etc.).
For simplicity, we use PME/RPME to refer to the porous media equation in (CP) without/with a reaction.

The well-posedness of the problem (CP) was studied in \cite{A1969, A1, Sacks, Vaz-book} etc..
For the qualitative property, the monostable equation was considered widely. For example, in 1979, Aronson \cite{A1} considered the problem with logistic reaction. He studied the existence of traveling wave solutions and compared the results with that in RDE. It was shown that (see also \cite{GK, PV, She}) the monostable RPME has a traveling wave solution $u(x,t)=\widetilde{U}(x-ct)$ if and only if $m\geq 1$ and only for the speed $c\geq c_*(m)$. For bistable and combustion RPMEs, however, there is only one traveling wave solution $\widetilde{U}(x-c_* t)$ (cf. \cite{GK}).
Recently, Du, Quir\'{o}s and Zhou \cite{DQZ} considered the ($N$-dimensional) RPME with Fisher-KPP type of reaction. In particular, they presented a precise estimate for the spreading speed of the free boundary with a logarithmic correction. On the other hand, the qualitative properties for the RPMEs with bistable or combustion reaction were not as clear as the monostable case.
In 1982, Aronson, Crandall and Peletier \cite{ACP} considered the bistable RPME in a bounded interval. Besides the well-posedness, they specified the stationary solutions which can be $\omega$-limits of the solutions. Recently, G\'{a}rriz \cite{Garriz}
considered the problem (CP) with reactions of the monostable, bistable and combustion types. He gave sufficient conditions ensuring the spreading or vanishing, and used the traveling wave solution to characterize the spreading solutions starting at compactly supported initial data.

Our aim in this paper is to give out a complete classification for the asymptotic behaviors of the solutions of (CP).
We will specify its stationary solutions; prove a general convergence result, that is, any nonnegative bounded global solution converges as $t\to \infty$ to a stationary one. Then, based on the general result, we prove the hair-trigger effect for monostable PME and spreading-transition-vanishing trichotomy results on the asymptotic behavior for solutions of bistable or combustion RPMEs. In some sense, this paper can be regarded as a RPME version of the important works  \cite{DL, DM, Zla} for RDEs.

Our basic assumption on $f$ is
\begin{equation*}\label{f-Lip}
\mbox{(F)}\hskip 5mm
  f(u)\in C^2([0,\infty)) \mbox{ with Lipschitz number } K,\ \
  f(0)= 0 \mbox{ and } f(u)<0 \mbox{ for } u>1.
 \hskip 25mm
\end{equation*}
The $C^2$ smoothness is mainly used to give the a priori estimates, including the lower bound for $v_{xx}$ (where $v$ is the pressure, see Appendix for details). In the study of asymptotic behavior, we actually only need $f\in C^1$. The assumption $f(u)<0$ for $u>1$ is natural in population dynamics, which means that the population has a finite capacity. If one only studies the well-posedness, this requirement can be weakened or omitted.  The initial data are chosen from the following set (see Figure 1)
\begin{equation*}\label{Initial-data}
{\rm (I)} \hskip 5mm u_0\in \mathfrak{X} := \left\{ \psi \in C(\R)
\left|
\begin{array}{l}

 \mbox{there exist } b>0 \mbox{ and a finite number of open intervals in } [-b,b]
 \\
 \mbox{such that } \psi(x)>0 \mbox{ in these intervals, and } \psi(x)=0 \mbox{ otherwise}
 \end{array}
 \right.
 \right\}.\ \ \ \
\end{equation*}
It is also possible to consider more general initial data, for example,  $u_0\in C(\R)$ with compact support, or, $u_0\in L^1(\R)\cap L^\infty(\R)$. For simplicity and clarity of presentation, especially, for clarity of the free boundaries, we consider in this paper only the initial data in $\mathfrak{X}$. In this case, the solution has a left-most free boundary $l(t)$, a right-most free boundary $r(t)$, and finite number of interior free boundaries in $(l(t),r(t))$.
\begin{figure}
\begin{center}
\includegraphics[width=4in,height=1in]{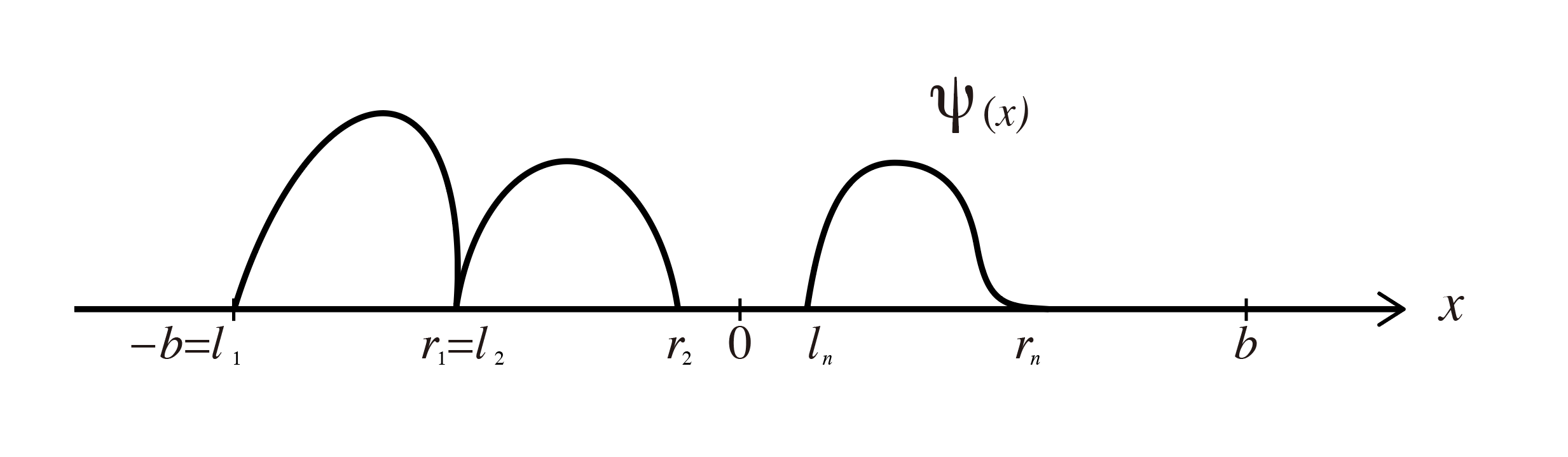}

\caption{An example of the initial data.}
\end{center}
\end{figure}

For any $T>0$, denote $Q_T := \R \times (0,T)$. A function $u(x,t)\in C(Q_T)\cap L^\infty (Q_T)$ is called a {\it very weak solution} of
(CP) if for any $\varphi \in C^\infty_c (Q_T)$, there holds
\begin{equation}\label{def-weak sol}
\int_{\R} u(x,T)\varphi(x,T) dx  =  \int_{\R} u_0(x) \varphi(x,0)dx +
\iint_{Q_T} f(u) \varphi  dx dt  + \iint_{Q_T} [u\varphi_t + u^m \varphi_{xx}] dx dt.
\end{equation}
As an extension of the definition, if $u$ satisfies \eqref{def-weak sol} with inequality \lq\lq $\geq$" (resp. \lq\lq $\leq$"), instead of equality, for every test function $\varphi\geq 0$, then $u$ is called a
{\it very weak supersolution/very weak upper solution} (resp. {\it very weak subsolution/very weak lower solution}) of (CP) (cf. \cite[Chapter 5]{Vaz-book}).

From the references \cite{ACP, PV, Sacks, Vaz-book} etc. we know that under the assumption (F) and (I), the problem (CP) has a unique very weak solution $u(x,t)\in C(Q_T)\cap L^\infty (Q_T)$ for any $T>0$, $u(x,t)\geq 0$ in $Q_T$,
and the support of $u(\cdot,t)$, denoted by ${\rm spt}[u(\cdot,t)]$, is contained in $[l (t),r(t)]$ for all $t>0$. Moreover, we will show below that both $-l(t)$ and $r(t)$ are non-decreasing Lipschitz functions, as in PME, and so the following limits exist:
\begin{equation}\label{def-l-r-infty}
l^\infty := \lim\limits_{t\to \infty} l(t),\quad
r^\infty := \lim\limits_{t\to \infty} r(t).
\end{equation}

With the global existence in hand, it is possible to study the asymptotic behavior for the solutions. For RDEs, the equations are uniform parabolic and the strong maximum principle is applicable. So, any solution is a classical one and the convergence of $u$ to its $\omega$-limit is taken in the topology of $C^{2,1}_{loc}(\R)$. For our RPME, however, we have only very weak solutions with $C^\alpha$ bounds in any compact domain. A nonnegative solution is not necessarily to be positive and classical. Hence the convergence of a solution to its limit is first considered in the topology of $C_{loc}(\R)$ or $L^\infty_{loc}(\R)$. In addition, if $u(x,t_n)\to w(x)>0$ as $t_n\to \infty$ in some domain $J$, then $u>0$ in $J$ by the positivity persistence and so it is a classical solution. Thus, the convergence $u(x,t_n)\to w(x)$ holds in $C^{2}$ topology in any compact subdomain of $J$. In summary, if no other specification, throughout this paper we use the following definition for $\omega$-limits of $u$:
\begin{equation}\label{def-conv}
\lim\limits_{n\to \infty} u(x,t_n)= w(x) \mbox{ means }
\left\{
 \begin{array}{l}
 u(x,t_n)\to w(x) \mbox{ in } C_{loc}(\R) \mbox{ topology};\\
 u(x,t_n)\to w(x) \mbox{ in } C^{2}_{loc}(J) \mbox{ topology if } w(x)>0 \mbox{ in }J.
 \end{array}
 \right.
\end{equation}
In some cases, our problem may have a {\it ground state solution}, which means one of the following two types of nonnegative stationary solutions of (CP):

\noindent
$\bullet$ {\bf Type I ground state solution} (see Figure 2): $U_0 (x)\in C^2(\R)$ is an even stationary solution of (CP) with
$$
U_0 (x)>0>U'_0 (x) \mbox{ for }x>0,\quad U_0 (+\infty)\in [0,1);
$$

\begin{figure}
\centering
\includegraphics[width=2.6in,height=1in]{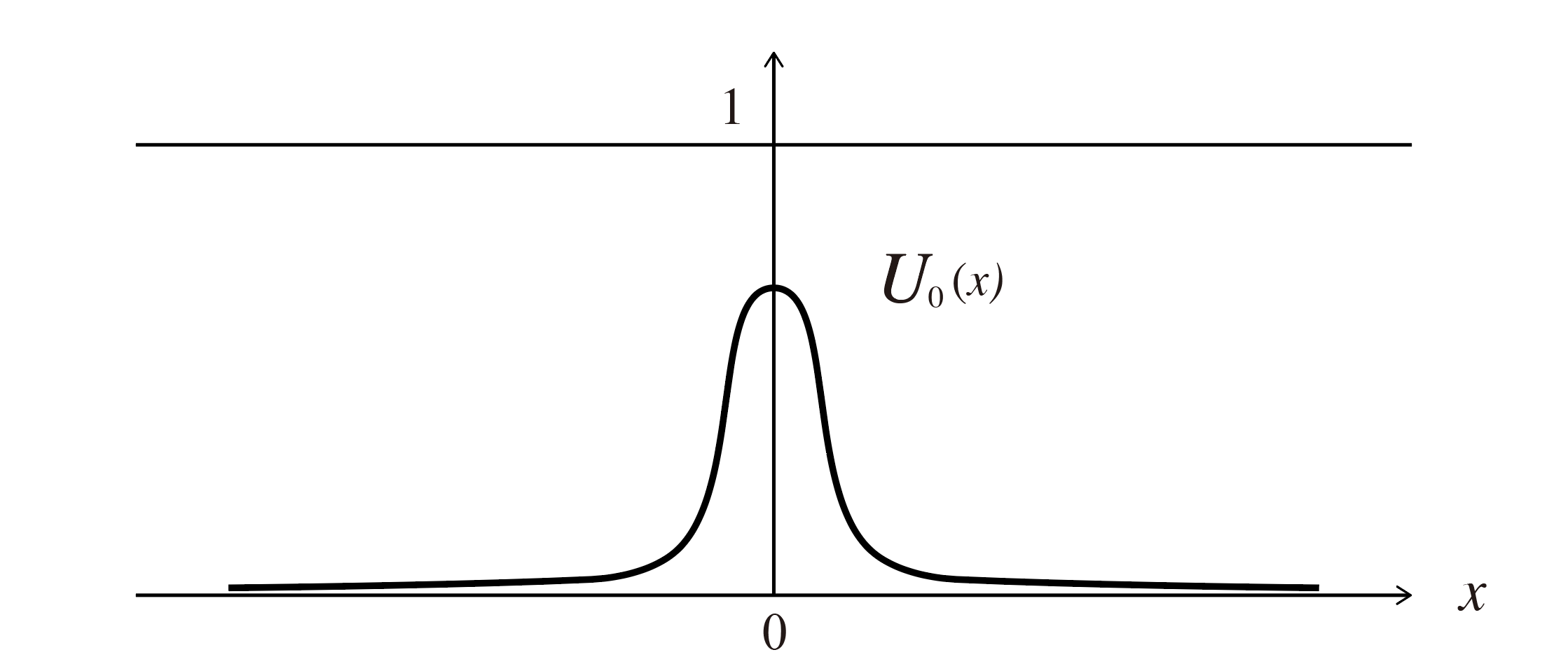}
\hspace{0.1in}
\includegraphics[width=2.6in,height=1.05in]{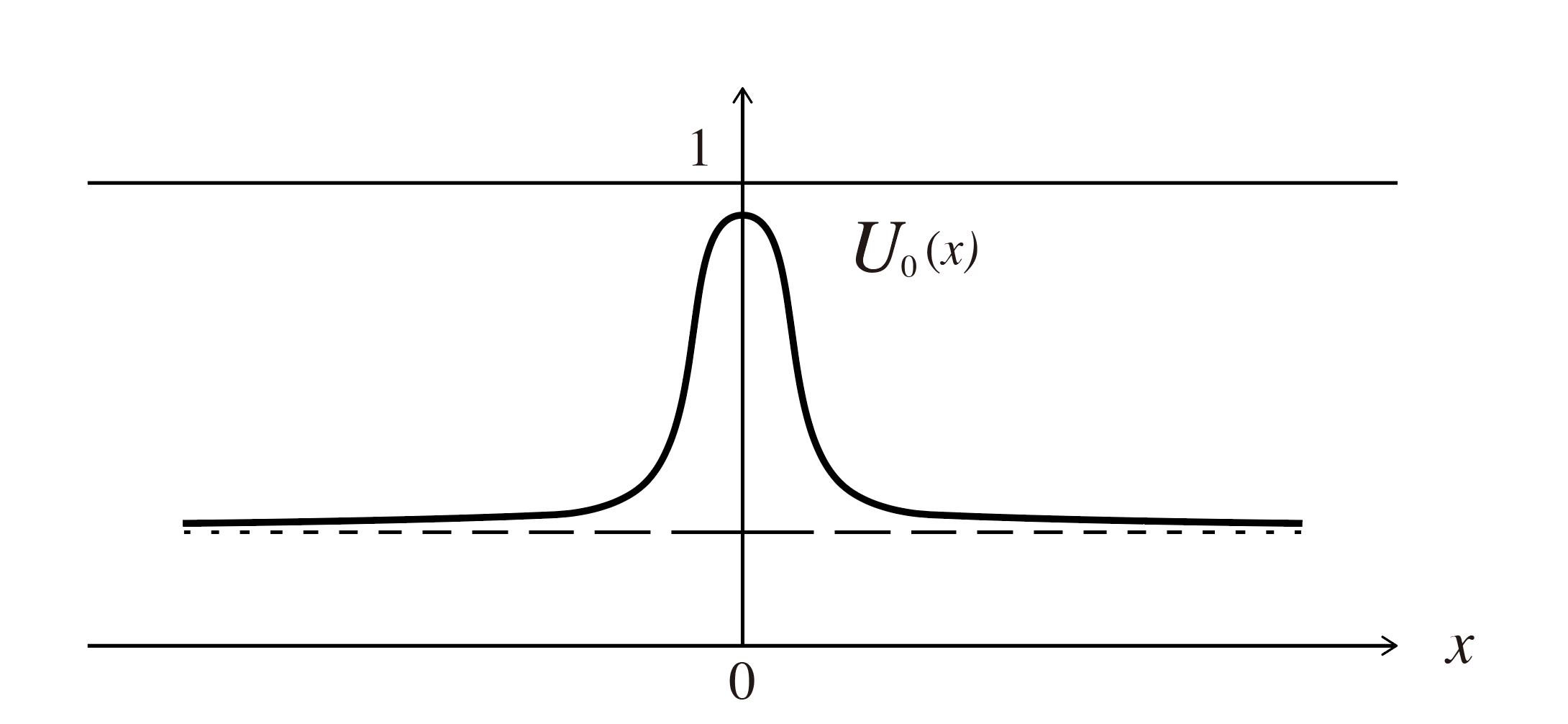}
\caption{Type I ground state solution.}
\end{figure}

\noindent
$\bullet$ {\bf Type II ground state solution} (see Figure 3): there exist a positive integer $k$, $z_i\in \R \ (i=1,2,\cdots,k)$ and $L>0$ with
$z_i+2 L \leq z_{i+1}\ (i=1,2,\cdots, k-1)$ such that
\begin{equation}\label{finite-ground-state-0}
\mathcal{U}(x)= U(x-z_1)+U(x-z_2)+\cdots +U (x-z_k), \quad x\in \R,
\end{equation}
where $U(x)\in C (\R)\cap C^2(\R\backslash\{\pm L\})$ is an even stationary solution of the problem (CP) with
\begin{equation}\label{def-type-II}
U (x)>0>U' (x) \mbox{ for } x\in (0,L),\quad U (x)=0 \mbox{ for }x\geq L,\quad (U^{m-1})'(L)=0.
\end{equation}
\begin{figure}
\begin{center}
\includegraphics[width=3.2in,height=1in]{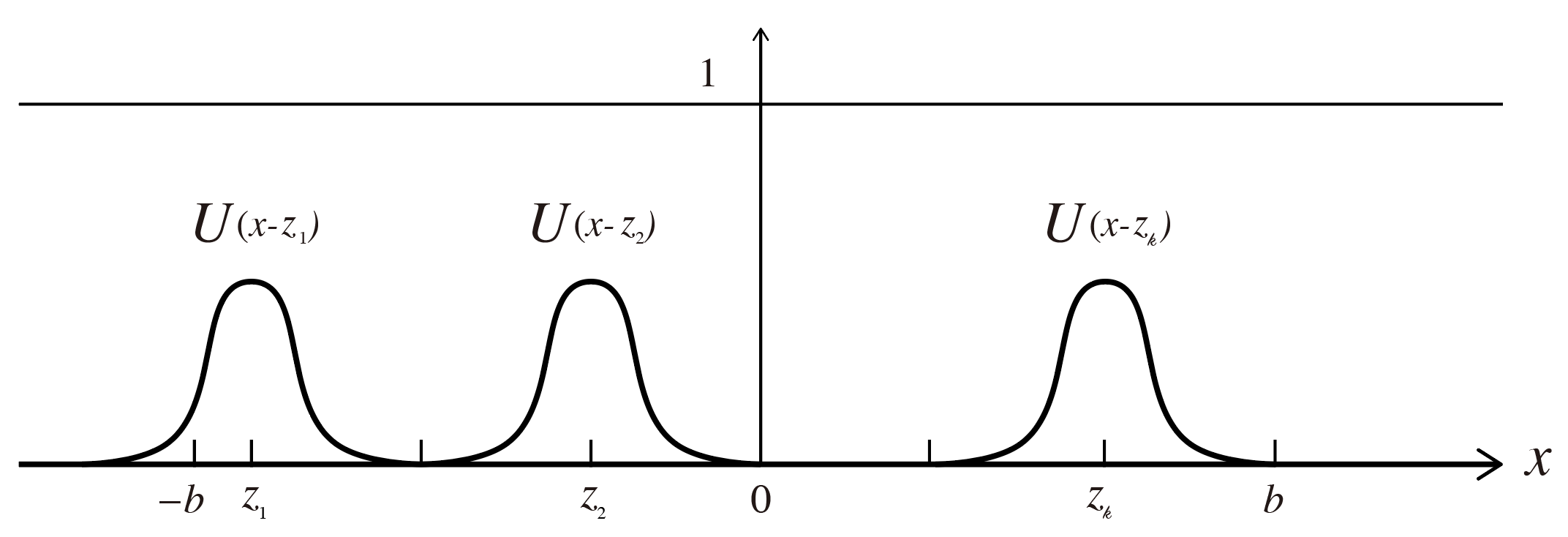}
\caption{Type II ground state solution.}
\end{center}
\end{figure}
\noindent
A Type I ground state solution is positive on the whole $\R$, which is the analogue of that in bistable RDEs (cf. \cite{DM,Zla}). For our RPME, this kind of solution exists if, in particular, $f$ is a bistable nonlinearity with
\begin{equation}\label{decay-rate-f}
f(u)=-\lambda u^\alpha (1+o(1)) \mbox{\ \ as\ \ } u\to 0+0,
\end{equation}
for some $\lambda >0$ and $\alpha\geq m$ (see Lemma \ref{lem:finite-infinite-GS}). The component $U$ in Type II ground state solution has a compact support, which is different from Type I. Such a solution exists when \eqref{decay-rate-f} holds for $0<\alpha<m$ (see Lemma \ref{lem:finite-infinite-GS}).
The last equality in \eqref{def-type-II} is used to indicate that $U(x)$ is a stationary solution not only of the RPME but also of the Cauchy problem (CP). In fact, a necessary and sufficient condition ensuring a stationary solution $\bar{u}$ of RPME with compact support $[-b,b]$ is also a stationary solution of the Cauchy problem (CP) is that the waiting times at its boundaries $\pm b$ are infinite, which requires that $(\bar{u}^{m-1})(\pm b)=0$ (see Lemma \ref{lem:ss-problem} below).

Our first main theorem is a general convergence result, which says that any nonnegative bounded global solution of (CP) converges to a nonnegative zero of $f$ or a ground state solution:

\begin{thm}[General convergence theorem]\label{thm:general-conv}
Assume {\rm (F)} and {\rm (I)}. Let $u(x,t)$ be a bounded, nonnegative, time-global solution of {\rm (CP)}. Then $u(\cdot,t)$ converges as $t\to \infty$, in the sense of \eqref{def-conv}, to a stationary solution of {\rm (CP)}, which is one of the following types:
\begin{enumerate}[{\rm (i)}]
\item a nonnegative zero of $f$;
\item $U_0 (x-z_0)$ for some $z_0\in [-b,b]$ and some Type I ground state solution $U_0$;
\item $\mathcal{U}(x)$ as in \eqref{finite-ground-state-0} for $k\geq 1$ points $z_1,z_2,\cdots, z_k\in [-b,b]$. In addition, $l^\infty,\ r^\infty$ are bounded in this case.
\end{enumerate}
\end{thm}

This theorem is a fundamental result and can be used conveniently to study various RPME. For example, when the reaction term in RPME is a typical monostable, bistable or combustion one, it is easy to clarify all of the nonnegative stationary solutions. Hence, using the general convergence result in the previous theorem, we can give a rather complete analysis on the asymptotic behavior for the solutions of (CP) with these reactions.

To obtain the above convergence result, our main tool is the so-called {\it zero number argument} in the version of porous medium equation, which can be used to describe the intersection points between two solutions. The main difficulty to establish the argument is brought by the existence of free boundaries and the degeneracy of diffusion here, which is a significant difference compared with the classical heat equation. In particular, the intersection number may increase. For the convenience of the readers, we present the intersection number properties as a theorem in the below. Let $u_1 ,u_2$ be two solutions of {\rm (CP)} with initial data in $\mathfrak{X}$, then they are zero outside of their supports. It is a little confusing to define the intersection number between $u_1$ and $u_2$ in these places. Instead, we denote $\mathcal{Z}_0(t)$ as the number of their positive intersection points, that is,
\begin{equation}\label{def-Z0}
\mathcal{Z}_0(t):= \#\{x\in \R \mid u_1(x,t)=u_2(x,t)>0\},\qquad t\geq 0,
\end{equation}
where $\# S$ denotes the number of the elements in a set $S$.

\begin{thm}[Intersection number properties]\label{thm:intersection num}
Assume {\rm (F)}. Let $u_1 ,u_2$ be bounded, nonnegative, time-global solutions of {\rm (CP)} with initial data in $\mathfrak{X}$. If $\mathcal{Z}_0(0)<\infty$, then $\mathcal{Z}_0(t)<\infty$ for all $t>0$, and there exists a time $T\geq 0$ such that
$$
\mathcal{Z}_0(t) \mbox{ decreases for } t>T.
$$
More specifically, there exist  $\{t_j\}_{0\leq j\leq k}$ with $1\leq k \in \mathbb{N}$ and   $0=t_0<t_1<\cdots< T:= t_{k-1} < t_k = +\infty$ such that
\begin{enumerate}[{\rm (i)}]
\item $\mathcal{Z}_0(t)$ decreases if $t_j<t<t_{j+1}\ (j=0,1,\cdots,k-1)$;
\item $\mathcal{Z}_0(t)$ strictly increases at $t=t_j\ (j=1,\cdots, k-1)$ in the following sense:
$$
\lim_{t\rightarrow t_j + 0}\mathcal{Z}_0(t)>\lim_{t\rightarrow t_j -0}\mathcal{Z}_0(t).
$$
\end{enumerate}
\end{thm}

The last property (that is, $\mathcal{Z}_0(t)$ increases at $t_j$) happens in particular when a component where $u_1>0$ meets a component where $u_2>0$ as their boundaries approaching to each other (see more in Lemma \ref{lem:diminish-intersection 1}). As we have mentioned above, this is significant different from the classical zero number diminishing properties. Such a phenomena happens only in PMEs where a solution has free boundaries but not in RDEs.

In order to describe the threshold of transition solutions, as in \cite{DL, DM,Zla}, we introduce a one-parameter family of initial data $u_0 = \phi_\sigma (x)$ satisfying the following conditions:
\begin{itemize}
\item[($\Phi_1$)] $\phi_\sigma \in \mathfrak{X}$ for every $\sigma>0$, and the map $\sigma \mapsto \phi_\sigma$ is continuous from $(0,\infty)$ to $L^\infty(\R)$;
\item[($\Phi_2$)] if $0<\sigma_1 <\sigma_2$, then $\phi_{\sigma_1}\leq,\not\equiv \phi_{\sigma_2}$;
\item[($\Phi_3$)] $\lim\limits_{\sigma\to 0} \|\phi_{\sigma}\|_{L^\infty(\R)} =0$.
\end{itemize}

\begin{thm}\label{thm:mono}
Assume $f\in C^2$ is a monostable reaction term satisfying \eqref{mono}. Let $u_\sigma(x,t)$ be the very weak solution of {\rm (CP)} with $u_0=\phi_\sigma (x)$ satisfying $(\Phi_1)$-$(\Phi_3)$.
Then the hair-trigger effect holds: for all $\sigma>0$, we have
$$
    u_\sigma(\cdot, t) \to 1 \mbox{\ \ as\ \ } t\to \infty,\ \ \mbox{ in the topology of } L^{\infty}_{loc}(\R).
$$
\end{thm}

\begin{rem}\label{rem:hair-Garriz}\rm
Theorem \ref{thm:mono} is a direct application of Theorem \ref{thm:general-conv}, since the structure of the $\omega$-limit set is rather simple in this case (see details in Section 4). Through a different proof, \cite{Garriz} obtained a deeper result on the hair-trigger effect for more general monostable RPME in $N$-dimensional space, where $f$ satisfies $\liminf\limits_{u\rightarrow 0+0}\frac{f(u)}{u^{m+2/N}}>0$.
\end{rem}

In the combustion case, we have the following result.

\begin{thm}\label{thm:com}
Assume $f\in C^2$ is a combustion reaction term satisfying \eqref{combus}. Let $u_\sigma (x,t)$ be the very weak solution of {\rm (CP)} with $u_0 = \phi_\sigma$ satisfying $(\Phi_1)$-$(\Phi_3)$.
Then there exist $0<\sigma_* \leq \infty$ such that the following trichotomy result holds:
\begin{itemize}
\item[(i)] {\bf spreading happens} for $\sigma >\sigma_*$:
    $$
    u_\sigma(\cdot, t) \to 1 \mbox{\ \ as\ \ } t\to \infty,\ \ \mbox{ in the topology of } L^{\infty}_{loc}(\R).
    $$
\item[(ii)] {\bf vanishing happens} for $0< \sigma < \sigma_*$:
    $$
    u_\sigma(\cdot, t) \to 0 \mbox{\ \ as\ \ } t\to \infty,\ \ \mbox{ in the topology of } L^{\infty}(\R).
    $$

\item[(iii)] {\bf transition case} for $\sigma =\sigma_*$:
$$
    u_\sigma(\cdot, t) \to \theta \mbox{\ \ as\ \ } t\to \infty,\ \ \mbox{ in the topology of } L^{\infty}_{loc}(\R).
    $$
    Moreover for large $t$, $u_\sigma $ has exactly two free boundaries $l(t)<r(t)$ satisfying
\begin{equation}\label{transition-asy-speed}
    -l(t), r(t) = 2 y_0 \sqrt{t} \; [1+o(1)]\mbox{\ \ as\ \ }t\to \infty,
\end{equation}
for some $y_0 \in (0, \theta^{\frac{m-1}{2}})$.
\end{itemize}
\end{thm}

In the bistable case, if we further assume  that $f'(0)<0$, then the problem (CP) has Type II ground state solutions, and so we have the following trichotomy result.

\begin{thm}\label{thm:bi}
Assume $f\in C^2$ is a bistable reaction term satisfying \eqref{bi} and $f'(0)<0$. Let $u_\sigma (x,t)$ be the very weak solution of {\rm (CP)} with $u_0 = \phi_\sigma$ satisfying $(\Phi_1)$-$(\Phi_3)$.
Then there exist $0<\sigma_* \leq \sigma^* \leq \infty$ such that the following trichotomy holds:
\begin{itemize}
\item[(i)] {\bf spreading happens} for $\sigma >\sigma^*$:
    $$
    u_\sigma(\cdot, t) \to 1 \mbox{\ \ as\ \ } t\to \infty,\ \ \mbox{ in the topology of } L^{\infty}_{loc}(\R).
    $$
\item[(ii)] {\bf vanishing happens} for $0< \sigma < \sigma_*$:
    $$
    u_\sigma(\cdot, t) \to 0 \mbox{\ \ as\ \ } t\to \infty,\ \ \mbox{ in the topology of } L^{\infty}(\R).
    $$

\item[(iii)] {\bf transition case} for $\sigma_* \leq\sigma\leq \sigma^* $:
$$
    u_\sigma(\cdot, t) \to \mathcal{U}(x) \mbox{\ \ as\ \ } t\to \infty,\ \ \mbox{ in the topology of } L^{\infty}(\R),
    $$
where $\mathcal{U}(x)$ is a Type II ground state solution defined as \eqref{finite-ground-state-0} for $k\; (\geq 1)$ points $z_1,z_2,\cdots, z_k\in [-b,b]$.
\end{itemize}
\end{thm}

\begin{rem}\rm

\begin{itemize}
\item[]
\item[(1)] Under some additional conditions, we can show the sharpness of the transition solution: $\sigma_*=\sigma^*$ (see details in Lemma \ref{lem:sharp-bi}).
\item[(2)] The behaviour of a transition solution in Theorem \ref{thm:bi} (iii) is quite different from that of the bistable RDE, where the transition solution converges to a Type I ground state solution which is positive on the whole $\R$.
\item[(3)] When spreading happens, the estimate on the propagation of the free boundaries has been shown in \cite{Garriz}.
\end{itemize}

\end{rem}

The main features of the paper are as follows:
\begin{itemize}
\item[(1)] we give out a complete classification for the asymptotic behaviors of the solutions of (CP) with general nonlinear term;
\item[(2)] we establish the zero number argument for RPME with solid proof.
\end{itemize}

This paper is arranged as follows. In Section 2 we first present some results on the well-posedness and the a priori estimates, which are not entirely new and some details  are postponed to Appendix. We also present sufficient conditions ensuring the waiting time of a free boundary to be $0$, finite or infinite, specify the relationship between the stationary solutions of the RPME and those of the Cauchy problem, and derive the monotonicity and the Darcy law for the free boundaries. In Section 3, we study the monotonicity of the solution outside of the initial support, show the intersection number properties, and then use it, as well as the classical zero number argument, to prove the general convergence result. In Section 4, we consider the RPME with monostable reaction term and prove the hair-trigger effect. In Section 5 we consider the RPME with combustion reaction term: we present all possible nonnegative stationary solutions, prove a trichotomy result with a sharp transition for the asymptotic behavior. We also present a precise estimate for the propagation speed of the free boundaries of the transition solutions. In Section 6, we consider the RPME with bistable reaction term: we present all possible nonnegative stationary solutions, specify Type I and Type II ground state solutions, prove a trichotomy result with transition for the asymptotic behavior. In Section 7, we give a sufficient condition for the complete vanishing phenomena. Finally, for the convenience of the readers, we present in Appendix (Section 8) the a priori estimates, well-posedness and other important properties for PME with general reaction terms. They are not entirely new but not well collected in the literature.

\medskip

We collect some notations and symbols which will be used in this paper.
\begin{itemize}
\item For $\psi\in C(\R)$, we use ${\rm spt}[\psi(\cdot)]$ to denote its support;
\item $v(x)\to b\ (x\to x_0 - 0)$, or, $\lim\limits_{x\to x_0-0} v(x)=b$ means the left limit  $\lim\limits_{x<x_0,\ x\to x_0} v(x)=b$. The right one is denoted similarly;
\item $D^-_x v(x_0) = \lim\limits_{x\to x_0-0} \frac{v(x)-v(x_0)}{x-x_0}$ denotes the left derivative of $v$ at $x_0$. The right one $D^+_x v(x_0)$ is similar;
\item $0<1-x \ll 1$ means that $x<1$ and $1-x$ is sufficiently small;
\item $a(x)\sim b(x)$ as $x\to x_0 \pm 0$ means that $a(x) = b(x) (1+o(1))$ as $x\to x_0\pm 0$.
\end{itemize}

\section{Preliminaries}\label{sec:pre}

In this section we first collect some basic results on the a priori estimates and the well-posedness. Then we specify the waiting time, stationary solutions, the regularity and the Darcy law of the free boundaries, which turn out not direct consequences of the corresponding results for PME.

\subsection{Well-posedness and a priori estimates}

As we have mentioned above, the existence and uniqueness of (CP) have been studied a lot by many authors (see, for example,  \cite{ACP, PV, Sacks, Vaz-book} etc.). In addition, by  the assumption $f(u)<0$ for $u>1$ in (F) and by the comparison principle, it is easily to obtain
\begin{equation}\label{lower-upper-bounds}
0\leq u(x,t)\leq M'_0 :=\max\{1, \|u_0\|_{L^\infty(\R)}\},\quad x\in \R,\ t>0.
\end{equation}

The persistence of positivity, the finite spreading speed and the a priori estimates were also studied widely, especially for the PME without reactions/sources or RPMEs with polynomial or logistic reactions. For our equation with a general reaction term, however, these properties seem not conveniently collected in literature, though they are not entirely new. We collect some of them here and present some detailed proof in the appendix.

1. {\it Positivity persistence}. If $u(x,t_0)\geq \varepsilon>0$ in a neighborhood of $x_0\in \R$, then $u(x_0, t)>0$ for all $t\geq t_0$ (see Proposition \ref{prop:positive}). Thus, $u$ is classical near the line $\{(x_0,t)\mid t >t_0\})$.

2. {\it Finite propagation speed}. The solution has
a left-most free boundary $l(t)$ and a right-most free boundary $r(t)$ (also called interfaces) which propagate in finite speed: ${\rm spt} [u(\cdot,t)] \subset [l(t),r(t)]\subset [-\bar{s}(t),\bar{s}(t)]$ with
\begin{equation}\label{def-bar-s}
\bar{s}(t) := O(1) (t+1)^{1/2} e^{\frac{K(m-1)(t+1)}{2}}, \quad t>0,
\end{equation}
(see Proposition \ref{prop:finite speed}).

In a gas flow problem through a porous medium, $u$ denotes the density of the gas, and
\begin{equation}\label{def-pressure}
v(x,t) := \frac{m}{m-1} [u(x,t)]^{m-1}
\end{equation}
represents the {\it pressure} of the gas. In many cases, it is convenient to consider $v$ rather than $u$. Using $v$, (CP) is converted into the following problem
$$
{\rm (pCP)}\hskip 30mm
 \left\{
 \begin{array}{ll}
 v_t = (m-1) v v_{xx} + v_x^2 + g(v), & x\in \R,\ t>0, \\
 \displaystyle  v(x,0) = v_0(x) := \frac{m}{m-1} u_0^{m-1}(x), & x\in \R,
 \end{array}
 \right.
 \hskip 50mm
$$
with
\begin{equation}\label{def-g}
g(v) := m \Big( \frac{(m-1)v}{m}\Big)^{\frac{m-2}{m-1}} f
\left( \Big( \frac{(m-1)v}{m}\Big)^{\frac{1}{m-1}} \right).
\end{equation}
For simplicity, we call the equation in (pCP) as {\it pRPME}. Since $f$ is Lipschitz with number $K$ we have
\begin{equation}\label{bound-g-Lip}
|g(v)|\leq K (m-1) v, \quad v\geq 0.
\end{equation}
Since $u$ is bounded as in \eqref{lower-upper-bounds}, we see that $v$ is also bounded:
\begin{equation}\label{bound-v}
0\leq v(x,t) \leq M_0 := \max\left\{ \frac{m}{m-1}, \|v_0\|_{L^\infty} \right\},\quad x\in \R,\ t>0.
\end{equation}

3. {\it A priori estimates}.  We have a priori estimates for $v_x, v_{xx}$ and $v_t$ as follows.
If $v>0$ in $R:= (a,b)\times (0,T]$, then for any $0< \delta <\frac{b-a}{2}, \ \tau < T$ there holds:
\begin{equation}\label{C1 bound}
|v_x(x,t)|\leq M_1 (m, f, \delta, \tau),\quad (x,t)\in [a+\delta, b-\delta]\times [\tau, T].
\end{equation}

\noindent
There exist $\tau_0 =\tau_0(m, v_0)$ and $C=C(m, v_0)$ such that
$$
v_{xx} (x,t)\geq - C(t+\tau_0),\quad x\in \R,
\ t\in (0,T],
$$
in the sense of distributions in $\R\times (0,T]$. For any $\tau>0$ there holds
\begin{equation}\label{bound-of-vt}
-C_3 t - C_4 \leq v_t(x,t) \leq C_1 t +C_2,\quad x\in \R,\ t\geq \tau,
\end{equation}
where $C_i$ depends on $m, f$ and $v_0$, while $C_2$ depends also on $v_t(x,\tau)$.

In addition, we have the following locally uniform H\"{o}lder estimates (cf. \cite{DB-F} and \cite[Theorem 7.18]{Vaz-book}): for any $\tau >0$ and any compact domain $D\subset \R\times (\tau, \infty)$, there are positive constants $C$ and $\alpha$, both depending only on $M'_0, m$ and $\tau$, such that
\begin{equation}\label{Holder-est}
\|v\|_{C^\alpha(D)}\leq C.
\end{equation}
Note that the H\"{o}lder estimates in \cite{DB-F,Vaz-book} are given for $u$, but it is easy to convert to that for $v$.

\subsection{Waiting times of the free boundaries and compactly supported stationary solutions}

In PME and RPMEs, the degeneracy of the equation not only causes finite propagation and free boundaries. It may even occur that a free boundary is stationary for a while if the mass distribution near the border of the support is very small.

Let $x_0\in \R$ be a zero of $u_0(x)$. Define the {\it waiting time at $x_0$} by
$$
t^*(x_0) := \inf \{t> 0 \mid u(x_0, t)>0\}.
$$
For the one dimensional PME, it is known that $t^*(x_0)$ can be $0$ or a positive number, depending on the decay rate of $u_0$ near $x_0$ (cf. \cite[\S 15.3]{Vaz-book}, \cite{Wu-book}). For our equation with a source, the situation is more complicated since positive reactions can promote the propagation while the negative reactions will restrain it. We will show the following results.
\begin{itemize}
\item[(1).] $t^*(x_0)=0$ when the pressure $v_0$ near $x_0$ is bounded from below by a linear function;
\item[(2).] $t^*(x_0)>0$ when the pressure $v_0$ lies below a quadratic function near $x_0$. In addition $t^*(x_0)<\infty$ if the reaction term $f\geq 0$ near $0$, or, if $f$ is negative but $v_0$ lies above another quadratic function with big coefficient.
\item[(3).] $t^*(x_0)=\infty$ when the pressure $v_0$ is less than a quadratic function near $x_0$ and $f(u)\leq - K'u$ for some large $K'>0$ and $0\leq u\ll 1$.
\end{itemize}
For simplicity, we assume $0$ is a zero of $v_0$ and consider its waiting time.
\medskip

1. {\it The case $t^*(0)=0$}.

\begin{lem}\label{lem:no-waiting-time}
Assume {\rm (F)} and {\rm (I)}. Assume $v_0(0)=0$ and
\begin{equation}\label{v0=line-in-part}
v_0 (x) \geq  \rho x,\quad 0\leq x\leq r_0,
\end{equation}
for some $\rho>0,\ r_0>0$, then $t^*(0)=0$.
\end{lem}

\begin{proof}

Set
$$
a:= \frac{\pi}{r_0},\quad h:= \frac{\rho r_0}{\pi},\quad \mu := K(m-1) + m h a^2,\quad b(t):= \frac{ha^2}{\mu} (e^{-\mu t}-1),
$$
and, for $t\geq 0$, define
$$
\underline{v}(x,t):=  \left\{
 \begin{array}{ll}
 h e^{-\mu t}\sin (ax -b(t)), & \ \ \ \displaystyle  \frac{b(t)}{a}\leq x \leq \frac{b(t)+\pi}{a},\\
 0, & \ \ \ \mbox{ otherwise}.
 \end{array}
 \right.
$$
By a direct calculation we have
$$
\underline{v}(x,0)\leq v_0(x),\quad x\in \R,
$$
and, with $z:= ax -b(t)$,
\begin{eqnarray*}
\mathcal{N}(\underline{v}) & := & \underline{v}_t - (m-1) \underline{v} \; \underline{v}_{xx} - \underline{v}_x^2 - g(\underline{v})\\
& \leq & \underline{v}_t - (m-1) \underline{v} \; \underline{v}_{xx} - \underline{v}_x^2 +K(m-1) \underline{v}\\
& = & he^{-\mu t} [ -\mu \sin z + (m-1)a^2 h e^{-\mu t} \sin^2 z + K(m-1)\sin z \\
& & - b'(t) \cos z - a^2 h e^{-\mu t} \cos^2 z ]\\
& \leq & a^2 h^2 e^{-\mu t} [- \sin z + e^{-\mu t} \cos z - e^{-\mu t} \cos^2 z]\\
& \leq & a^2 h^2 e^{-2\mu t} [- \sin z + \cos z - \cos^2 z]\\
& = & -\frac12 a^2 h^2 e^{-2\mu t} [(1-\cos z)^2 + 2\sin z -\sin^2 z]\\
& \leq & 0,\qquad  \frac{b(t)}{a}\leq x \leq \frac{b(t)+\pi}{a}, \ t\geq 0.
\end{eqnarray*}
Therefore, $\underline{v}(x,t)$ is a subsolution, and so by the comparison result (see, for example, \cite[Theorem 6.5]{Vaz-book}) we have $v(0,t)\geq \underline{v}(0,t)>0$ for small $t>0$. Then  $t^*(0)=0$ since $b'(t)<0$ for all $t\geq 0$.
\end{proof}

2. {\it The case $t^*(0)\in (0,\infty)$}.

\begin{lem}\label{lem:finite-waiting-time}
Assume {\rm (F)} and {\rm (I)}. Assume $v_0(0)=0$,
\begin{equation}\label{v0-large}
v_0(x)\geq A_2 x^2,\quad 0\leq x\leq r_0,
\end{equation}
for some $r_0>0$ and some $A_2$ satisfying
\begin{equation}\label{choice of A}
A_2 > A_2 (m,K):=\frac{K(m-1)}{2\sigma^{\frac{m}{m+1}} - 2\sigma},\quad \sigma:= \Big(\frac{m}{m+1} \Big)^{m+1}.
\end{equation}
and
\begin{equation}\label{v0-small}
v_0 (x) \leq A_1 x^2, \quad x\geq 0,
\end{equation}
for some $A_1>A_2$. Then $t^*(0)\in (0,\infty)$.
\end{lem}

\begin{proof}
First we show that $t^*(0)>0$. Set
$$
a:= \frac{(m-1)K}{2(m+1)A_1 +(m-1)K},\quad T_1 := \frac{a}{(m-1)K},\quad C_1 := \frac{1-a}{2(m+1)},
$$
and define
$$
\bar{v}(x,t) := \frac{C_1 x^2}{T_1-t},\quad x\in \R,\ 0\leq t <T_1.
$$
Then $\bar{v}(x,0)= A_1 x^2 \geq v_0(x)$, and
\begin{eqnarray*}
\mathcal{N}(\bar{v})& := & \bar{v}_t - (m-1)\bar{v} \bar{v}_{xx} - \bar{v}_x^2 - g(\bar{v}) \\
& \geq &
\bar{v}_t - (m-1)\bar{v} \bar{v}_{xx} - \bar{v}_x^2 - K(m-1) \bar{v}\\
& = & \frac{C_1 x^2}{(T_1-t)^2} [1-2(m+1)C_1 - K(m-1)(T_1-t)]\\
& \geq & 0,\quad x\in \R,\ t\in [0,T_1).
\end{eqnarray*}
By comparison we have $v(x,t)\leq \bar{v}(x,t)$ in $t\in [0,T_1)$. Hence, the waiting time of $v$ at $0$ is not smaller than $T_1=T_1 (m,A_1,K)$.

Next we prove $t^*(0)<\infty$. Set, for some $\delta \in (0,1)$,
$$
\tau:= \frac{\sigma}{K(m^2 -1)},\quad C_2 := \frac{\delta r_0^2 } {2(m+1)\tau^{\frac{1}{2(m+1)}}}, \quad
x_1 := \left[ \frac{C_2 (1+2(m+1)A_2\tau )}{A_2 \tau^{\frac{m}{m+1}}} \right]^{1/2}.
$$
Then by the choice of $A_2$ we have
$$
\frac{x_1^2}{C_2} = \frac{1}{A_2 \tau^{\frac{m}{m+1}}} + \frac{2(m+1)\tau}{\tau^{\frac{m}{m+1}}}\\
 < 2 \frac{\sigma^{\frac{m}{m+1}} -\sigma} {K(m-1)\tau^{\frac{m}{m+1}}} + 2\frac{(m+1)\tau }{\tau^{\frac{m}{m+1}}}\\
 =  2\frac{[K(m^2-1)]^{\frac{m}{m+1}}}{K(m-1)}.
$$

Define
$$
\underline{v}(x,t):=
  \left\{
  \begin{array}{ll}
  \displaystyle C_2 ( t + \tau )^{-\frac{m}{m+1}} - \frac{(x-x_1)^2}{2(m+1) ( t +\tau ) }, & x\in [\underline{l}(t), \underline{r}(t)],\ 0\leq t<T_2,\\
  0, & x< \underline{l}(t) \mbox{ or }x> \underline{r} (t),\ 0\leq t<T_2,
  \end{array}
  \right.
$$
where
$$
\underline{l} (t) := x_1 - \sqrt{2(m+1)C_2} (t+\tau)^{\frac{1}{2(m+1)}},\quad
\underline{r} (t) := x_1 + \sqrt{2(m+1)C_2} (t+\tau)^{\frac{1}{2(m+1)}}
$$ and
$$
T_2 := \frac{1-\sigma}{K(m^2 -1)}.
$$

We now show that $\underline{v}(x,t)$ is a subsolution  in $0\leq t<T_2$. In fact, by our construction it is easy to verify that
$$
A_2 x^2 \geq \underline{v}(x,0),\quad x\in [\underline{l} (0), \underline{r}(0)]\subset [0, r_0]
$$
provided $\delta\in (0,1)$ is small, and the equality holds at
$$
x= \hat{x}:= \frac{x_1} {2(m+1)\tau A_2 + 1}.
$$
In addition, for $x\in [\underline{l}(t), \underline{r}(t)],\ 0\leq t\leq T_2$, by a direct calculation we have
\begin{eqnarray*}
\mathcal{N}\underline{v} & := & \underline{v}_t - (m-1)\underline{v} \underline{v}_{xx} - \underline{v}_x^2 - g(\underline{v}) \\
& \leq & \underline{v}_t - (m-1)\underline{v} \underline{v}_{xx} - \underline{v}_x^2 + K(m-1)\underline{v} \\
& \leq & \frac{C_2}{(t+\tau)^{\frac{2m+1}{m+1}}}\Big[ \frac{-1}{m+1} +K(m-1)(t+\tau) \Big]  \leq  0.
\end{eqnarray*}
This implies that $\underline{v}$ is a subsolution  at leat in the time interval $[0,T_2]$.

By the choice of the above parameters, we have
$$
\frac{\underline{l}(0)}{\sqrt{2C_2}}  =  \left( \frac{1+2(m+1)A_2 \tau}{A_2 \tau^{\frac{m}{m+1}}}\right)^{1/2} - \left[ 2(m+1)\tau^{\frac{1}{m+1}}\right]^{1/2} >0,
$$
and
\begin{eqnarray*}
\frac{\underline{l}(T_2)}{\sqrt{2C_2}} & < & \left(\frac{[K(m^2-1)]^{\frac{m}{m+1}}}{K(m-1)}\right)^{1/2} - \left[
(m+1) (T_2 +\tau)^{\frac{1}{m+1}} \right]^{1/2} \\
& = & \left(\frac{[K(m^2-1)]^{\frac{m}{m+1}}}{K(m-1)}\right)^{1/2} - \left(
\frac{m+1}{[K(m^2 -1)]^{\frac{1}{m+1}}} \right)^{1/2} =0.
\end{eqnarray*}
This means that before $t=T_2$, the left free boundary $\underline{l}(t)$ of $\underline{v}(x,t)$ moves leftward from a positive point to a negative one. Then, $\underline{v}(0, t)$ becomes positive at $T_2$, and so does $v(0,t)$. This proves our lemma.
\end{proof}

The case $t^*(0)>0$ still includes two subcases: $t^*(0)\in (0,\infty)$  and $t^*(0) =\infty$. In the PME, it is known that only the former subcase is possible, which is called the {\it hole-filling} phenomena (see, for example, \cite[Theorem 15.15]{Vaz-book}). For our RPME with a reaction term, however, the latter is also possible.

\medskip

3. {\it The case where $t^*(0)=\infty$ and compactly supported stationary solutions}. To discuss the possibility of infinite waiting time, we first specify the compactly supported stationary solutions, including its sufficient and necessary conditions, examples and uniqueness.  We remark that the following two concepts should be distinguished:
\begin{center}
{\it compactly supported stationary solutions of the equation pRPME,\ \ and}\\
{\it compactly supported stationary solutions of the Cauchy problem {\rm (pCP)}}.
\end{center}
The latter means all compactly supported functions $\bar{v}$ such that the solution $v(x,t;\bar{v})$ of (pCP) with initial data $v_0 =\bar{v}$ satisfies $v(x,t;\bar{v})\equiv \bar{v}$ for all $x\in \R$ and $t\geq 0$.

\begin{lem}\label{lem:ss-problem}

Assume $\bar{v}$ is a continuous nonnegative stationary solution of pRPME with support $[l,r]$ and $\bar{v}(x)>0$ in $(l,r)$. Then
\begin{enumerate}[{\rm (i).}]
\item $\bar{v}$ is also a stationary solution of {\rm (pCP)} if and only if $t^*(l),\ t^*(r)\geq \tau$ for some $\tau>0$;
\item a sufficient condition for {\rm (i)} is $\bar{v}(x)\leq A (x-l)^2$ and $\bar{v}(x)\leq A (x-r)^2$ in $(l,r)$ for some $A>0$; a necessary condition for {\rm (i)} is $\bar{v}'(l)=\bar{v}'(r)=0$;
\item $\bar{v}$ is not a stationary solution of {\rm (pCP)} if $\bar{v}'(l)>0$ or $\bar{v}'(r)<0$;
\item in any case, $\bar{v}$ is a time-independent (also called a stationary) very weak subsolution of {\rm (pCP)}, and $v(x,t;\bar{v})\geq \bar{v}(x)$ for $x\in \R,\ t>0$. The inequality is strict on $[l,r]$ when $\bar{v}$ is not a stationary solution of {\rm (pCP)}.
\end{enumerate}
\end{lem}

\begin{proof}
(i). If $t^*(l),\ t^*(r)\geq \tau$, then for all $t\in [0,\tau]$ we have
$$
v(x,t;\bar{v})= \bar{v}\ \ (x\in \R),\quad l(t)= l,\quad r(t)= r.
$$
Therefore, these equalities hold for all $t\geq 0$, and so $\bar{v}$ is a stationary solution of (pCP).

(ii). The sufficient condition follows from Lemma \ref{lem:finite-waiting-time}, and the necessary condition follows from Lemma \ref{lem:no-waiting-time}.

(iii). The conclusion is a consequence of (ii).

(iv). Denote the corresponding density functions of $\bar{v}(x)$ and $v(x,t;\bar{v})$ by $\bar{u}(x)$ and $u(x,t;\bar{u})$, respectively. For any $T>0$ and any test function $\varphi\in C^\infty_c (Q_T)$ ($Q_T:= \R\times (0,T)$) with $\varphi\geq0$,
\begin{eqnarray*}
\iint_{Q_T} [\bar{u}\varphi_t + \bar{u}^m \varphi_{xx}] dx dt & = & \int_{\R} \bar{u}(x)[\varphi(x,T)-\varphi(x,0)]dx + \int_{0}^T \int_{l}^r (\bar{u}^m)_{xx}(x) \varphi(x,t) dxdt\\ & & + \int_0^T [(\bar{u}^m)_x (l)\varphi (l,t) - (\bar{u}^m)_x (r)\varphi (r,t) ] dt\\
&\geq & \int_{\R} \bar{u}(x)[\varphi(x,T)-\varphi(x,0)]dx + \int_{0}^T \int_{l}^r (\bar{u}^m)_{xx}(x) \varphi(x,t) dxdt\\
& = & \int_{\R} \bar{u}(x)[\varphi(x,T)-\varphi(x,0)]dx - \iint_{Q_T}f(\bar{u}(x)) \varphi(x,t) dxdt.
\end{eqnarray*}
So $\bar{u}$ is a time-independent very weak solution of (CP), and then
$u(x,t;\bar{u})\geq \bar{u}(x)$ follows from the comparison result (cf. \cite[Theorem 6.5]{Vaz-book}).

To show the strict inequality, we note that, in the domain $(l,r)$, both $\bar{v}$ and $v$ are positive, and so the classical strong maximum principle is applied. On the other hand, by the conclusion in (i), the waiting times at $l$ and $r$ are zero since $\bar{v}$ is not a stationary solution of (pCP). So $v(l,t;\bar{v})>0=\bar{v}(l)$ and $v(r,t;\bar{v})>0=\bar{v}(r)$. This proves the strict inequality on $[l,r]$.
\end{proof}

Due to the strong maximum principle, the Cauchy problem of a RDE does not have
compactly supported stationary solutions. For (pCP), however, such solutions are possible since the strong maximum principle is not true on the free boundaries. In fact, we will show in Section 5 (see details in Lemma \ref{lem:finite-infinite-GS}) that, if $f$ is a bistable nonlinearity with
$$
f(u)= -\lambda u^\alpha (1+o(1))\mbox{\ \  as\ \  } u\to 0+0,
$$
for some $\lambda>0$ and $1\leq \alpha<m$, then the bistable pRPME has nonnegative compactly supported stationary solution $\bar{v}$, which is a ground state solution satisfying
$$
\bar{v}(x)>0 \mbox{ in }(-L,L),\quad \bar{v}(x)=0 \mbox{ for }|x|\geq L,
$$
and
$$
\bar{v}(x)\sim A (x\pm L)^{\frac{2(m-1)}{m-\alpha}} \mbox{\ \ as\ \ }x\to \mp L\pm 0.
$$
Since $\frac{2(m-1)}{m-\alpha}\geq 2$, by Lemma \ref{lem:ss-problem} (ii) we know that $\bar{v}$ is also a stationary solution of (pCP).

Comparing the ZKB solution of PME (named after Zel'dovich, Kompaneets and Barenblatt) with the solution of (pCP) for monostable and combustion reactions, it is easy to see that compactly supported
stationary solutions do not exist in such problems. But they can exist for (pCP) with bistable reactions (see Lemma \ref{lem:finite-infinite-GS}).
We now show that it is unique when it exists.

\begin{lem}\label{lem:unique-GS}
Assume {\rm (F)} and $V$ is a compactly supported stationary solution of {\rm (pCP)} with
$$
V(x)> 0 \mbox{ in } (-L,L), \quad V(x) =0 \mbox{ for } |x|\geq L,
$$
for some $L>0$. Then
\begin{equation}\label{k-V-GSs}
 \quad V(x)= V(-x) \mbox{ and } V' (x)<0 \mbox{ in } (0, L),
 \quad V'(0)=V' (\pm L) =0.
\end{equation}
(So, $V$ is a Type II ground state solution of {\rm (pCP)}.)
Moreover, such compactly supported stationary solution is unique.
\end{lem}

\begin{proof}
By the equation we know that $V$ is symmetric with respect to each positive critical points, so we have the symmetry and monotonicity in \eqref{k-V-GSs} and $V'(0)=0$. The equality $V'(\pm L)=0$ follow from Lemma \ref{lem:ss-problem}.

We now prove the uniqueness of such compactly supported stationary solution. Set
\begin{equation}\label{w-expression}
w := \Big( \frac{m-1}{m}V\Big)^{\frac{1}{m-1}} V'.
\end{equation}
Then $w(\pm L)=0$. Hence, on the $(V,w)$-phase plane (cf. Subsections 4.1 and 5.1 for more details), $V(x)$ corresponds to a homoclinic orbit starting and ending at $(0,0)$. Using the equation in (pCP) we have
$$
w'= -f \left( \Big(\frac{m-1}{m}V\Big)^{\frac{1}{m-1}}\right).
$$
Multiplying this equation by \eqref{w-expression} and integrating it over $[x,0]$ for $x\in [-L,0)$, we have
$$
w^2 (x) = F(V(x)) - F(V(0))
$$
with
$$
F(v) := -2\int_0^{v} f \left( \Big(\frac{m-1}{m}r\Big)^{\frac{1}{m-1}}\right) \cdot  \Big(\frac{m-1}{m}r\Big)^{\frac{1}{m-1}} dr.
$$
Taking $x=-L$ we have $w^2(-L) =0 = F(0)-F(V(0))$, and so $F(V(0))=0$. For $x\in (-L,0)$ we have $V'(x)>0$ and $w(x)>0$. Thus,
$$
F(v)>F(V(0))=0,\quad v\in (0,V(0)).
$$
(This happens only if $f(u)<0$ for $0<u\ll 1$.) Therefore,
\begin{equation}\label{def-tilde-theta1}
\Theta_1:= V(0)
\end{equation}
is the smallest positive zero of $F(\cdot)$, and there exists $\Theta \in (0, \Theta_1)$ such that
\begin{equation}\label{def-tilde-theta}
f \left( \Big(\frac{m-1}{m}\Theta\Big)^{\frac{1}{m-1}}\right)= g(\Theta)=0 \mbox{ and }
f \left( \Big(\frac{m-1}{m}v\Big)^{\frac{1}{m-1}}\right) >0 \mbox{ for  }v\in (\Theta, \Theta_1).
\end{equation}
Consequently, for $x<0$, $V(x)$ is uniquely determined by the problem
$$
\Big( \frac{m-1}{m}V\Big)^{\frac{1}{m-1}} V' = \sqrt{F(V(x))} \mbox{\ for\ }x<0, \qquad V(0)= \Theta_1,\quad V'(0)=0.
$$
This proves the uniqueness of compactly supported stationary solution of (pCP), that is, the Type II ground state solution with one connected support.
\end{proof}

We have given sufficient and necessary conditions, and uniqueness for compactly supported stationary solutions of the Cauchy problem (pCP). Clearly, they correspond to the analogues of the original problem (CP). Using them we finally give some examples of initial data, which are not necessarily to be stationary solutions, such that the waiting times on their boundaries are infinite. For example, for a ground state solution $V$ of the bistable pRPME in the previous lemma, we choose $v_0(x)\in C(\R)$ such that
$$
v_0(x)\leq V(x) \mbox{ in }\R,\quad v_0(x)>0 \mbox{ in }(-L,L).
$$
Then $v_0$ lies below $V$ and has the same initial boundaries $\pm L$ as $V$.
Since the boundaries $l(t)$ and $r(t)$ of $v(x,t;v_0)$ are monotone in time and since $v$ always lies below $V$ by comparison, we have $t^*({\pm L})=\infty$.

\subsection{Lipschitz continuity, Darcy law and strict monotonicity}

We now address the Lipschitz continuity, the Darcy law and the monotonicity for each free boundary.
For definiteness, in this subsection, we work with the right free boundary $x=r(t)$, that is,
$$
v(x,t)>0 \mbox{ for } r(t)-1\ll x< r(t),\quad v(x,t)=0 \mbox{ for } x\geq r(t).
$$
We use one-sided partial derivatives in space and time:
$$
D^\pm r(t) := r'(t \pm 0),\quad  D^\pm_x v(x,t) = v_x(x \pm 0, t).
$$
Also we use $r'(t)$ to denote the derivative of $r$ when $D^+r(t)=D^-r(t)$.

First we recall the related results on PME.

\begin{lem}[{\cite[Chapter 15]{Vaz-book}}]\label{lem:Darcy-law-PME}
Let $v$ be the global solution of {\rm (pCP)} with $g\equiv 0$, $r(t)$ be its right free boundary with waiting time $t^*\geq 0$. Then, for any
$t_0>t^*$, there exists a small $\varepsilon>0$ such that
\begin{enumerate}[{\rm (i).}]
\item $r\in C^\infty ((t_0 -\varepsilon, t_0 +\varepsilon))$, $v\in C^\infty (\mathcal{N}_\varepsilon)$ for
$$
\mathcal{N}_\varepsilon = \{(x,t) \mid x\leq r(t) \mbox{ and } |(x-x_0, t-t_0)|\leq \varepsilon\},\quad x_0 := r(t_0);
$$

\item there holds
$$
r'(t_0) = - D^-_x v(x_0,t_0)>0,\quad r''(t)= m r'(t) D^-_{xx} v(x_0,t_0);
$$

\item for any positive integer $j$, there exists $C_j$ depending on $(x_0,t_0), m,j,\varepsilon,v$ such that
 $$
 \left| \frac{\partial^j}{\partial x^j} v(x,t) \right|\leq C_j,\quad (x,t)\in \mathcal{N}_\varepsilon.
 $$

\end{enumerate}
\end{lem}

We will use this lemma to prove the important properties for the free boundaries, that is, the $C^1$ regularity and the Darcy law. Since we have no semiconvexity for the free boundary $r(t)$ of RPME as for the PME (cf. \cite[\S 15.4.1]{Vaz-book}), we have to deal with the problem in a more complicated way.

We first show the left side continuity of $v_x$:

\begin{lem}\label{lem:left-continuity}
Let $v$ be the solution of {\rm (pCP)}, $r(t)$ be its right free boundary. Then
\begin{equation}\label{D-x-v=lim}
D^-_x v(r(t),t)  = \lim\limits_{x \to r(t)-0} v_x(x,t),\quad t>0.
\end{equation}
\end{lem}

\begin{proof}
As we mentioned before, $v_{xx}$ is bounded from below: $v_{xx}\geq -C_1 t - C_2$ for some $C_1, C_2>0$. Hence
$$
\tilde{v}(x,t) := v(x,t) + \frac12 (C_1 t + C_2) x^2
$$
is a continuous and convex function of $x$ for $t>0$. So, for every $x_1\in \R,\ t_1>0$, the function $\tilde{v}(\cdot, t_1)$ admits one-sided derivatives at $x_1$:
$D^+_x \tilde{v}(x_1 , t_1)$ and $D^-_x \tilde{v}(x_1, t_1)$ with
$D^-_x \tilde{v}(x_1 ,t_1)\leq D^+_x \tilde{v}(x_1,t_1)$, and the functions $D^\pm_x \tilde{v} (\cdot , t_1)$ are non-decreasing ones.
Consequently, $D^\pm_x v(x_1, t_1)$ exist with $D^-_x v(x_1,t_1)\leq D^+_x v(x_1, t_1)$ and $D^\pm_x v(x,t_1)+(C_1 t_1 +C_2)x$ are non-decreasing in $x$. Note that, for small $\varepsilon>0$, when $x\in I_1:= [r(t_1)-\varepsilon ,r(t_1))$, $v(x, t_1)$ is smooth and so
$$
v_x(x,t_1)+(C_1 t_1+C_2)x \leq D^\pm_x v(r(t_1),t_1) + (C_1 t_1 +C_2)r(t_1).
$$
The right hand side is bounded since $D^-_x v(r(t_1),t_1)\leq 0 = D^+_x v(r(t_1) , t_1)$. So $k_1:= \lim\limits_{x\to r(t_1)-0}v_x(x,t_1)$ exists and $k_1\leq D^-_x v(r(t_1),t_1)$. The conclusion \eqref{D-x-v=lim} is proved if we can show the equality. By contradiction, we assume that $k_1< D^-_x v(r(t_1),t_1)\leq 0$. Then there exists $k_2$ lying between them and
$$
v_x(x,t_1)< k_2,\qquad x\in I_1.
$$
For any $x_2, x_3$ satisfying $r(t_1)-\varepsilon \ll x_3 <x_2 <r(t_1)$ we have
$$
\frac{v(x_3,t_1)-v(x_2,t_1)}{x_3-x_2}<k_2.
$$
Taking the limit as $x_2\to r(t_1)-0$ and using the continuity of $v$ we obtain
$$
\frac{v(x_3,t_1)-v(r(t_1),t_1)}{x_3- r(t_1)}\leq k_2.
$$
This implies that $D^-_x v(r(t_1),t_1)\leq k_2$, a contradiction.
\end{proof}

\begin{thm}\label{thm:Darcy-law}
Let $v$ be the solution of {\rm (pCP)}, $r(t)$ be its right free boundary with waiting time $t^*\in [0,\infty)$. Then $r\in C^1((t^*,\infty))$ and the Darcy law holds:
\begin{equation}\label{C1-Darcy-law}
r'(t) = - D^-_x v(r(t),t) \geq 0,\quad t>t^*.
\end{equation}
Moreover, $r'(t)\not\equiv 0$ in any interval in $(t^*,\infty)$.
\end{thm}

\begin{proof}
We prove the Darcy law at any given time $t_0>t^*$. Denote $x_0 := r(t_0)$ and $\beta:= - D^-_x v(x_0,t_0)$.

\medskip
{\it Step 1. Assume $\beta >0$ and to show $D^+ r(t_0) = \beta$.} Consider the following problem
$$
\left\{
 \begin{array}{ll}
 \bar{v}_t = (m-1)\bar{v} \bar{v}_{xx} + \bar{v}_x^2 + K(m-1) \bar{v} , & x\in \R,\ t>0,\\
 \bar{v} (x,0)= \bar{v}_0(x), & x\in \R,
 \end{array}
 \right.
$$
where $\bar{v}_0(x)$ satisfies
$$
\bar{v}'_0(x_0-0) = -\beta,\quad \bar{v}_0(x)\geq v(x,t_0) \mbox{ for } x\in \R,
$$
and $\bar{v}_0(x)$ is sufficiently smooth so that the right free boundary $\bar{r}(t)$ of $\bar{v}(\cdot,t)$ is smooth in $t\geq 0$. Then $\bar{v}$ is a supersolution of (pCP), and so
$$
r(t+t_0) \leq \bar{r}(t), \quad t\geq 0.
$$
To estimate $\bar{r}(t)$ precisely, we change the time variable from $t$ to $s$ by
$$
s = S(t):= \frac{e^{K(m-1)t} -1}{K(m-1)} \Leftrightarrow t = T(s):= \frac{\ln (1+K(m-1)s)}{K(m-1)},
$$
and define
$$
V(x,s):= \frac{\bar{v}(x, T(s))}{1+K(m-1)s},
$$
then we have
\begin{equation}\label{problem-of-w-u3}
\left\{
 \begin{array}{ll}
  V_s = (m-1)V V_{xx} + V_x^2 ,& x\in \R, \ s>0,\\
  V(x,0)= \bar{v}_0(x), & x\in \R.
  \end{array}
  \right.
\end{equation}
By Lemma \ref{lem:Darcy-law-PME}, for some small $\varepsilon_0>0$, the right free boundary $x=R(s) = \bar{r}(T(s))$ of $V(\cdot,s)$ satisfies
$$
R'(s) = -D^-_x V(R(s),s),\quad |R''(s)|\leq \bar{C}_1,\quad 0\leq s\leq S(\varepsilon_0),
$$
or, equivalently,
$$
\bar{r}'(t)= R'(s)\frac{ds}{dt} = - D^-_x \bar{v}(\bar{r}(t),t),\quad
|\bar{r}''(t)|\leq \bar{C},\quad 0\leq t\leq \varepsilon_0.
$$
Consequently, we have upper estimate for $r(t+t_0)$:
\begin{equation}\label{upper-est-rtt0}
r(t+t_0)\leq \bar{r}(t) \leq \bar{r}(0)+ \bar{r}'(0)t + \bar{C} t^2
\leq x_0 +\beta t +\bar{C}t^2,\quad 0\leq t\leq \varepsilon_0.
\end{equation}

Next we consider the problem
$$
\left\{
 \begin{array}{ll}
 \underline{v}_t = (m-1) \underline{v} \underline{v}_{xx} + \underline{v}_x^2 - K(m-1) \underline{v} , & x\in \R,\ t>0,\\
 \underline{v} (x,0)= \underline{v}_0 (x), & x\in \R,
 \end{array}
 \right.
$$
where $\underline{v}_0(x)$ satisfies
$$
\underline{v}'_0(x_0-0) = -\beta,\quad \underline{v}_0(x)\leq v(x,t_0) \mbox{ for } x\in \R,
$$
and $\underline{v}_0(x)$ is sufficiently smooth so that the right free boundary $\underline{r}(t)$ of $\underline{v}(\cdot,t)$ is smooth in $t\geq 0$. Then $\bar{v}$ is a subsolution  of (pCP) and, as above,
\begin{equation}\label{lower-est-rtt0}
r(t+t_0)\geq \underline{r}(t) \geq x_0 + \beta t - \underline{C} t^2,\quad 0\leq t\leq \varepsilon_0,
\end{equation}
for some positive $\underline{C}>0$. Combining with \eqref{upper-est-rtt0} we have $D^+ r(t_0)=\beta$.

\medskip
{\it Step 2. To show $D^+ r(t)\not\equiv 0$ in any interval in
$(t^*, \infty)$ with positive measure.}
Note that this conclusion implies that $r(t)$ is strictly increasing in $t>t^*$. By the positivity persistence, $r(t)$ is increasing in $t>t^*$. Assume by contradiction that $D^+r(t)= 0$ for $t\in [t_1,t_2] \subset (t^*, \infty)$. We will derive
\begin{equation}\label{equiv=0-t2}
D^+r(t)=0,\quad t\in [0,t_2],
\end{equation}
which, then, contradicts the definition of $t^*$ and the fact $t_2 >t^*$.

To show \eqref{equiv=0-t2}, we suppose, on the contrary, that $t_3 := \inf \{\tilde{t}\mid D^+ r(t)\equiv 0 \mbox{ in } [\tilde{t},t_2]\} \in (0,t_1]$, and, by the continuity, there exists $\tau \in (t^*,t_3)$ such that
\begin{equation}\label{def-Delta}
 D^+ r(\tau) >0 \mbox{\ \ and\ \ } r(t_3)-r(\tau)<
\Delta := \frac{P\pi}{\mu} (1-e^{-\mu (t_2 -t_1)}),
\end{equation}
where $-P$ is the lower bound of $v_{xx}$: $v_{xx}>-P$ for $x\in \R,\ t\in [t^*,t_2]$, and $ \mu:= K(m-1)+ mP\pi$. By Step 1 we have
$$
0< 2\rho := D^+ r(\tau ) = -D^-_x v(r(\tau ), \tau ).
$$
Hence, for $\varepsilon := \frac{\rho}{P}$ we have
$$
-D^-_x v(x,\tau)\geq \rho \mbox{\ \  and\ \  } v(x,\tau ) \geq - \rho (x-r(\tau)),\quad r(\tau)-\varepsilon \leq x \leq r(\tau).
$$
Denote
$$
a:= \frac{P\pi}{\rho},\quad h:= \frac{\rho^2}{P\pi},\quad b (t):= \frac{P\pi}{\mu}(1-e^{-\mu (t-\tau)})
$$
and using the subsolution  $\underline{v}$ as in Lemma \ref{lem:no-waiting-time} we conclude that
\begin{equation}\label{increase-difference}
r(t) \geq r(\tau) + b (t),\quad t>\tau.
\end{equation}
In particular, at $t=t_3$ we have
$$
r(t_3) = r(t_2) \geq r(\tau) + b (t_2) \geq r(\tau) + \Delta,
$$
contradict \eqref{def-Delta}. This proves \eqref{equiv=0-t2}, and then the conclusion in this step is proved.

\medskip
{\it Step 3. Assume $\beta >0$ and to show $D^- r(t_0)= \beta$.}  We assume that, for some increasing sequence $\{t_n\}$ tending to $t_0$,
\begin{equation}\label{tn-to-t0}
\beta' := \lim\limits_{n\to \infty} \frac{r(t_n) -r(t_0)}{t_n -t_0} ,\quad D^+ r(t_n) = \beta_n := -D^-_x v(r(t_n),t_n)>0.
\end{equation}
(The latter holds by Step 2.) If we can show that, for each of such time sequence $\{t_n\}$, the limit $\beta' =\beta$. Then we obtain the conclusion $D^- r(t_0)=\beta$.

By contradiction, we assume $\beta' >\beta$. The reversed case is proved similarly. Then there exist a small $\delta>0$ and a large $n_0$ such that
\begin{equation}\label{contra-gradient}
\frac{r(t_n) -r(t_0)}{t_n -t_0} \geq \beta +2\delta,\quad n\geq n_0.
\end{equation}
For any given $n\geq n_0$, using the estimates \eqref{upper-est-rtt0} and \eqref{lower-est-rtt0} at the time $t_n$ instead of $t_0$ we see that, for some $\varepsilon_1 \in (0,\varepsilon_0)$ and some positive numbers $\underline{C}',\ \bar{C}'$, there holds
\begin{equation}\label{est-rn}
r(t_n) +\beta_n t -\underline{C}' t^2 \leq r(t+t_n) \leq
r(t_n) +\beta_n t +\bar{C}' t^2 , \quad 0\leq t\leq \varepsilon_1.
\end{equation}
Taking $t = t_0 -t_n$ in the second inequality we have
$$
r(t_0) - r(t_n) \leq  \beta_n (t_0 -t_n) + \bar{C}'(t_0 -t_n)^2.
$$
Together with \eqref{contra-gradient} we see that, when $n$ is sufficiently large,
$$
\beta_n \geq \beta +\delta.
$$
Taking $n$ large such that $\epsilon := t_0 -t_n \leq \frac{\varepsilon_1}{2}$ and taking $t=2\epsilon$ in \eqref{est-rn} we have
\begin{eqnarray*}
r(t_0 +\epsilon) & \geq & r(t_n)+ 2 \beta_n \epsilon +O(\epsilon^2) \\
& \geq & r(t_n)+ \beta_n \epsilon + \bar{C}' \epsilon^2 + (\beta +\delta) \epsilon +O(\epsilon^2)\\
& \geq & r(t_0) + (\beta+\delta)\epsilon +O(\epsilon^2).
\end{eqnarray*}
Together with \eqref{upper-est-rtt0} we have
$$
O(\epsilon^2)  \geq \delta \epsilon + O(\epsilon^2).
$$
This is a contradiction when $\epsilon $ is small, or, equivalently $n$ is sufficiently large.

\medskip
{\it Step 4. Existence and continuity of $r'(t)$ at $t_0$ when $\beta = r'(t_0)>0$.} Combining Step 1 and Step 3 we have
$$
r'(t_0) = D^+ r(t_0) =D^- r(t_0) = \beta, \quad \mbox{ when } \beta =-D^-_x v(x_0, t_0)>0.
$$

The proof in Step 3 also implies that, for any subsequence $\{n_i\}$ of $\{n\}$, $\beta_{n_i} \geq \beta+\delta$ is impossible, and so $\limsup_{n\to \infty} \beta_n \leq \beta$. In a similar way as in Step 3, one can show that
$\beta_{n_i} \leq \beta -\delta$ is impossible, and so $\liminf_{n\to \infty} \beta_n \geq \beta$. Therefore, $r'(t)$ is continuous at $t_0$ from the left side.
On the other hand, if the time sequence $\{t_n\}$ in \eqref{tn-to-t0} is a decreasing one tending to $t_0$ from right side, then one can prove as above that $r'(t_n)\to \beta\ (n\to \infty)$. This means that $r'(t)$ is continuous at $t_0$ from the right side. Combining the above conclusions together we get $r\in C^1$ at any time $t_0$ where $r'(t_0)>0$.

\medskip
{\it Step 5. The case $\beta := -D^-_x v(x_0,t_0)=0$.} For any small $\varepsilon>0$, we define
$$
\tilde{v}_0 (x) := v(x,t_0) + \varepsilon (x_0 -x),\quad x\in \R,
$$
and consider the equation of $v$ for $t\geq t_0$, with initial data $v(x,t_0)$ replaced by $\tilde{v}_0 (x)$. Then by Step 1 we have $D^+ r(t_0)\leq - D^-_x \tilde{v}_0 (x_0)= \varepsilon$. Since $\varepsilon>0$ is arbitrary we actually have $D^+ r(t_0)=0$. Now as the proof in Step 3, with $\beta>0$ being replaced by $\beta=0$, one can prove in a similar way that $D^- r(t_0)=0$, and so $r'(t_0)=0$. This implies that the Darcy law \eqref{C1-Darcy-law} remains hold even if $\beta=0$.  Finally, the continuity of $r'(t)$ at $t_0$ where $r'(t_0)=0$ can be shown as in Step 4.

This completes the proof of the theorem.
\end{proof}

\begin{rem}\label{rem:Darcy=0}\rm
From this theorem we know that, for $t>t^*$, $r'(t)\geq 0$ and it
is not identical zero in any time interval. Though we do not have the strong conclusion $r'(t)>0$ as in PME (cf. \cite[Corollary 15.23]{Vaz-book}), the current properties are enough to continue the study for the asymptotic behavior of the solutions. For example, they are enough in the application of the zero number argument.
\end{rem}

\section{Zero Number Argument and General Convergence Result}

In this section we present the intersection number properties, prove that the solution is strictly monotone outside its initial support, and then prove the general convergence result.

\subsection{Zero number and intersection number properties}
A useful tool we will use in the qualitative study is the so-called zero number argument. For the convenience of the readers, we first prepare some basic results in this area. Consider
\begin{equation}\label{linear}
\eta_t=a(x,t) \eta_{xx}+b(x,t) \eta_x+c(x,t) \eta\quad \mbox{ in } E_0 := \{(x,t) \mid b_1 (t) <x < b_2 (t),\ t\in (t_1, t_2) \},
\end{equation}
where $b_1$ and $b_2$ are continuous functions in $(t_1,t_2)$.
For each $t\in (t_1, t_2)$, denote by
$$
\Z(t) := \#\{x\in \bar{J} (t)\mid \eta(\cdot ,t)=0\}
$$
the number of zeroes of $\eta (\cdot,t)$ in the interval $\bar{J}(t) := [b_1(t), b_2(t)]$.
A point $x_0\in \bar{J}(t)$ is called a {\it multiple zero} (or {\it degenerate zero}) of
$\eta (\cdot, t)$ if $\eta (x_0,t) = \eta_x (x_0,t)=0$.
In 1988, Angenent \cite{A} proved a zero number diminishing property, and in 1998, the conditions in \cite{A} were weakened by Chen \cite{ChXY} for solutions with lower regularity. One of their results can be summarized as the following:

\begin{prop}[\cite{A, ChXY}, zero number diminishing properties]\label{prop:zero}
Assume the coefficients in \eqref{linear} satisfies
\begin{equation}\label{smoothy}
a, a^{-1}, a_t, a_x, b, c\in L^\infty.
\end{equation}
Let $\eta$ be a nontrivial $W^{2,1}_{p,loc}$ solution of \eqref{linear}. Further assume that, for $i=1,2$, either
\begin{equation}\label{bdry-cond}
\eta (b_i(t),t )\not= 0,\quad t\in (t_1, t_2),
\end{equation}
or,
\begin{equation}\label{bdry-cond1}
b_i(t) \in C^1 \mbox{ and } \eta (b_i(t),t ) \equiv 0,\quad t\in (t_1, t_2), \end{equation}
or,
\begin{equation}\label{bdry-cond2}
b_i(t) \in C^1 \mbox{ and } \eta_x (b_i(t),t) \equiv 0,\quad t\in (t_1, t_2).
\end{equation}
Then
\begin{itemize}
\item[(i)] $\Z(t)$ is finite and decreasing in $t\in (t_1, t_2)$;
\item[(ii)] if $s\in (t_1, t_2)$ and $x_0 \in \bar{J}(s)$ is a multiple zero of $\eta(\cdot,s)$, then $\Z (s_1) > \Z (s_2)$ for all $s_1, s_2$ satisfying $t_1 < s_1< s< s_2 <t_2$.
\end{itemize}
\end{prop}

Note that, in the original results in \cite{A,ChXY}, the problem was considered in fixed intervals, that is, $b_i(t)$ are constants. Using our assumptions \eqref{bdry-cond}-\eqref{bdry-cond2}, for any time interval with small length, we can straighten the domain boundaries so that the zero number diminishing properties remain hold. The boundedness assumption of $a^{-1}$ in \eqref{smoothy} is clearly a difficulty in using the zero number argument for RPMEs. Nevertheless, under some additional conditions, we can prove a diminishing property for the number of intersection points between two solutions of (pCP), despite the degeneracy of the equation. The analogue results for RDEs was proved in \cite[Lemma 2.4]{DLZ}, but the proof for our current result is more complicated due to the degeneracy.

For $i=1,2$, let $v_i$ be the solution of {\rm (pCP)} with initial data in
$$
\{ \psi\in C(\R) \mid \psi(x)>0 \mbox{ in } (l_i(0) ,r_i(0)),\quad \psi(x)=0 \mbox{ for } x\not\in (l_i(0),r_i(0))\}.
$$
Denote $l_i(t)$ and $r_i(t)$ the left and right free boundaries of $v_i$, respectively. Write
$$
l(t):= \max\{l_1(t),l_2(t)\},\quad r(t):= \min\{r_1(t),r_2(t)\}, \qquad t\geq 0.
$$
Then in the case where $l(t)<r(t)$, the number of the positive intersection points between $v_1(\cdot,t)$ and $v_2(\cdot,t)$ in $\R$, which is denoted by $\mathcal{Z}_0(t)$ as in \eqref{def-Z0}, is the same as the number of intersection points in $J_0(t):= (l(t),r(t))$.

\begin{thm}\label{thm:zero num 1}
Assume {\rm (F)}, $v_1, v_2$ are given as above. Assume that $\mathcal{Z}_0(0)$ is finite. Then there exists a time $\tau\in [0,+\infty]$  , such that
\begin{enumerate}[{\rm (i)}]
\item $\mathcal{Z}_0(t)$ decreases for $t\in (0,\tau)$ and for $t\in (\tau,\infty)$;
\item $\lim\limits_{t\rightarrow \tau + 0}\mathcal{Z}_0(t)=1$ and $\lim\limits_{t\rightarrow \tau -0}\mathcal{Z}_0(t)=0$ if $\tau\in (0,\infty)$.
\end{enumerate}
\end{thm}
The proof of this theorem consists of the following two lemmas.

\begin{lem}\label{lem:diminish-intersection}
Assume {\rm (F)}, $v_1, v_2$ are given as above. Further assume that $l(0)<r(0)$ and one of the following
hypotheses holds:
\begin{enumerate}[{\rm (a)}]
\item $l_1(0)\not= l_2(0),\ r_1(0)\not= r_2(0)$;
\item $\mathcal{Z}_0(0)<\infty$.
\end{enumerate}
Then
\begin{enumerate}[{\rm (i).}]
\item $\mathcal{Z}_0(t)$ is decreasing in $t>0$;
\item $\mathcal{Z}_0(t)$ is strictly decreasing when $t$ passes through $s$, if $v_1(\cdot,s)-v_2(\cdot,s)$ has a degenerate zero in $J_0(s)$, or, if one interior zero moves to the boundary $l(t)$ or $r(t)$ as $t\to s-0$.
\end{enumerate}
\end{lem}

\begin{proof}
We first prove the conclusions under the assumption (a).
The main complexity is caused by the possibility that an interior intersection point moves up to the domain boundaries. Without loss of generality, we assume
$$
l_1(t)<l_2(t) \mbox{ for small } t\geq 0,\quad r_1(t)<r_2(t) \mbox{ for all } t\geq 0.
$$
Then the left boundaries $l_1(t)$ and $l_2(t)$ may meet after some time, but the right boundaries will not.

At the early stage, say $t\in [0,t_1]$ for some $t_1>0$, $l(t)=l_2(t)>l_1(t)$, and $\eta(x,t):= v_1(x,t)-v_2(x,t)$ satisfies a linear equation
\begin{equation}\label{eta-eq-0}
\eta_t = a(x,t) \eta_{xx} + b(x,t) \eta_x + c(x,t) \eta, \quad x\in J_0(t):=(l(t),r(t)),\ t\in (0,t_1],
\end{equation}
with
$$
a(x,t) := (m-1)v_1(x, t) ,\quad b(x,t):= v_{1x}(x, t) + v_{2x} (x,t),
$$
and
$$
c(x,t):= (m-1) v_{2xx}(x,t)+ \left\{
 \begin{array}{ll}
  \frac{g(v_1(x, t))-g( v_2 (x,t))}{v_1(x,t)- v_2(x,t)}, & v_1(x,t)\not= v_2(x,t),\\
  0 , & v_1(x,t)=v_2(x,t).
  \end{array}
  \right.
$$

By the continuity of $v_1,v_2$,
$$
\eta(l(t),t) = v_1(l_2(t),t)\geq \rho,\quad \eta(r(t),t)= -v_2(r_1(t),t) \leq -\rho,\quad t\in [0,t_1],
$$
for some small $\rho>0$. Take two smooth curves $\hat{l}(t)$ and $\hat{r}(t)$ such that
$$
l(t)<\hat{l}(t)\ll l(t)+1,\quad r(t)-1\ll \hat{r}(t)<r(t),
$$
\begin{equation}\label{eta=not=0}
\eta(x,t)\not= 0,\quad x\in [l(t), \hat{l}(t)]\cup [\hat{r}(t),r(t)],\ t\in [0,t_1],
\end{equation}
and, for some $\rho_1>0$,
$$
v_1(x,t),\ v_2(x,t)\geq \rho_1 \mbox{\ \ in \ \ } E_1:=\{(x,t) \mid x\in \hat{J}(t):= [\hat{l}(t),  \hat{r}(t)],\ t\in [0,t_1]\}.
$$
Then both $v_1$ and $v_2$ are classical in $E_1$, we can use the classical zero number diminishing property in $E_1$ to conclude that $\mathcal{Z}_{\hat{J}(t)}(t)$, which denotes the zero number of $\eta(\cdot,t)$ in the interval $\hat{J}(t)$, is finite, decreasing in $t$, and strictly decreasing when $t$ passes through a moment $s$ when $\eta(\cdot,s)$ has a degenerate zero. Note that, the same conclusions hold true for $\mathcal{Z}_0(t)$ since it equals to $\mathcal{Z}_{\hat{J}(t)}(t)$ by \eqref{eta=not=0}.

If $l_2(t)$ will never approach $l_1(t)$ for all $t>0$, then the conclusions are proved.

We now consider the case: there is a moment $T > t_1$ such that $l_2(t)$ moves leftward and meets $l_1(t)$ as $t\to T-0$, and suppose that $T$ is the first one of such moments. Due to the diminishing property for $\mathcal{Z}_0 (t)$ in the time interval $(0, T)$ we see that there are at most finite moments when $\eta(\cdot,t)$ has degenerate zeros (only in $J_0(t)$). Thus, there exists $t_2\in (t_1,T)$ such that $\eta(\cdot,t)$ has no degenerate zeros in $J_0 (t)$ in the time interval $(t_2,T)$, that is, in this time interval $\eta(\cdot,t)$ has fixed finite number of non-degenerate zeros. For simplicity and without loss of generality, in the rest of the proof we assume that $\eta(\cdot,t)$ has no interior degenerate zeros in $J_0(t)$ for all $t\in (t_2,\infty)$. So we can assume the null curves of $\eta$ in $E_2 := \{(x,t)\mid x\in J_0(t),\ t_2<t<T\}$ are
$$
\gamma_1(t) < \gamma_2(t)<\cdots <\gamma_k(t),\quad t\in (t_2, T),
$$
for some positive integer $k$.
Using \cite[Theorem 2]{Fer} as in the proof of Lemma 2.4 in \cite{DLZ} one can show that
$$
x_j := \lim\limits_{t\to T-0} \gamma_j (t),\quad j=1,2,\cdots, k,
$$
exist. We divide the situation into two cases.

{\it Case 1}. $l(T)< x_1$, that is, $l_2(t)$ moves leftward to catch up with $l_1(t)$ as $t\to T-0$, while the left-most null curve $\gamma_1(t)$ remains on the right of $l(t)$ till $t=T$. In this case, by using the comparison principle for very weak solutions in the domain
$$
E_3 := \{(x,t) \mid -\infty <x<\gamma_1(t),\ t\in (t_2, T)\}
$$
we have
$$
\eta(x,T)>0 \mbox{ for } l(T)<x< x_1,\quad \eta(l(T),T)=\eta(x_1,T)=0.
$$
As a consequence we have $\mathcal{Z}_0(T) =k$. The boundary zero $l(T)$ of $\eta(\cdot,T)$ is not included in $\mathcal{Z}_0(T)$. Recalling the Darcy law we have the additional fact $D^+_x \eta(l(T),T)=0$.
We remark that, this will lead to a contradiction with the Hopf lemma immediately in the Stefan problems for RDEs (cf. \cite{DL, DLZ}). However, in the current problem, the fact $D^+_x \eta(l(T),T)=0$ does not contradict the Hopf lemma, since the Hopf lemma is no longer necessarily to be true at $(l(T),T)$ where the equation \eqref{eta-eq-0} is degenerate.

Next, we consider the time period from the time $T$ to $t_3\in (T,\infty]$, where $t_3$ is the smallest moment when $\gamma_1(t_3)$ meets $l(t_3)$. Using the comparison principle in the domain
$$
E_4:= \{(x,t) \mid -\infty <x<\gamma_1(t),\ t\in [T,t_3)\},
$$
we conclude that $\eta(x,t)\geq 0$ in this domain. This implies that, for $t\in [T,t_3)$, $l_1(t)\leq l_2(t)$, but the equality can hold from time to time. However, no matter $l_1(t)<l_2(t)$ or $l_1(t)=l_2(t)$, we have $v_1, v_2>0$ for $x\in (l(t),\gamma_1(t)]$, and so the strong maximum principle holds in
$$
E_5 := \{(x,t) \mid l(t)<x\leq \gamma_1(t),\ T\leq t<t_3\},
$$
which implies that $\eta(x,t)>0$ in $E_5$. As a consequence, $\mathcal{Z}_0 (t) =k$ for $t\in [T,t_3)$.

At the moment $t_3$, $\gamma_1(t_3)$ meets $l(t_3)$. The following analysis for $t\geq t_3$ is similar as what we will do in Case 2.

{\it Case 2}. $\gamma_1 (t)$ tends to $l(t)$ at $t\to T-0$, and so $l_1(T)=l_2(T)=\gamma_1(T)$ (see Figure 4). Recall that we assumed before that $\eta(\cdot,t)$ no longer has interior degenerate zeros. However, this does not exclude that possibility that $\gamma_1(T)=\gamma_2(T)=\cdots = \gamma_{j_0}(T)=l(T)$, that is, $j_0$-null curves tend to the point $(l(T),T)$ at the same time as $t\to T-0$. Since the discussion is similar, in what follows, we only consider the case
$$
l_1(T)=l_2(T)=l(T)=\gamma_1(T) <\gamma_2(T) <\cdots <\gamma_k(T).
$$
Then we have $\mathcal{Z}_0(T)=k-1$, and
$$
\eta(x,T)<0,\quad l(T)<x<\gamma_2(T).
$$
\begin{figure}
\begin{center}
\includegraphics[width=5in,height=1in]{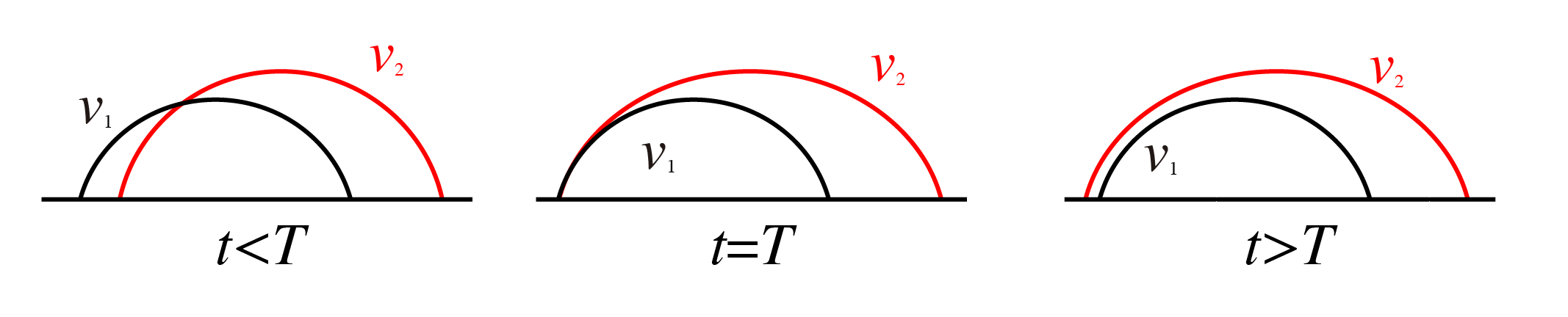}

\caption{Vanishing of zero at $t=T$.}
\end{center}
\end{figure}
As above we next consider the time period from $T$ to $t_4\in (T,\infty]$, where $t_4$ is the smallest moment when $\gamma_2(t_4)$ meets $l(t_4)$. Using the comparison principle for very weak solutions in the domain
$$
E_6:= \{(x,t) \mid -\infty<x<\gamma_2(t),\ t\in [T,t_4)\},
$$
we conclude that $\eta(x,t)\leq 0$ in this domain. This implies that, for $t\in [T,t_4)$, $l_2(t)\leq l_1(t)$, but the equality can hold from time to time. However, no matter $l_2(t)<l_1(t)$ or $l_2(t)=l_1(t)$, we have $v_1, v_2>0$ for $x\in (l(t),\gamma_2(t)]$, and so the strong maximum principle holds in
$$
E_7 := \{(x,t) \mid l(t)<x\leq \gamma_2(t),\ T\leq t<t_4\},
$$
which implies that $\eta(x,t)<0$ in $E_7$. As a consequence, $\mathcal{Z}_0 (t) =k-1$ for $t\in [T,t_4)$.

In the case where $t_4=\infty$, there is nothing left to prove. In the case where $t_4<\infty$, we can analyze the situation for $t\geq t_4$ as in the current case.

As a conclusion, we see that, in case no interior degenerate zeros appear, the zero number of $\eta(\cdot,t)$ in the open interval $J_0(t)$ is decreasing. It is strictly decreasing when some interior null curves touch the boundaries.

Next we consider the assumption (b). Without loss of generality, we assume that $l_1(0)=l_2(0)$, that is, the left free boundaries of $v_1$ and $v_2$ glue together at the beginning. Since $\mathcal{Z}_0(t)<\infty$, there exists $\tilde{x}>l_1(0)=l_2(0)$, such that $v_1(x,0)$ and $v_2(x,0)$ has no intersection points in $(l_1(0),\tilde{x})$. This is a situation like that in Case 1 at time $T$. The rest discussion is then similar as above.

This completes the proof of the lemma.
\end{proof}

\begin{rem}\label{rem:MP-unbounded}\rm
In the above proof, as well as in other related places of the paper, we use the maximum principle for $\eta$. One may worry about the unboundedness of the coefficient $c$ in \eqref{eta-eq-0} at the points where $v_2 =0$. We remark that, though we write $v_1-v_2$ in the form of $\eta$ for convenience, we actually compare the very weak solutions $v_1$ and $v_2$. As in \cite[Theorem 5.5]{Vaz-book} the comparison principle between $v_1, v_2$ holds safely. The main reason is that each of them can be approximated by a sequence of decreasing classical solutions.
\end{rem}

\begin{lem}\label{lem:diminish-intersection 1}
Assume {\rm (F)}, $v_1, v_2$ are given as above. Further assume that $l(0)\geq r(0)$.
Then there exists a time $t_1\in [0,\infty]$ such that
$$
\begin{cases}
\mathcal{Z}_0(t)=0, & t\leq t_1,\\
\mathcal{Z}_0(t)=1, & t_1 <t \ll t_1 +1,\\
\mathcal{Z}_0(t)\mbox{ decreases}, & t>t_1.
\end{cases}
$$
\end{lem}

\begin{proof}
By the assumption $l(0)\geq r(0)$, we see that the supports of two initial data $v_1(x,0)$ and $v_2(x,0)$ do not have overlap. Without loss of generality, we assume $v_1$ lies on the left of $v_2$ in the beginning time, that is $r_1(0)\leq l_2(0)$. Define $t_1 := \sup\{t \geq 0 \mid r_1(t)\leq l_2(t)\}\in [0,\infty]$, then $\mathcal{Z}_0(t)=0$ when $t\leq t_1$ by the definition of $\mathcal{Z}_0(t)$.
\begin{figure}
\begin{center}
\includegraphics[width=5in,height=1in]{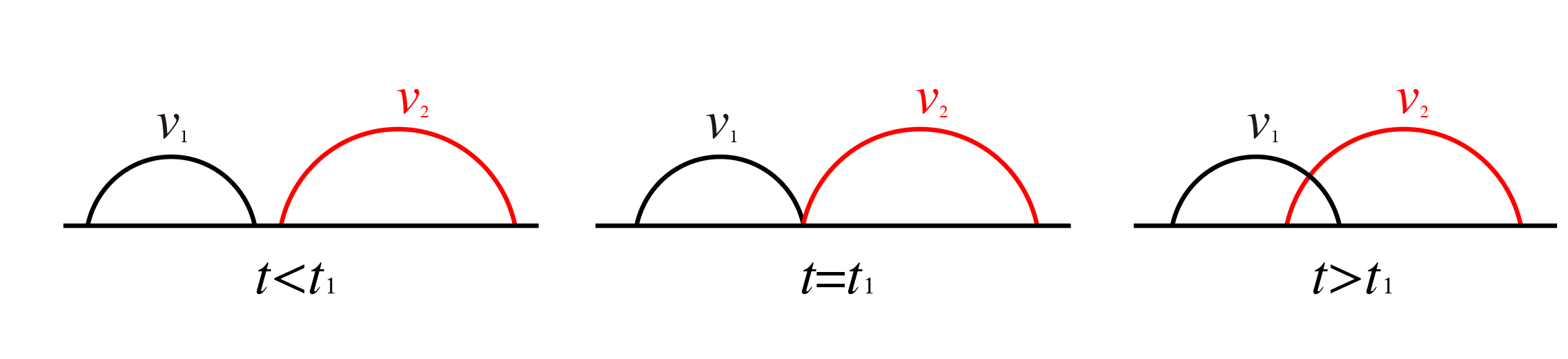}
\caption{Generation of zero at $t=t_1$.}
\end{center}
\end{figure}
When $t_1<\infty$, $r_1(t)$ will surpass $l_2(t)$ at $t=t_1$ (see Figure 5). According to the monotonicity of $v_i$ near the free boundaries, we obtain $\mathcal{Z}_0(t)=1$ for $t_1 < t \ll t_1 +1$. After $t_1$, the problem can be treated like that in Lemma \ref{lem:diminish-intersection}. Then the lemma follows.
\end{proof}

\medskip
\noindent
{\it Proof of Theorem \ref{thm:intersection num}}. The proof is almost the same as that of Theorem \ref{thm:zero num 1}. We only note that the support of initial data may not be singly connected, but consists of finite intervals. Similar to $t_1$ in Lemma \ref{lem:diminish-intersection 1}, we define $t_j$ as the $j$-th time for two components (in one of them $v_1>0$ and in the other $v_2>0$) to meet. It is easy to see that $\# \{t_j\}$ is at most finite up to the choice of initial data. However, more complex than the case in Lemma \ref{lem:diminish-intersection 1}, $\mathcal{Z}_0(t)$ may not increase strictly at $t=t_j$, since there may exist other interior degenerate zeros of $v_1-v_2$ which vanishes at the same time. To exclude such cases, we select a subset of $\{t_j\}$ such that $\mathcal{Z}_0(t)$ strictly increases at these times. This proves Theorem \ref{thm:intersection num}.
\qed

\subsection{Monotonicity outside of the initial support}
Assume the initial data $u_0$ of (CP) belongs to
\begin{equation}\label{def-X1}
\mathfrak{X}_1 :=\{\psi \in C(\R) \mid \psi(x)>0 \mbox{ for } x\in (-b,b), \mbox{ and } \psi(x)=0 \mbox{ for } |x|\geq b\}.
\end{equation}
Then the solution has exactly two free boundaries: $l(t)<r(t)$ with $l(0)=-b$ and $r(0)=b$. Moreover, $l(t)$ is strictly decreasing after the waiting time $t^*_1(-b)$, $r(t)$ is strictly increasing after the waiting time $t^*_2(b)$, and $\mbox{spt} [u(\cdot,t)] =  [l(t),r(t)]$.
We now prove a monotonicity property for $u(\cdot,t)$ outside of $[-b,b]$.

\begin{lem}\label{lem:monotonocity}
Assume {\rm (F)} and $u_0\in \mathfrak{X}_1$.  Then

\begin{enumerate}[{\rm (i).}]
\item for any given $\bar{x}>b$, if $r(\bar{t})=\bar{x}$ for some
$\bar{t}>  t^*(b)$, then $v_x (\bar{x},t)<0$ for $t> \bar{t}$.
Similarly, for any given $\bar{x}' <-b $, if $l(\bar{t}')=\bar{x}'$ for some $\bar{t}' >t^* (-b)$, then $v_x (\bar{x}',t) >0$ for $t> \bar{t}'$;

\item $-2b\leq r(t)+l(t)\leq 2b$ for all $t>0$.
\end{enumerate}
\end{lem}

\begin{proof}
(i). We only work on $\bar{x}>b$ since the case $\bar{x}'<-b$ is studied similarly. Given $\bar{x}>b$, either $r(t)\leq \bar{x}$ for all $t>0$, or, by Theorem \ref{thm:Darcy-law}, there exists a unique $\bar{t}>t^*(b)$  such that $r(\bar{t})=\bar{x}>b$ and $r(t)>\bar{x}$ for $t>\bar{t}$. We consider only the latter case and will use the so-called Aleksandrov's reflection principle.
Set
$$
\eta(x,t):= v(x,t)- v(2\bar{x}-t),\quad (x,t)\in E_1 := \{(x,t)\mid -\infty<x\leq \bar{x},\ t\geq \bar{t}\}.
$$
Then, for some functions $c_1 = c_1(v_x)$ and $c_2 =c_2 (m,g',v,v_{xx})$, $\eta$ satisfies
$$
\left\{
 \begin{array}{ll}
 \eta_t = (m-1)v \eta_{xx} + c_1 \eta_x + c_2\eta, & (x,t)\in E_1,\\
 \eta(\bar{x},t)=0,& t\geq \bar{t},\\
 \eta(x,\bar{t}) \geq 0, & x\leq \bar{x}.
 \end{array}
 \right.
$$
Using the maximum principle (see Remark \ref{rem:MP-unbounded} if one worries about the unboundedness of $c_2$) we have
\begin{equation}\label{reflection-eta}
\eta(x,t)\geq 0 \mbox{ in } E_1,\quad \eta_x (\bar{x},t) \leq 0 \mbox{ for }t\geq \bar{t}.
\end{equation}
Furthermore, $v(\bar{x},t)>0$ for $t>\bar{t}$ by the positivity persistence, and so both $v$ and $\eta$ are classical near the line $\{x=\bar{x},\ t>\bar{t}\}$. This implies that the Hopf lemma is applicable and so
$$
\eta_x (\bar{x},t) = 2v_x (\bar{x},t)< 0 \mbox{ for }t >\bar{t}.
$$

(ii). The first inequality in \eqref{reflection-eta} also implies that
$$
l(t) \leq 2\bar{x} - r(t),\quad t>\bar{t},
$$
that is,
$$
l(t)+r(t) \leq 2\bar{x},\quad t>\bar{t}.
$$
Finally, since $\bar{x}-b>0$ can be chosen as small as possible, so does $\bar{t} -t^*(b)$, we actually have  $r(t)+l(t) \leq 2b$ for $t >0$. In a similar way one can show that $r(t)+l(t)\geq -2b$ for all $t>0$.
\end{proof}

In the above lemma we consider initial data with exactly one connected support.
Now we consider the case $u_0\in \mathfrak{X}$. More precisely, assume
\begin{equation}\label{def-Xn}
u_0\in \left\{ \psi \in C(\R)
\left|
\begin{array}{l}
 \mbox{there exist } -b\leq l_1 <r_1\leq l_2 <r_2\leq \cdots \leq l_n <r_n \leq b \mbox{ such that }\\
  \psi(x)>0 \mbox{ in } (l_j,r_j) \mbox{ for } j=1,\cdots,n, \mbox{ and } \psi(x)=0 \mbox{ otherwise}
 \end{array}
 \right.
 \right\}.
\end{equation}
Note that this set is exactly the same as $\mathfrak{X}$. Here we explicitly give the name of each boundary (see Figure 1).
As we will see below that, the purpose for choosing $u_0$ in this set is mainly for the clarity of the statement for the free boundaries. Our approach remains valid for general continuous and compactly supported initial data with infinite many free boundaries.

For each $i=1,2,\cdots,n$, denote by $l_i(t)$ (resp. $r_i(t)$) the free boundary of the solution starting at $l_i$ (resp. $r_i$).
 Given $j\in \{1,2,\cdots,n-1\}$, the boundary point $r_j$ either keeps stationary for all time (that is, $t^*(r_j)=\infty$), or it moves rightward after the finite waiting time. In the latter case, we define
$$
T_j := \sup \{T \mid \mbox{there exists }x_j \in [r_j, l_{j+1}] \mbox{ such that } u(x_j,t)=0 \mbox{ for } 0\leq t\leq T\}.
$$
In case $T_j=\infty$, the free boundaries $r_j(t)$ will never meet $l_{j+1}(t)$, or, $r_j=l_{j+1}$ and both of them have infinite waiting times.
In case $T_j<\infty$, these two free boundaries meet in some finite time and then this free boundary vanish after $T_j$.
(Note that, the following case will not happen: they meet together in finite time, and this zero of $u$ will never vanish. In fact, by Step 2 in the proof of Theorem \ref{thm:Darcy-law}, any free boundary will never stop once it begins to move.) Hence, for $t>T_j$, the intervals $(l_j(t),r_j(t))$ and $(l_{j+1}(t),r_{j+1}(t))$ merge into one: $(l_j(t), r_{j+1}(t))$. We remark that, it may merge other intervals at the same time, or at some time after $T_j$. Whatever happens, there exists a large $T$ such that, for $t\geq T$, there exist fixed number of free boundaries:
\begin{equation}\label{final-points}
l_1(t)\equiv l_{n_1}(t)< r_{n_2}(t)\leq l_{n_3}(t)< r_{n_4}(t)\leq \cdots \leq  l_{n_{2k-1}}(t)<r_{n_{2k}}(t) \equiv r_n(t),
\end{equation}
for $n_1, \cdots, n_{2k}\in \{1,2,\cdots,n\}$ satisfying
$$
n_1 \leq n_2 < n_3 \leq n_4 <\cdots <n_{2k-1} \leq n_{2k},
$$
such that $v(x,t)>0$ in $(l_{n_{2i-1}}(t), r_{n_{2i}}(t))\ (i=1,2,\cdots,k)$ and $v=0$ otherwise. Moreover, by \eqref{def-Xn} and the monotonicity of $l_i(t),r_i(t)$ we have
\begin{equation}\label{finial-left-right-bdry}
-b<r_{n_2}(t),\quad l_{n_{2k-1}}(t)<b,\quad t\geq 0,
\end{equation}

Using the previous Lemma \ref{lem:monotonocity} in each of these intervals we have the following result.

\begin{cor}\label{cor:multi-mono}
Assume {\rm (F)} and {\rm (I)}. Then there exists a large $T$ such that, for $t\geq T$, $u(\cdot,t)$ has $2k$ free boundaries as in \eqref{final-points} and \eqref{finial-left-right-bdry}. Moreover, for $1\leq j\leq k$, $u(\cdot,t)$ is strictly increasing in $[l_{2j-1}(t), l_{2j-1}(0))$, and strictly decreasing in $(r_{2j}(0), r_{2j}(t)]$.
In particular, $u(\cdot,t)$ is strictly increasing in $[l_1(t), l_1(0))$ and strictly decreasing in $(r_n(0), r_n(t)]$.
\end{cor}

\subsection{General convergence}

\medskip
\noindent
{\it Proof of Theorem \ref{thm:general-conv}}. We prove the results by using the pressure $v$ instead of $u$.

\medskip
{\it Step 1. To show the $\omega$-limit set is non-empty.}
By the locally uniform H\"{o}lder bound in \eqref{Holder-est}, there exist $C>0,\ \alpha_1\in (0,1)$, both depend on $\|v_0\|_C$ and $m$, such that, for any $M>0$ and any increasing time sequence $\{t_n\}$, there holds,
$$
\|v(x,t_n +t)\|_{C^{\alpha_1} ([-M,M]\times [-1,1])} \leq C.
$$
Hence, for any $\alpha \in (0,\alpha_1)$, there is a subsequence of $\{t_n\}$, denoted again by $\{t_n\}$, such that
$$
\|v(x,t_n +t)-w_M (x,t)\|_{C^{\alpha} ([-M,M]\times [-1,1])} \to 0 \mbox{\ \ as \ \ }n\to \infty.
$$
Using Cantor's diagonal argument, there exist a subsequence of $\{t_n\}$, denoted again by $\{t_n\}$ and a function $w(x,t)\in C^{\alpha} (\R\times \R)$ such that
$$
v(x,t_n +t)\to w (x,t) \mbox{\ \ as\ \ }n\to \infty,\quad \mbox{in the topology of } C^{\alpha}_{loc} (\R^2).
$$

In addition, if $w(x,t)>0$ in a domain $E\subset \R^2$, then for any compact subset $D\subset E$, there exists small $\rho>0$ such that
$$
w(x,t),\  v(x,t_n+t)\geq \rho>0,\quad (x,t)\in D, \ n\gg 1.
$$
Then, $v(x,t_n +t)$ is classical in $D$, and so, for any $\beta_1\in (0,1)$, $\|v(x,t_n +t)\|_{C^{2+\beta_1, 1+\beta_1 /2}(D)}\leq C$ for any large $n$ and some $C$ independent of $n$. This implies that a subsequence of $\{v(x,t_n +t)\}$ converges in $C^{2+\beta, 1+\beta/2}(D)\ (0<\beta<\beta_1)$ to the limit $w(x,t)$. Thus, $w(x,t)$ is a classical solution of (pCP) in the domain where $w(x,t)>0$, and it is a very weak solution of (pCP) for $(x,t)\in \R^2$. 
Consequently, the $\omega$-limit set of $v$ in the topology of \eqref{def-conv} is non-empty, compact, connected and invariant.

\medskip
{\it Step 2. Limits of the free boundaries.} By Corollary \ref{cor:multi-mono}, there exists a large $T$ such that $v(\cdot,t)$ has exactly $2k$ free boundaries as in \eqref{final-points} and \eqref{finial-left-right-bdry} for all $t\geq T$. For each $i=1,2,\cdots,k$, since $l_{n_{2i-1}}(t)$ is decreasing and $r_{n_{2i}}(t)$ is increasing, the following limits exist:
\begin{equation}\label{limits-bdry}
l^\infty := \lim\limits_{t\to\infty} l_{n_1}(t),\quad
l^\infty_{n_{2i-1}} := \lim\limits_{t\to\infty} l_{n_{2i-1}}(t),\quad
r^\infty_{n_{2i}} := \lim\limits_{t\to\infty} r_{n_{2i}}(t),\quad
r^\infty := \lim\limits_{t\to\infty} r_{n_{2k}}(t).
\end{equation}
We have the following possibilities:

{\it Case 1. $k\geq 2$.} In this case all the free boundaries lie in $[-b,b]$ except for the left most one $l_{n_1}(t)\equiv l_1 (t)$ and the right most one $r_{n_{2k}}(t)\equiv r_n(t)$.
Together with Lemma \ref{lem:monotonocity} (ii) we have
$$
l_{1}(t) \geq -2b -r_{n_2}(t) \geq -3b, \quad r_{n}(t) \leq 2b - l_{n_{2k-1}}(t) \leq 3b.
$$
Consequently, all the limits in \eqref{limits-bdry} are in $[-3b,3b]$.

{\it Case 2. $k=1$ and $-l^\infty, r^\infty <\infty$.} In this case $v(\cdot,t)$ has exactly one connected compact support $[l_{n_1}(t),r_{n_2}(t)]=[l_1(t),r_n(t)]$ which tends to the bounded interval $[l^\infty, r^\infty]$ as $t\to \infty$.
Note that, even in this case, the support of the $\omega$-limit $w(x,t_1)$ may have several separated intervals as in the ground state solution in \eqref{finite-ground-state-0}, since the convergence $v\to w$ is taken in $C^\alpha_{loc}(\R)$ topology, the positivity of $v$ does not guarantee the positivity of its limit $w$.

{\it Case 3. $k=1$ and $-l^\infty,r^\infty=\infty$.} In this case $v(\cdot,t)$ has exactly one connected compact support $[l_{n_1}(t),r_{n_2}(t)]=[l_1(t),r_n(t)]$ which tends to $\R$ as $t\to \infty$.

\medskip
{\it Step 3. Quasi-convergence and Type II ground state solutions in the case $-l^\infty, r^\infty<\infty$.}
This includes Cases 1 and 2 in Step 2. We first use the Lyapunov functional to prove the quasi-convergence result, that is, any $\omega$-limit is a stationary solution of (pCP).
 Define
\begin{equation}\label{def-Lyapunov}
E[v(\cdot,t)] := \int_{l_1 (t)}^{r_n (t)} \left[\frac{m-1}{2} v^{\frac{2}{m-1}} v_x^2 - G(v) \right]dx, \quad t\geq T,
\end{equation}
with
$$
G(v):= \int_0^v g(r) r^{\frac{3-m}{m-1}} dr,
$$
which is convergent by $|g(v)|\leq K(m-1)|v|$.
When $v$ is the solution of (pCP), a direct calculation shows that
$$
\frac{d}{dt} E[v(\cdot,t)] = -\int_{l_1 (t)}^{r_n (t)} v^{\frac{3-m}{m-1}} v_t^2 dx \leq 0.
$$
So, $E[v(\cdot,t)]$ is a Lyapunov functional. It is bounded from below by our assumption $l^\infty \leq l_1(t)<r_n(t)\leq r^\infty$. Hence, we can use the standard argument to show that any $\omega$-limit of $v(\cdot,t)$ is a stationary solution of the equation pRPME.

We now explain that any $\omega$-limit is also a stationary solution of the Cauchy problem (pCP). Assume, $w(x,t)$ is the limit function obtained in Step 1. For each $\tau_1\in \R$, we claim that the solution $v(x,t;w(\cdot,\tau_1))$ of (pCP) with $v_0 (x)= w(x,\tau_1)$ is a stationary one: $v(x,t;w(\cdot,\tau_1))\equiv w(x,\tau_1)$. Since $w(x,\tau_1)$ has boundaries $l^\infty, l^\infty_{n_{2i-1}}, r^\infty_{n_{2i}}, r^\infty$ as in \eqref{limits-bdry}, by Lemma \ref{lem:ss-problem} we only need to show that the waiting time of each of these boundaries is not zero. Without loss of generality, we only prove $t^*(l^\infty)>0$. By contradiction, assume the left boundary $l_w(t)$ of $v(x,t;w(\cdot,\tau_1))$ satisfies $l_w(\tau_2)<l^\infty$ for some $\tau_2>0$. Since $w(x,\tau_1+\tau_2)$ is also an $\omega$-limit of $v$:
$$
v(x,t_n + \tau_1+\tau_2) \to w(x,\tau_1 +\tau_2) \equiv v(x,\tau_2; w(\cdot,\tau_1)),\quad n\to \infty.
$$
(Note that $w(x,\tau_1+t)$ solves the (pCP) with initial data $w(x,\tau_1)$ at $t=0$. It is nothing but $v(x,t;w(\cdot,\tau_1))$ by the uniqueness of the very weak solution of (pCP).)
Thus, we have
$$
l_1 (t_n +\tau_1+\tau_2) \to l_w(\tau_2) <l^\infty , \quad n\to \infty,
$$
contradicts the definition of $l^\infty$. Therefore, $w(x,\tau_1)$ is a stationary solution of (pCP).

We continue to specify the possible shapes of the $\omega$-limits of $v$.
Denote by $w(x)$ one of them. By the monotonicity in Lemma \ref{lem:monotonocity} and Corollary \ref{cor:multi-mono}, we see that either $w(x)\equiv 0$, or, it has maximums in $[-b,b]$. In the latter case, we claim that $w$ does not have positive minimums. Otherwise, by the equation pRPME, $w$ is symmetric with respect to each maximum point and each minimum point as long as it keeps positive, and so $w$ must be a positive periodic function, which is impossible by the monotonicity outside of $[-b,b]$. Therefore, the only possible choice for $w$ is that it has only maximum points in $[-b,b]$ but has no positive minimum points, and so
it is a Type II ground state solution of (pCP). More precisely, for some
$$
-b\leq z_1< z_2<\cdots< z_k\leq b\mbox{\ \ and\ \ } L>0,
$$
with $z_i+2 L \leq z_{i+1}\ (i=1,2,\cdots, k-1)$, there holds
\begin{equation}\label{finite-GS-V}
w(x) = \mathcal{V}(x) := V(x-z_1)+ V(x-z_2)+\cdots + V (x-z_k), \quad x\in \R,
\end{equation}
where $V(x)$ is the unique nonnegative compactly supported stationary solution of (pCP) as in Lemma \ref{lem:unique-GS}.

\medskip
{\it Step 4. Convergence in case $-l^\infty, r^\infty<\infty$.} In the previous step we know that any $\omega$-limit $w(x)$ of $v(x,t)$ is either $0$ or a Type II ground state solution as $\mathcal{V}$. In the former case we actually have the convergence result:
$$
v(\cdot,t)\to 0\mbox{\ \ as\ \ }t\to \infty,
$$
since the $\omega$-limit set of $v$ is connected. In the latter case, we now improve the conclusion in Step 3 to a convergence one: $v(x,t)\to \mathcal{V}(x)\ (t\to \infty)$, that is, the shifts $\{z_1,z_2,\cdots,z_k\}$ in \eqref{finite-GS-V} is unique.

We first prove the following claim.

\medskip
\noindent
{\bf Claim 1}. There exists $T_1 >0$ such that $v(\cdot,t)$ has exactly $k$ maximums larger than $\Theta$ for $t\geq T_1$, where $\Theta$ is the largest positive zero of $g$ in $(0,\Theta_1)$ as in \eqref{def-tilde-theta}.

\medskip

First, we show that $v(\cdot,t)-\Theta$ has exactly $2k$ non-degenerate zeros for large time.  Using the standard zero number argument for $v(\cdot,t)-\Theta$ we know that, for some $T_2>0$, $v(\cdot,t)-\Theta$
has fixed even number of non-degenerate zeros when $t>T_2$. Since $\mathcal{V}$ defined by \eqref{finite-GS-V} is one  $\omega$-limit of $v$, we see that the fixed number must be $2k$. Denote them by
$$
y_1(t)<y_2(t)<\cdots <y_{2k-1}(t)<y_{2k}(t).
$$

Next, we show that $v$ has exactly one maximum $z_i(t)$ in $(y_{2i-1}(t),y_{2i}(t))$ for all large $t$ and each $i=1,2,\cdots, k$. We only prove the conclusion in $[y_1(t), y_2(t)]$. In this interval $\zeta(x,t):= v_x(x,t)$ satisfies
\begin{equation}\label{zeta-p}
\left\{
 \begin{array}{ll}
 \zeta_t = (m-1)v\zeta_{xx} + (m+1)\zeta \zeta_x + g'(v) \zeta, & y_1(t)<x<y_2(t), t>T_2,\\
 \zeta(y_1(t),t)>0>\zeta(y_2(t),t), & t>T_2.
 \end{array}
 \right.
\end{equation}
Since $v(x,t)>\Theta$ for $y_1(t)<x<y_2(t), t>T_2$, the equation is a uniform parabolic one, and so we can use the classical zero number argument to conclude that, for some $T_3 >T_2$, $\zeta(\cdot,t)$ has fixed number of zeros in $(y_1(t),y_2(t))$.
By our assumption, for some increasing time sequence $\{t_n\}$, we have
$$
[y_1(t_n),y_2(t_n)]\subset  (-L+z_1, L+z_1),
$$
and $v(x,t_n)\to \mathcal{V}(x)\ (n\to \infty)$ in the topology  $C^{2,1}_{loc}((-L+z_1,L+z_1))$.
This implies that $v(x,t_n)$ has exactly one maximum point $z_1(t_n)\in (y_1(t_n),y_2(t_n))$ for large $n$. 
This prove  Claim 1.

\medskip

By $v(x,t_n)\to \mathcal{V}(x)$ again, for some large $T_4 >T_3$, $v(x,T_4)> \Theta$ near $z_1$. For simplicity, when $t\geq T_4$, denote by $l(t),\ r(t)$ the two free boundaries of the support of $v(\cdot,t)$ containing $z_1$.
They are nothing but $l_1(t)\equiv l_{n_1}(t)$ and $r_{n_2}(t)$ in \eqref{final-points}. Then $-l(t)$ and $r(t)$ are non-decreasing continuous functions.

We now use Claim 1 to prove the convergence result, that is, the shift $z_1$ is unique. By contradiction we assume that $v$ has another $\omega$-limit point
$$
\check{\mathcal{V}}(x):= V (x-\check{z}_1)+V(x-\check{z}_2)+\cdots + V(x-\check{z}_k),
$$
and, without loss of generality, assume $z_1 < \check{z}_1$. Choose $x_0,x_1,x_2\in [z_1, \check{z}_1]$ such that
$$
x_0\not= \frac{l(T_4)+ r(T_4)}{2},\quad x_0-x_1 =x_2-x_0 >0,
$$
that is, $x_0$ is the center of $[x_1, x_2]\subset [z_1, \check{z}_1]$ but not that of $[l(T_4),r(T_4)]$. Then,
 by the connectedness of the $\omega$-limit set $\omega(v)$, each  $V(x-x_i)\ (i=0,1,2)$ is a part of an $\omega$-limit of $v$. Define
$$
\tilde{l}(t):= 2x_0 - r (t),\quad \tilde{r} (t):= 2x_0 - l(t),\quad
\tilde{v}(x,t):= v(2 x_0 -x, t) \mbox{\ \ for\ \ } x\in \R,\ t\geq T_4.
$$
Then $\tilde{l} (t)$ and $\tilde{r} (t)$ are free boundaries of $\tilde{v}(\cdot,t)$, $\tilde{v}(x,t)>0$ for $x\in (\tilde{l}(t),\tilde{r}(t)),\ t\geq T_4$, and
$$
\tilde{v}_t = (m-1)\tilde{v} \tilde{v}_{xx} + \tilde{v}_x^2 + g(\tilde{v}),\quad x\in \R, t\geq T_4.
$$
It is easily seen that $\tilde{l}(t)< r(t),\ \tilde{r}(t)> l(t)$, and so $v(\cdot,t)$ and $\tilde{v}(\cdot,t)$ have non-empty common domain $[\hat{l}(t),\hat{r}(t)]$ with
$$
\hat{l}(t):= \max\{l(t),\tilde{l}(t)\}\ < x_0<\ \hat{r}(t):= \min\{r(t),\tilde{r}(t)\},\quad t\geq T_4.
$$
By the choice of $x_0$ we see that
$$
l (T_4) \not= \tilde{l}(T_4) =2 x_0 -r(T_4),\quad r(T_4)\not= \tilde{r} (T_4)=2 x_0 - l (T_4).
$$
Hence, we can use Lemma \ref{lem:diminish-intersection} which is valid in open intervals to conclude that the number of the zeros of $\eta(\cdot,t):= v(\cdot,t)-\tilde{v}(\cdot,t)$ in the open interval $J_0(t):= (\hat{l} (t),\hat{r} (t))$ is finite, decreasing in $t>T_4$, and strictly decreasing when $\eta(\cdot,t)$ has degenerate zeros in $J_0(t)$. Consequently, for some large time $T_5>T_4$,
\begin{equation}\label{no-degenerate-zero}
x_0 \mbox{ is not a degenerate zero of } \eta(\cdot,t):= v(\cdot,t)-\tilde{v}(\cdot,t) \mbox{ for } t>T_5.
\end{equation}
On the other hand, by the previous assumption, there exist two time sequences $\{s^{(1)}_n\}$ and $\{s^{(2)}_n\}$ with
\begin{equation}\label{time-sequences}
\cdots < s^{(1)}_n < s^{(2)}_n < s^{(1)}_{n+1} < s^{(2)}_{n+1}<\cdots, \quad n=1,2,\cdots
\end{equation}
such that, for $i=1,2$,
$$
\|v(x,s^{(i)}_n )- V(x- x_i)\|_{L^\infty([-L+x_i, L+ x_i])}\to 0 \mbox{\ \ as\ \ } n\to \infty.
$$
Therefore, for all large $n$ we have
$$
z_1 (s^{(i)}_n)\approx x_i,\quad i=1,2.
$$
By the continuity of $z_1(t)$, there exists $s_n\in (s^{(1)}_n, s^{(2)}_n)$ such that $z_1 (s_n)= x_0$. Thus, for large $n$ we have
$$
\eta(x_0,s_n) = v(x_0,s_n) -\tilde{v}(x_0,s_n) =0,\quad
 \eta_x(x_0,s_n) = v_x (x_0,s_n)-\tilde{v}_x ( x_0,s_n)= 0,
$$
contradicting \eqref{no-degenerate-zero}.
This proves the uniqueness of $z_1$, and so the convergence in this step is proved.

\medskip
{\it Step 5. To show that any $\omega$-limit is a stationary solution of pRPME in case $-l^\infty, r^\infty=\infty$.} This is Case 3 in Step 2. For large $t$, say $t\geq T_6$, $v(\cdot,t)$ has exactly one connected support. Denote, for simplicity, its left and right free boundaries by $l(t)$ and $r(t)$, respectively.  Due to the unboundedness of $l(t)$ and $r(t)$, the Lyapunov functional as in \eqref{def-Lyapunov} is no longer a suitable tool to derive  the quasi-convergence result. Instead, we will use the zero number argument to continue our argument.

Using the limit function $w(x,t)$ in Step 1, we want to show that, for each $t\in \R$, $w(x,t)$ is a stationary solution of (pCP). It is sufficient to prove that $w(x,0)$ is so. When $w(x,0)\equiv 0$ in $x\in \R$, there is nothing left to proof. So, we assume without loss of generality that $w(0,0)>0$. 
We now construct directly a positive stationary solution as the following:
\begin{equation}\label{V-p}
(m-1)\bar{v} \bar{v}_{xx} + \bar{v}^2_x + g(\bar{v})=0,\quad \bar{v}(0)=w(0,0)>0,\ \ \bar{v}'(0)=w_x(0,0).
\end{equation}
As long as $\bar{v}>0$ the equation is a non-degenerate second order ordinary differential equation, and so the existence and uniqueness of the solution of \eqref{V-p} follow from the standard theory.
Noticing from the equation we see that $\bar{v}$ is symmetric with respect to
any positive critical point. Hence, we have the following possibilities:
\begin{enumerate}[(a).]
\item $\bar{v}(x)>0$ for all $x\in \R$;

\item $\bar{v}(x)>0$ and $\bar{v}'(x)>0$ in $ (l_0, \infty)$ for some $l_0\in (-\infty, 0)$, and $\bar{v}(l_0)=0$;

\item $\bar{v}(x)>0$ and $\bar{v}'(x)<0$ in $(-\infty, r_0)$ for some $r_0\in (0,\infty)$, and $\bar{v}(r_0)=0$;

\item $\bar{v}(x)>0$ for $x\in (l_0, r_0)$ for some $l_0<0<r_0$, $\bar{v}(l_0)=\bar{v}(r_0)=0$ and it has a unique maximum point $z_0:= \frac{l_0+r_0}{2}$.
\end{enumerate}
In this step we will prove $w(x,0)\equiv \bar{v}(x)$ in ${\rm spt}[\bar{v}]$.

(a). Set
$$
\eta(x,t) := v(x,t)- \bar{v}(x),\quad x\in \R,\ t\geq T_6.
$$
In the current case (a) we have
\begin{equation}\label{eta-is-positive-bdry}
\eta(l(t),t) <0,\quad \eta(r(t) ,t)<0,\quad t\geq T_6,
\end{equation}
and
\begin{equation}\label{eta-eq}
\eta_t = a(x,t) \eta_{xx} + b(x,t) \eta_x + c(x,t) \eta, \quad x\in J_0(t):=(l(t),r(t)),\ t\geq T_6,
\end{equation}
with
$$
a(x,t) := (m-1)v(x, t) ,\quad b(x,t):= v_x(x, t) + \bar{v}_x (x),
$$
and
$$
c(x,t):= (m-1) \bar{v}_{xx}(x)+ \left\{
 \begin{array}{ll}
  \frac{g(v(x, t))-g( \bar{v}(x))}{v(x,t)- \bar{v}(x)}, & v\not= \bar{v},\\
  0 , & v=\bar{v}.
  \end{array}
  \right.
$$

We will use Proposition \ref{prop:zero} to study the zero number of $\eta(\cdot,t)$ in any bounded time interval $(t_1, t_2)\subset (T_6,\infty)$. For this purpose, we should require that $a,b,c$ satisfy the assumption \eqref{smoothy}. This is obviously not true in the domain
$$
E_1 := \{(x,t)\mid x\in J_0(t),\ t_1 < t < t_2\},
$$
since $a^{-1}=[(m-1)v]^{-1}$ is not bounded near the boundaries of $E_1$. This difficulty will be solved below by cutting $J_0(t)$ a little bit near its boundaries. Since $l(t)$ and $r(t)$ are continuous and bounded functions in $[t_1, t_2]$, there exists $\rho_1 >0$ small such that
$$
\eta({l}(t),t) = -\bar{v} (l(t)) \leq -\rho_1,\quad
\eta({r}(t),t) = - \bar{v}(r(t)) \leq -\rho_1,
\quad t\in [t_1, t_2].
$$
By the uniform continuity of $v$, there exists $\varepsilon_1>0$ small such that
$$
\eta(x,t) <0,\quad l(t)\leq x \leq l(t)+ \varepsilon_1 \mbox{ or } r(t)+ \varepsilon_1 \leq x\leq r(t),\ t\in [t_1, t_2]
$$
and, for some small $\rho_2>0$, there holds
$$
v(x,t)\geq \rho_2 \mbox{ in } E_2:= \{(x,t) \mid l(t)+\varepsilon_1 \leq x\leq r(t)-\varepsilon_1,\ t_1 \leq t\leq t_2\}.
$$
Now we see that $a,b,c$ satisfy the conditions in \eqref{smoothy} in the domain $E_2$. Hence the classical zero number diminishing properties as in Proposition \ref{prop:zero} hold.
Since $\varepsilon_1>0,  t_1-T_6>0$ can be chosen as small as possible and $t_2$ can be chosen as large as possible, we actually have the same zero number diminishing properties in $E_1$ with $(t_1,t_2)=(T_6,\infty)$.
As a consequence, $\eta$ has only degenerate zeros (not on the boundaries $l(t)$ and $r(t)$) in finite time, that is, there exists $T_7 >T_6$ such that $\eta(\cdot,t)$ not longer has degenerate zeros for $t\geq T_7$. Thus, for any $t\in [-1,1]$ and all large $n$, $\eta(x,t_n +t)=v(x, t_n +t) -\bar{v}(x)$ has no degenerate zeros in $J_0(t_n +t)$. This implies that (see, for example
\cite[Lemma 2.6]{DM}) its limit $w(x,t)- \bar{v}(x)$ either is $0$ identically, or has no degenerate zeros in the spatial domain where $w(\cdot,t)>0$. The latter, however, contradicts the initial conditions in \eqref{V-p}. This proves that $w(x,0)\equiv \bar{v}(x)$ in ${\rm spt}[\bar{v}]$.

(b). Now we assume $\bar{v}$ is a stationary solution as in case (b). Since $l(t)\to -\infty\ (t\to \infty)$, there exists $T_8>T_6$ such that $l(t) <l_0$ for $t\geq T_8$. Hence, we should consider $\eta$ in
$$
E_3 := \{(x,t) \mid l_0 < x < r (t),\ t\geq T_8\}.
$$
On its left boundary $\{x=l_0,\ t>T_8\}$, we have $\eta>0$. By continuity, for any bounded time interval $(t_1, t_2)\subset (T_8 ,\infty)$,  there exists $\varepsilon_2>0$ and $\rho_3 = \rho_3 (t_1,t_2, l_0) >0$ such that
$$
\eta(x, t) > \rho_3,\quad l_0 \leq x\leq l_0+ \varepsilon_2,\ t\in (t_1, t_2).
$$
Furthermore, the term $\bar{v}_{xx}(l_0 +\varepsilon_2)$ in the coefficient $c$ of \eqref{eta-eq} is also bounded, though $\bar{v}_{xx}(l_0)$ maybe not. Thus, we can use the zero number diminishing properties in
the domain $\{(x,t)\mid l_0 +\varepsilon_2 < x < r(t), t_1<t<t_2\}$. By the arbitrariness of $\varepsilon_2$ and $(t_1,t_2)$, we see that the zero number diminishing properties hold actually in the domain $E_3$. The rest proof is similar as that in (a). This proves the quasi-convergence in case (b).

The proof for $w(x,0)\equiv \bar{v}\ (x\in {\rm spt}[\bar{v}])$ in cases (c) and (d) are similar as that in case (b).

\medskip
{\it Step 6. To show that any $\omega$-limit of $v$ is a positive stationary solution of {\rm (pCP)} in the case $-l^\infty, r^\infty=\infty$.} From the previous step we see that, in the case $-l^\infty, r^\infty=\infty$, any $\omega$-limit of $v$ must be a stationary solution of the equation pRPME as in (a)-(d). By the monotonicity in Corollary \ref{cor:multi-mono}, the cases (b) and (c) are impossible. We now prove that the case (d) is also impossible.

1). By the monotonicity of $v$ and $w$ outside of $[-b,b]$, we see that the maximum point $z_0:= \frac{l_0 +r_0}{2}$ of $\bar{v}(x)\equiv w(x,0)|_{{\rm spt}[\bar{v}]}$ lies in $[-b,b]$, and so, with  $\ell := \frac{r_0-l_0}{2}$, we have
$$
-b -\ell \leq l_0 <z_0 <r_0 \leq b +\ell.
$$

2). We show that $v(-3b-2\ell,t)<\bar{v}(z_0)$ for all $t>0$. Assume by contradiction that the reversed inequality holds at some time $t_1>0$. Then by the monotonicity of $v$ in $(l(t_1),-b)$ we have
$$
v(x,t_1) \geq \bar{v}(x+z_0 + 3b+\ell),\quad x\in [-3b-2\ell, -3b].
$$
Since $\bar{v}(x+z_0 + 3b+\ell)$ is a time-independent weak subsolution  of the problem (pCP), by Lemma \ref{lem:ss-problem} we have
$$
v(x,t) \geq \bar{v}(x+z_0 + 3b+\ell),\quad x\in \R,\ t\geq t_1,
$$
and so
$$
w(x,0) \geq \bar{v}(x+z_0 + 3b+\ell),\quad x\in \R.
$$
By the monotonicity of $w(x,0)$ in $[-3b-2\ell, -b]$ we have
\begin{equation}\label{w(x,0)-large}
w(x,0) \geq \bar{v}(z_0),\quad x\in [-3b-\ell, -b].
\end{equation}
On the other hand, $w(x,0)=0$ at $l_0\in [-b-\ell, b+\ell]$. So $w(x,0)$ has at least one maximum point $x_1\in [-b,l_0)$. Without loss of generality, assume $x_1$ is the smallest one of such points. Then, as a stationary solution of pRPME, $w(x,0)$ is symmetric with respect to $x_1$. This implies by \eqref{w(x,0)-large} that
$$
w(x,0) \geq \bar{v}(z_0),\quad x\in [-3b-\ell , 2x_1+3b+\ell ]\supset [-b-\ell, b+\ell],
$$
contradicts 1).

3). If $\bar{v}(x) \equiv w(x,0)|_{{\rm spt}[\bar{v}]}$ is not only a stationary solution of the equation pRPME, but also a stationary solution of (pCP), then we can derive a contradiction as the following. Choose $t'$ large such that $l(t')<-3b-4\ell$, and choose $L'$ such that
$$
-L' + l_0 <l(t') < -L'+r_0
$$
and $\bar{v}(x+L')$ has exactly one intersection point with $v(x,t')$. Then using the intersection number properties Lemma \ref{lem:diminish-intersection} and using the fact in 2) we conclude that $l(t)$ will never pass across the left boundary
$-L'+l_0$ of $\bar{v}(x+L')$. This contradicts the assumption $l^\infty=-\infty$.

4). We then consider the case where $\bar{v}(x) \equiv w(x,0)|_{{\rm spt}[\bar{v}]}$ is just a stationary solution of the equation pRPME, but not a stationary one of (pCP). Denote by $v(x,t; \bar{v})$ the solution of (pCP) with initial data $\bar{v}$, and denote its free boundaries by $\bar{l}(t)$ and $\bar{r}(t)$. By Lemma \ref{lem:ss-problem}, the waiting times $t^*(l_0)=t^*(r_0)=0$. Hence, $-\bar{l}(t)$ and $\bar{r}(t)$ are strictly increasing in $t>0$, and
$$
v(x,t;\bar{v})|_{[\bar{l}(t),\bar{r}(t)]} \equiv w(x,t)|_{[\bar{l}(t),\bar{r}(t)]},\quad 0<t\ll 1.
$$
Since $\bar{v}$ is a (stationary) subsolution  of (pCP), by Lemma \ref{lem:ss-problem} we have $v(x,t;\bar{v})>\bar{v}(x)$ on the support of $\bar{v}$, and so $v(x,t;\bar{v})$ is strictly increasing in $t$. Therefore, for any small $\tau>0$, we have
$$
\bar{l}(2\tau)<\bar{l}(\tau)<\bar{l}(0)=l_0 < r_0 = \bar{r}(0) <\bar{r}(\tau)<\bar{r}(2\tau),
$$
and
$$
w(x,0) < w(x,\tau) \mbox{ for } x\in [l_0, r_0],
\qquad w(x,\tau) <w(x,2\tau) \mbox{ for } x\in [\bar{l}(\tau),\bar{r}(\tau)].
$$
By Step 1, $v(x,t_n+ 2\tau)\to w(x,2\tau)$ as $n\to \infty$, and so for some large $n_0$ we have
$$
v(x,t_{n_0} +2\tau) \geq w(x,\tau),\quad x\in [\bar{l}(\tau),\bar{r}(\tau)].
$$
From the previous step we know that $w(x,\tau)$ is a stationary solution of pRPME in $[\bar{l}(\tau),\bar{r}(\tau)]$. Hence, $w(x,\tau)|_{[\bar{l}(\tau),\bar{r}(\tau)]}$ is a time-independent weak subsolution of (pCP). By comparison we have
$$
v(x,t)\geq w(x,\tau),\quad x\in [\bar{l}(\tau),\bar{r}(\tau)],\ t>t_{n_0} +2\tau.
$$
This contradicts the fact that $w(x,0)$ is an $\omega$-limit of $v(x,t)$.

\medskip
Consequently, the case (d) is also impossible. In other words, the only possible case is (a) in Step 5. We then have the following claim.

\medskip
\noindent
{\bf Claim 2.} When $-l^\infty, r^\infty=\infty$, any $\omega$-limit $w(x)$ of $v(x,t;u_0)$ is a stationary solution of pRPME which is positive in $\R$, and so is a stationary solution of the Cauchy problem (pCP).
\medskip

\medskip
{\it Step 7. Convergence in case $-l^\infty, r^\infty=\infty$.}
By Claim 2, the convergence $v(x,t_n;u_0)\to w(x)$ in Step 1 holds actually in $C^2_{loc}(\R)$ topology. Therefore, the solution $v(\cdot,t)$ is a classical one in $(l_{n_1}(t),r_{2n}(t))$.
Then, using the classical zero number diminishing properties and using similar arguments as proving \cite[Theorem 1.1]{DL}, \cite[Theorem 1.1]{DLZ} (see also \cite[Theorem 1.1]{DM}) we conclude that $v(\cdot,t)$ converges as $t\to \infty$ to either a nonnegative zero of $g$, or
$$
v(\cdot,t)\to V_0 (\cdot-z_0)>0,
$$
where $z_0\in [-b,b]$ and $V_0 $ is an evenly decreasing positive stationary solution of pRPME, that is,
$$
U_0 (x) := \left( \frac{m-1}{m}V_0 (x) \right)^{\frac{1}{m-1}}>0, \quad x\in \R
$$
is a Type I ground state solution of (CP) as given in Section 1.

This completes the proof of Theorem \ref{thm:general-conv}.
\hfill $\Box$

\section{Monostable RPME}
In this section we study the hair-trigger effect for the solutions of (CP) or (pCP) with monostable type of reaction.

\subsection{Stationary solutions}\label{subsec:-mono-ss}

A stationary solution $u=q(x)$ of (CP) satisfies
\begin{equation}\label{SS-eq}
(q^m)'' + f(q) =0,\quad x\in J\subset \R.
\end{equation}
We will use the phase portrait to present all the nonnegative stationary solutions that we are interested in. The equation \eqref{SS-eq} can be rewritten as a pair of first order equations
\begin{equation}\label{SS-sys}
\left\{
 \begin{array}{l}
 p= (q^m)' = m q^{m-1} q', \\
 p' = - f(q),
 \end{array}
 \right.
\end{equation}
or the first order equation
\begin{equation}\label{dp-dq}
pdp = - mq^{m-1} f(q) dq.
\end{equation}
Its first integral is
\begin{equation}\label{first-integral}
p^2 = C -2m \int_0^q r^{m-1}f(r) dr,
\end{equation}
for suitable choice of $C$.  We consider only the trajectories in the region $\{q\geq 0\}$ which correspond to nonnegative solutions.

{\bf Case A}. The points $(0,0)$ and $(1,0)$ are singular ones on the $(q,p)$-phase plane for the system \eqref{SS-sys}. Hence, they correspond to constant stationary solutions:
$$
U_q (x) \equiv q, \quad x\in \R,
$$
for $q=0,\ 1$. We will show below that $U_1$ is the only possible choice for the $\omega$-limit of the solution $u$ of (CP).

{\bf Case B}. When $C = C_0 :=  2m\int_0^1 r^{m-1}f(r) dr $, the trajectories on the phase plane are two curves, whose functions are
$$
p = P_0(q):= \left( 2m \int_q^1 r^{m-1}f(r) dr\right)^{1/2} \mbox{\ \ and\ \  }p=-P_0(q),\quad 0\leq q<1.
$$
They connect regular points $Q_\pm := (0, \pm \sqrt{C_0})$ to the singular point $(1,0)$.
Denote the corresponding stationary solution by $U^\pm (x)$. Then
$$
U^+(l_0)=0,\quad U^+(x)>0 \mbox{ for }x>l_0, \mbox{\ \ and\ \ } U^+(x)\nearrow 1\mbox{ as }x\to +\infty;
$$
$$
U^-(r_0)=0,\quad U^-(x)>0 \mbox{ for }x<r_0, \mbox{\ \ and\ \ } U^-(x)\nearrow 1\mbox{ as }x\to -\infty,
$$
where $l_0$ is the left boundary of $U^+$ and $r_0$ is the right boundary of $U^-$.

{\bf Case C}. When $C>C_0$, the trajectories start at $(0,\pm \sqrt{C})$, go beyond the line $q=1$ and so the the corresponding stationary solutions are unbounded ones, which will not be used below.

{\bf Case D}. When $0<C<C_0$, the trajectory $p^2 +2m \int_0^q r^{m-1}f(r) dr=C$ is a connected
curve, symmetric with respect to the $q$-axis, and connecting the points
$Q_+ := (0, \sqrt{C} )$, $Q_0:= (q_0 ,0)$ and $Q_-:= (0, - \sqrt{C})$ for a unique $q_0\in (0,1)$ which is defined as the following
$$
C= 2m \int_0^{q_0} r^{m-1}f(r) dr.
$$
This trajectory corresponds to a stationary solution $U_{q_0}(x)$, under the normalized condition $U'_{q_0}(0)=0$, it satisfies
\begin{equation}\label{compact-ss}
U_{q_0}(0)=q_0,\quad U_{q_0}(\pm L_0)=0 \mbox{ for some }L_0>0,\quad U'_{q_0}(x)<0,\ \ U_{q_0}(x)=U_{q_0}(-x) \mbox{ for }x\in (0,L_0).
\end{equation}

We now study the order of the solution $U$ in Cases B (in this case, it is $U^+$), C or D near its left boundary. Denote the left boundary by $l_0$. By \eqref{first-integral} we have
$$
(U^m)' =  \left( C - 2m \int_0^U r^{m-1}f(r) dr\right)^{1/2} \sim \sqrt{C},\quad x\to l_0 +0.
$$
So,
\begin{equation}\label{combustion-q-V-0}
U^m (x) \sim \sqrt{C} (x-l_0),\quad x \to l_0 + 0.
\end{equation}
If we use $V := \frac{m}{m-1}U^{m-1}$ to denote the corresponding stationary solution of pRPME we have
\begin{equation}\label{rate-V-0}
V(x)= \frac{m}{m-1} C^{\frac{m-1}{2m}} (x-l_0)^{\frac{m-1}{m}},\quad x \to l_0+0.
\end{equation}

\begin{rem}\label{rem:ss-eq-p}
\rm We point out that the solutions $U^\pm $ in Case B and $U_{q_0}$ in Case D are just stationary solutions of RPME but not stationary solutions of the Cauchy problem (CP). In fact, the expression in \eqref{rate-V-0} and Lemma \ref{lem:no-waiting-time} imply that, when we choose $V(x)$ as an initial data of the problem (pCP), the waiting time of its boundaries are $0$. In particular, in the bistable RPME, $U_{q_0}$ is not a Type II ground state solution.
\end{rem}

We continue to specify the relationship between the support width $2L_0$ of $U_{q_0}$ in Case D and its height $q_0$.

\begin{lem}\label{lem:width-height}
When $f$ is a monostable reaction term, there holds $L_0\to 0$ as $q_0\to 0+0$. In addition, if $f$ satisfies $f(\rho u)>\rho^m f(u)$ for $\rho, u\in (0,1)$, then $L_0$ is strictly increasing in $q_0\in (0,1)$.
\end{lem}

\begin{proof}
By $f'(0)>0$, there exists $\delta>0$ small such that
\begin{equation}\label{u0-is-not-small-near0}
f(u)\geq \frac{\pi^2}{4\varepsilon^2} u^m,\quad u\in [0,\delta).
\end{equation}
For any $q_0 \in (0,\delta)$, by \eqref{first-integral} we have
$$
p^2 = 2m\int_q^{q_0} r^{m-1} f(r) dr \geq \frac{\pi^2}{4 \varepsilon^2} (q_0^{2m} -q^{2m}).
$$
Therefore, on $x\in [-L_0, 0]$ we have
$$
p= (q^m)' \geq  \frac{\pi}{2\varepsilon} q_0^{m} \Big[ 1- \Big(\frac{q}{q_0}\Big)^{2m} \Big]^{1/2} .
$$
Since $q (-L_0)=0, q (0)=q_0$, by integrating the above inequality over $[-L_0,0]$ we have
$$
{\frac{\pi}{2\varepsilon}} L_0 \leq \int_0^1 \frac{dr}{\sqrt{1-r^2}}= \frac{\pi}{2},
$$
that is, $L_0\leq \varepsilon$.

Denote $\tilde{v}:= u^m$ and $\tilde{f}(s):= f(s^{\frac{1}{m}})$, then the stationary problem
$$
(u^m)'' +f(u)=0,\ \ u(x)>0 \mbox{ in } (-L_0,L_0),\qquad u(\pm L_0)=0,
$$
is equivalent to
$$
\tilde{v}'' + \tilde{f}(\tilde{v})=0,\ \ \tilde{v}(x)>0 \mbox{ in } (-L_0,L_0),\qquad \tilde{v}(\pm L_0)=0.
$$
By our assumption $\tilde{f}$ is a Fisher-KPP type of nonlinearity: $\tilde{f}(s)/s$ is decreasing in $s\in (0,1)$. Hence $\tilde{v}(0)$ is strictly increasing in $L_0$. This proves the lemma.
\end{proof}

\subsection{Hair-trigger effect}
Now we can use the general convergence theorem to prove the hair-trigger effect easily.

\medskip
\noindent
{\it Proof of Theorem \ref{thm:mono}.}
By Lemma \ref{lem:width-height}, there exist sufficiently small $q_0>0$, small $L_0>0$ and some $x_0\in \R$ such that
\begin{equation}\label{u0-is-not-small}
u_0(x) \geq U_{q_0}(x-x_0),\quad x\in [-L_0+x_0, L_0+x_0].
\end{equation}
Since $U_{q_0}(x-x_0)$ is a time-independent very weak subsolution of (CP), by the comparison result in Lemma \ref{lem:ss-problem} (iv) we have
$$
u(x,t) \geq U_{q_0}(x-x_0),\quad x\in \R,\ t>0.
$$
In the monostable PME, the only stationary solution of (CP) larger than $U_{q_0}$ is $1$. So, by using the general convergence theorem we conclude that $u(x,t)\to 1$ in $L^\infty_{loc}(\R)$ topology. This proves the hair-trigger effect.
\hfill $\Box$

\begin{rem}\label{rem:hair-weak}\label{rem:hair-trigger}\rm
From the above proofs we see that \eqref{u0-is-not-small-near0} holds if $\lim\limits_{u\to 0+0} \frac{f(u)}{u^m} =\infty$. So, the hair-trigger effect for the monostable RPME actually holds under this condition rather than the stronger one $f'(0)>0$.
\end{rem}

\section{Combustion RPME}
In this section we study the asymptotic behavior for the solutions of (CP) or (pCP) with combustion type of reaction: $f$ satisfies (f$_C$).

\subsection{Stationary solutions}\label{subsec:-com-ss}

On the nonnegative stationary solutions of the equation in (CP), we first have Case A - Case D types as in the monostable RPME. The difference is that:
in Case A the solutions are $U_q\equiv q$ for $q\in [0,\theta]\cup \{1\}$ in the combustion case (However, only $U_0, U_\theta$ and $U_1$ are possible choices for the $\omega$-limits of the solution $u$ of (CP).); in Case D, the compactly supported stationary solutions $U_{q_0}$ exists only for $q_0 \in (\theta,1)$. The decay rates in \eqref{combustion-q-V-0} and \eqref{rate-V-0} remain valid (in fact, they are true for $x\in [l_0, l_0 + C^{-1/2} \theta^m]$ due to $f(u)\equiv 0$ in $[0,\theta]$). As we mentioned in Remark \ref{rem:ss-eq-p}, each $U_{q_0}$ is just a stationary solution of the equation but not of the Cauchy problem (CP).

\subsection{Asymptotic behavior of the solutions}
Since the current problem has no Type I and Type II ground state solutions, by the general convergence results in Theorem \ref{thm:general-conv} we know that $u(\cdot,t)$ converges as $t\to \infty$ to some constant solution in Case A.

Now we present a sufficient condition for the spreading phenomena: $u(\cdot,t)\to 1$ as $t\to \infty$.

\begin{lem}\label{lem:sufficient-for-spreading}
Assume $f$ is a combustion reaction, $u_0\in \mathfrak{X}$ satisfies
$$
u_0(x) \geq U_{q_0}(x),\quad x\in \R,
$$
for some stationary solution $U_{q_0}$ in Case D. Then spreading happens for the solution $u$ of {\rm (CP)}.
\end{lem}

\begin{proof}
By the comparison result, the solution $u(x,t)$ of (CP) remains above $U_{q_0}(x)$. Hence any $\omega$-limit of $u$ is also above $U_{q_0}(x)$. The only possible choice for such stationary solutions in Case A is nothing but $U_1 \equiv 1$. Moreover, due to the finite propagation speed and the properties in Lemma \ref{lem:monotonocity}, we know that ${\rm spt}[u(\cdot,t)] =[l(t),r(t)]$ for large $t$, hence the convergence $u\to 1$ holds in the sense of \eqref{def-conv}.
\end{proof}

Next we present some sufficient conditions for the vanishing phenomena: $u(\cdot,t)\to 0$ as $t\to \infty$.

\begin{lem}\label{lem:sufficient-vanishing}
Assume $f$ is a combustion reaction, $u_0\in \mathfrak{X}$ satisfies $u_0(x)\leq \theta$. Then vanishing happens in the topologies of $L^\infty(\R)$ for the solution $u$ of {\rm (CP)}.
\end{lem}

\begin{proof}
The proof is simple since $u$ actually solves the PME when $u(x,t)\leq \theta$, and so we can take a ZKB solution as a supersolution to suppress $u$ to $0$.
\end{proof}

\begin{rem}\label{rem:no-small-constant}\rm
For any $q\in (0,\theta)$, $U_q \equiv q$ will not be an $\omega$-limit. In fact, if $u(x,t_n)\to q\ (n\to \infty)$ in the $C_{loc}(\R)$ topology, then, in the bounded interval $[-b,b]$, $u(x,T)<\theta$ for some large $T$. This inequality also holds in $[l(T),r(T)]$ by the monotonicity outside of $[-b,b]$. Hence vanishing happens for $u$  by the above lemma, a contradiction.
\end{rem}

To prove the sharpness of the transition: $u\to \theta$, we need the following lemmas.

\begin{lem}\label{lem:separate}
Assume {\rm (F)}. For $i=1,2$, assume $u_{i0}\in C(\R)$ satisfies
$$
u_{i0}(x)>0 \mbox{ in }(l_i(0),r_i(0)),\quad u_{i0}(x)=0 \mbox{ for }x\leq l_i(0) \mbox{ and }x\geq r_i(0),
$$
and
$$
u_{10}(x)\leq u_{20}(x),\quad x\in \R.
$$
Let $u_i$ be the solution of {\rm (CP)} with initial data $u_{i0}$ and let $l_i(t),\ r_i(t)$ be its left and right free boundaries. Then there exist $T\geq 0,\ \tau\geq 0$ and $\varepsilon_0>0$ such that
\begin{equation}\label{shift<shift}
u_1(x,t)\leq u_2(x+\varepsilon,t+\tau),\quad x\in \R,\ t\geq T,\ \varepsilon\in [0,\varepsilon_0] \mbox{ or } \varepsilon\in [-\varepsilon_0,0],
\end{equation}
provided one of the following conditions holds:

\begin{enumerate}[{\rm (i).}]
\item $r_2 (t)-l_2(t)\not\equiv r_1(t)-l_1(t)$ for all $t\geq 0$;
\item $t^*(l_2(0))$ or $t^*(r_2(0))$ is finite, and $r_2 (t)-l_2(t)= r_1(t)-l_1(t)$ for all $t\geq 0$.
\end{enumerate}

\end{lem}

\begin{proof}
By comparison we have
$$
u_1(x,t)\leq u_2(x,t)\mbox{ for } x\in \R \mbox{ and } t\geq 0,
\quad l_2(t)\leq l_1(t)<r_1(t)\leq r_2(t)\mbox{ for } t\geq 0.
$$

(i). In case $r_2 (T_0)-l_2(T_0) > r_1(T_0)-l_1(T_0)$ for some $T_0\geq 0$, we assume without loss of generality that $l_2(T_0)<l_1(T_0)$, and by continuity assume $T_0>0$. Then using the strong maximum principle in $\{(x,t)\mid l_1(t)\leq x<r_1(t),\ T_0-1 \ll t \leq T_0\}$ we have
$$
u_1(x,T_0)< u_2(x,T_0)\mbox{ for } x\in [l_1(T_0), r_1(T_0)).
$$
So, there exists a small $\varepsilon_0>0$ such that
$$
u_1(x,T_0)< u_2(x+\varepsilon,T_0)\mbox{ for } x\in [l_1(T_0), r_1(T_0)],\quad \varepsilon\in [-\varepsilon_0,0).
$$
The conclusion then follows from the comparison principle for $T=T_0$ and $\tau=0$.

(ii). Assume without loss of generality that $t^*(l_2(0))<\infty$. Given $T_1 > t^*(l_2(0))$, then both $u_1(\cdot,t)$ and $u_2(\cdot,t)$ are strictly increasing in $J_1:= [x_1, l_1(0)]$, where $x_1:=l_2(T_1) = l_1(T_1)$. By the monotonicity of $l_2(t)$ after the waiting time we see that $l_2(t)$ is strictly decreasing in $t\geq T_1$. Hence, for any small $t>0$ we have
$$
l_2(T_1+ t)<x_1 = l_2(T_1) \mbox{\ \  and\ \  }
0< u_2(x_1,T_1 +t) = u_1(x_2 (t), T_1)< u_1(l_1(0),T_1),
$$
for some $x_2(t)\in [x_1, l_1(0)]$.  
Choose $\tau>0$ small, and denote
$$
E_1:= \{(x,t)\mid l_2(T_1+t)\leq x\leq x_1, 0\leq t\leq \tau\} \mbox{\ \ and\ \ }
E_2:= \{(x,t)\mid x_1\leq x\leq x_2(t), 0\leq t\leq \tau\}.
$$
By the monotonicity of $u_1(\cdot,T_1)$ and $u_2(\cdot,T_1+\tau)$ in $[x_1, x_2(\tau)]$, we have
$$
u_1(x,T_1) < u_1 (x_2(\tau),T_1) =  u_2(x_1,T_1+\tau) < u_2(x,T_1+\tau),\quad x\in [x_1, x_2(\tau)].
$$
Thus there exists $\varepsilon_0>0$ small such that
\begin{equation}\label{left-u1<u2}
u_1(x,T_1) < u_2(x+\varepsilon,T_1+\tau),\quad x\in [x_1, x_2(\tau)],\ \varepsilon \in [-\varepsilon_0,0].
\end{equation}
On the other hand, when $\tau>0$ and $\varepsilon_0$ are sufficiently small we have
$$
u_1(x,T_1) < u_2(x+\varepsilon,T_1+\tau),\quad x\in [x_2(\tau), r_1(T_1)],\ \varepsilon \in [-\varepsilon_0,0].
$$
Combining with \eqref{left-u1<u2} we have
$$
u_1(x,T_1) < u_2(x+\varepsilon,T_1+\tau),\quad x\in [l_1(T_1), r_1(T_1)],\ \varepsilon \in [-\varepsilon_0,0].
$$
Then the conclusion follows from the comparison principle.
\end{proof}

\begin{lem}\label{lem:com-sharp}
Under the assumptions of the previous lemma, if we further assume that $u_1(\cdot,t) \to \theta$ as $t\to \infty$, then $u_2(\cdot,t)\not\to \theta$ as $t\to \infty$.
\end{lem}

\begin{proof}

Since $u_1\to \theta$, there exists $T'\geq T+1$ for $T$ in \eqref{shift<shift} such that
\begin{equation}\label{u*-T}
u_1 (x,t) < \theta + \frac{\delta}{2},\quad x\in \R,\ t\geq T'-1,
\end{equation}
where $\delta$ is given in \eqref{combus}. Now, for $\lambda\in(0,1)$, we define
\[
w^\lambda(x,t):=\lambda^{-2}\:\! u_1 (\lambda^m x, \lambda^2 t),\quad x\in \R,\ t\geq T'.
\]
We choose
$\lambda_0\in(0,1)$ close enough to $1$ so that, for every
$\lambda\in[\lambda_0,1)$,
\begin{equation}\label{vsigma-bound}
 w^\lambda(x,t)\leq \theta+ \delta, \quad x\in \R, \ t\geq  T',
\end{equation}
and, by \eqref{shift<shift},
\begin{equation}\label{vsigma-u-T}
w^\lambda(x,T')\leq u_2 (x+\varepsilon_1,T'),\quad x\in  [\lambda^{-m}l_1(T'), \lambda^{-m}r_1(T')]
\end{equation}
for some $\varepsilon_1$.  Observe that $w^\lambda$ satisfies the equation
\[
w^\lambda_t= (w^\lambda)^m_{xx}+f(\lambda^2 w^\lambda),\quad x\in \R,\ t>0.
\]
By \eqref{vsigma-bound} and the assumption in (f$_C$) we have $f(\lambda^2 w^\lambda)\leq f(w^\lambda)$.  (Here is the only place we use the assumption $f'(u)>0$ in $(\theta, \theta+\delta]$.)
Therefore in view of \eqref{vsigma-u-T}, we
find that $w^\lambda$ is a subsolution  of (CP) for $t\geq T'$. It follows that $u_2(x+\varepsilon_1,t)\geq w^\lambda (x,t)$ for $t\geq T'$.
If $u_2\to \theta\ (t\to \infty)$, then by taking limit as $t\to \infty$ we conclude that $\theta \geq \lambda^{-2} \theta$, a contradiction.
\end{proof}

\bigskip
\noindent
{\it Proof of Theorem \ref{thm:com}}.
From above lemmas we know that $u(\cdot,t)$ converges as $t\to \infty$ to $0$, or $\theta$, or $1$.

For the initial data $\phi_\sigma \in \mathfrak{X}$, when $\sigma>0$ is small we have by $(\Phi_3)$ that $\phi_\sigma \leq \theta$. Hence, vanishing happens for $u_\sigma$ by Lemma \ref{lem:sufficient-vanishing}. Denote
$$
\Sigma_0:= \{\sigma>0 \mid \mbox{vanishing happens for } u_\sigma\}.
$$
Then $\Sigma_0$ is not empty, and, by comparison, it is an interval. Since $\phi_\sigma$ depends on $\sigma$ continuously, so does the solution $u_\sigma$ (cf. \cite{ACP}). Then, $\Sigma_0$ is an open interval $(0,\sigma_*)$ for some $\sigma_* \in (0,\infty]$.

In case $\sigma_*=\infty$ (this is the complete vanishing case, as stated in Theorem \ref{thm:complete-vanishing-main}), there is nothing left to prove.

The left case is $\sigma_* \in (0,\infty)$. (This happens in particular when  the condition in Lemma \ref{lem:sufficient-for-spreading} holds.) In this case, the solution $u_{\sigma_*}$ with critical initial data $\phi_{\sigma_*}$ converges to $\theta$ or $1$. We will show that $u_{\sigma_*}$ is a transition solution. Otherwise, spreading happens for $u_{\sigma_*}\to 1$, that is,
$$
\sigma_* \in \Sigma_1:= \{\sigma\geq \sigma_* \mid \mbox{spreading happens for } u_\sigma\}.
$$
Then, for any $q_0\in (\theta,1)$ and the corresponding stationary solution $U_{q_0}$ in Case D in Subsection \ref{subsec:-com-ss}, there exists a sufficiently large $T$ such that
$$
u_{\sigma_*}(x,T) > U_{q_0} (x),\quad x\in {\rm spt} [U_{q_0}].
$$
By the continuous dependence of $u$ on the initial data, for any $\sigma$ satisfying $0<\sigma_* -\sigma\ll 1$, we also have
$$
u_{\sigma}(x,T) > U_{q_0} (x),\quad x\in {\rm spt} [U_{q_0}].
$$
Consequently, the $\omega$-limits of $u_{\sigma}$ are bigger than $U_{q_0}$, but these $\omega$-limits are nothing but $1$, that is, spreading happens for $u_\sigma$. This implies that $\Sigma_1$ is an open set and it contains $\sigma_*$, contradicts the definition of $\sigma_*$.
Therefore, $u_{\sigma_*}$ is a transition solution, that is, $u(x,t)\to \theta$ as $t\to \infty$, and so both of its left and right boundaries tend to infinity.  It follows from
Lemmas \ref{lem:separate} and \ref{lem:com-sharp} that $u_\sigma\not\to \theta$ for any $\sigma>\sigma_*$. Consequently, $\Sigma_1 = (\sigma_*, \infty)$.

To finish the proof of Theorem \ref{thm:com}, we only need to prove the asymptotic
speeds of the free boundaries in \eqref{transition-asy-speed} for the transition solution, which follows from Proposition \ref{prop:transition-speed} below.

This completes the proof of Theorem \ref{thm:com}.
\hfill $\Box$

\medskip
At the end of this section we study the asymptotic speeds for the boundaries of the transition solution of (CP).
In the current combustion case, the equation is PME $u_t=(u^m)_{xx}$ when $u\leq \theta$. So, it is convenient to compare the solutions of this equation, especially, the selfsimilar solutions, with the solution of (CP).
For this purpose, we need a special selfsimilar solution which solves the following problem:
\begin{equation}\label{p-without-source}
\left\{
 \begin{array}{ll}
 v_t = (m-1)v v_{xx} + v_x^2, & 0<x<\rho(t),\ t>0,,\\
 v(0,t)= \Theta:= \frac{m}{m-1}\theta^{m-1},& t>0,\\
 v(\rho(t),t)=0, & t>0,\\
 \rho'(t)= -v_x(\rho(t),t), & t>0.
 \end{array}
\right.
\end{equation}

\begin{lem}\label{lem:selfsimilar-sol-comb}
There exists $y_0$ satisfying
\begin{equation}\label{def-xi0}
0<y_0 < \theta^{\frac{m-1}{2}}
\end{equation}
such that, with $\rho(t)= 2y_0 \sqrt{t}$, the problem \eqref{p-without-source} has a selfsimilar solution $v(x,t)=V(\frac{x}{2\sqrt{t}})$.
\end{lem}

\begin{proof}
The proof is divided into two steps.

{\it Step 1. An auxiliary problem}. We first consider the following initial value problem
\begin{equation}\label{p-auxi}
\left\{
 \begin{array}{l}
\displaystyle \xi''(y) = - 2 y \left(\xi^{\frac{1}{m}} \right)'(y)= - 2 \frac{y}{m} \xi^{\frac{1-m}{m}} \xi'(y),\quad y>0,\\
\displaystyle \xi(0) = \theta^m,\ \xi'(0)= - 2 \theta^{\frac{m+1}{2}}.
 \end{array}
\right.
\end{equation}
It is easily seen that, for any $h\in (0,\theta^m)$, the constant function $\xi\equiv h$ is a solution of the equation. Hence, it follows from $\xi'(0)<0$ that $\xi'(y)<0$ as long as $\xi>0$. In other words,  there exists $y_0 \in (0,\infty]$ such that
$$
\xi'(y)<0,\quad \xi(y)>0,\quad 0\leq y <y_0.
$$

We now show that $y_0<\infty$ and so $\xi(y_0)=0$. In fact, for any $\tilde{y}, y\in (0,y_0)$ with $\tilde{y} < y$ we have
$$
\xi''(s) < -2 \frac{\tilde{y}}{m} \xi^{\frac{1-m}{m}}(s) \xi'(s),\quad 0< s< \tilde{y}.
$$
Integrating it over $[0,\tilde{y}]$ we have
$$
\xi'(\tilde{y}) < - 2 \theta^{\frac{m+1}{2}} + 2\tilde{y} \left( \theta- \xi^{\frac{1}{m}}(\tilde{y})\right) < - 2 \theta^{\frac{m+1}{2}} + 2\theta \tilde{y}.
$$
Integrating it again over $[0,y]$ we have
$$
\xi(y) < \tilde{\xi}(y):= \theta^m - 2\theta^{\frac{m+1}{2}} y + \theta y^2.
$$
Since $\tilde{\xi}(y)$ has a unique zero $y_*:= \theta^{\frac{m-1}{2}}$, we conclude that
\begin{equation}\label{def-y*}
y_0 < y_*.
\end{equation}

{\it Step 2. Selfsimilar solution.} A direct calculation shows that
$$
u (x,t)= \xi^{\frac{1}{m}} \left(\frac{x}{2\sqrt{t}} \right)
$$
is a selfsimilar solution of
$$
\left\{
 \begin{array}{l}
 u_t = (u^m)_{xx}, \quad  0<x<\rho(t),\ t>0,,\\
 u(0,t)= \theta, \ \ u(\rho(t),t)=0, \quad t>0,
 \end{array}
\right.
$$
with $\rho(t):= 2y_0 \sqrt{t}$. Therefore,
$$
v(x,t) = V \left( \frac{x}{2\sqrt{t}} \right) :=  \frac{m}{m-1} \xi^{\frac{m-1}{m}} \left( \frac{x}{2\sqrt{t}} \right)
$$
is a selfsimilar solution of the first equation in \eqref{p-without-source}. It also satisfies the boundary conditions in \eqref{p-without-source} at $x=0$ and $x=\rho(t)$. Moreover, the interface $x= \rho(t)$ satisfies the Darcy law (cf. \cite[Theorem 15.19]{Vaz-book}):
$$
\rho'(t) = - v_x(\rho(t),t).
$$
In other words, we have $V'(y_0) = -2y_0$.
This proves the lemma.
\end{proof}

This lemma implies that $V \left( \frac{x}{2\sqrt{t}} \right)$ is not only a solution of the equation in \eqref{p-without-source}, but also a solution of the following Cauchy-Dirichlet problem
\begin{equation}\label{self-Cauchy-p}
\left\{
 \begin{array}{ll}
 v_t = (m-1)v v_{xx} + v_x^2, & x>0,\ t>0,,\\
 v(0,t)= \Theta:= \frac{m}{m-1}\theta^{m-1},& t>0,
 \end{array}
\right.
\end{equation}
with one free boundary $\rho(t)$.

\begin{prop}\label{prop:transition-speed}
Assume $f$ satisfies {\rm (f$_C$)}, $u$ is a transition solution of {\rm (CP)}. Then its left and right free boundaries $l(t)$ and $r(t)$ satisfy the asymptotic speed as in \eqref{transition-asy-speed}.
\end{prop}

\begin{proof}
The conclusion can be proved in a similar way as the authors did for the reaction diffusion equation in \cite[Proposition 4.2]{DLZ}. Note that there are several key points in the proof: (1) using the zero number argument to study the number of intersection points between two solutions. Here we see that the argument remains valid for our problem by Lemma \ref{lem:diminish-intersection} and by the Darcy law; (2) using a selfsimilar solution satisfying the Darcy law (which is called the Stefan boundary condition in \cite{DLZ}) to construct lower and supersolutions. Here we only need to use the selfsimilar solution $V(\frac{x}{2\sqrt{t}})$ obtained in the previous lemma to replace the selfsimilar solution $\Phi(t,x)$ in \cite{DLZ}; (3) using the limit $\theta(t)/r(t)\to 0\ (t\to \infty)$ as an extra condition, where $\theta(t)$ is the $\theta$-level set of $u(\cdot,t)$. This limit remains true for our problem and the proof is the same as that in Subsection 4.3 in \cite{DLZ}.
\end{proof}

\section{Bistable RPME}

In this section we study the asymptotic behavior for the solutions of (CP) or (pCP) with bistable type of reaction: $f$ satisfies (f$_B$).

\subsection{Stationary solutions}\label{subsec:bi-ss}
On the nonnegative stationary solutions of the equation in (CP), we first have Case A - Case D types as in the monostable RPME. The difference is that:
in Case A the solutions are $U_q$ for $q=0,\theta,1$ in the current case; in Case D, the compactly supported stationary solutions $U_{q_0}$ exists for $q_0 \in (\theta_1 ,1)$, where $\theta_1$ is defined below by \eqref{def-theta1}. The decay rates in \eqref{combustion-q-V-0} and \eqref{rate-V-0} remain valid. As we mentioned in Remark \ref{rem:ss-eq-p}, each $U_{q_0}$ is just a stationary solution of the equation but not of the Cauchy problem (CP).

Besides these stationary solutions, for  the bistable RPME, there are two other kinds of stationary solutions: periodic solutions and the {\it ground state solution}.

{\bf Case E}. In the $(q,p)$-phase plane, besides $(0,0)$ and $(1,0)$, we have another singular point $(\theta, 0)$, which is a center. There are infinitely many closed trajectories surrounding $(\theta,0)$. Each of them corresponds to a positive stationary periodic solution $U_{per}(x)$. By the monotonicity result in Corollary \ref{cor:multi-mono}, none of them can be a candidate of the $\omega$-limit for the solution of (CP) with $u_0\in \mathfrak{X}$.

{\bf Case F}. In the phase plane, there is a homoclinic orbit which starts and ends at $(0,0)$, crossing the point $(\theta_1,0)$, with $\theta_1\in (\theta,1)$ given by
\begin{equation}\label{def-theta1}
\int_0^{\theta_1} r^{m-1} f(r)dr=0.
\end{equation}
The function of this homoclinic orbit is
$$
p^2 = - 2m \int_0^q r^{m-1} f(r) dr,\quad 0<q\leq \theta_1.
$$
Denote the corresponding solution by $U (x)$, and assume it satisfies the normalized condition
$$
U (0)=\theta_1,\quad U'(0)=0.
$$
Assume the largest interval where $U$ remains positive is $(-L, L)$ for some $L\in (0,\infty]$. We can give a formula for $L$ as the following. Denote $\eta:= U^m$, then
\begin{equation}\label{def-p-eta}
\frac{d\eta}{dx}=p(\eta) := \left(-2m \int_0^{\eta^{1/m}} r^{m-1}f(r) dr\right)^{1/2},\quad x\in (-L, 0].
\end{equation}
Therefore,
\begin{equation}\label{formula-L*}
L := \int_0^{\theta_1^m} \frac{d\eta}{p(\eta)} =
\int_0^{\theta_1^m}  \left(-2m \int_0^{\eta^{1/m}} r^{m-1}f(r) dr\right)^{-1/2} d\eta.
\end{equation}
Under the assumptions in (f$_B$) we have $f'(0)<0$ and so $f(q)\sim f'(0)q$ as $q\to 0+0$.
We will see below that $L<\infty$ in this case. However, $L$ might be $\infty$ if $f$ is a bistable nonlinearity but without the assumption $f'(0)<0$.

\begin{lem}\label{lem:finite-infinite-GS}
Assume $f\in C([0,\infty)) \cap C^1((0,\infty))$ is a bistable nonlinearity satisfying \eqref{bi}. Further assume that
\begin{equation}\label{decay-rate-f-0}
f(q)= -\lambda q^\alpha (1+o(1))\mbox{\ \  as\ \  } q\to 0+0,
\end{equation}
for some $\lambda>0$ and $\alpha>0$. Then $L=\infty$ when $\alpha\geq m$; $L<\infty$ when $0<\alpha<m$.
\end{lem}

\begin{proof}
Substituting \eqref{decay-rate-f-0} into the formula $p(\eta)$ in \eqref{def-p-eta} we have
\begin{equation}\label{rate-p-eta}
p(\eta) \sim \left(\frac{2m \lambda}{m+\alpha}\right)^{1/2} \eta^{ \frac{m+\alpha}{2m}}\mbox{\ \ as\ \ } \eta\to 0+0.
\end{equation}

{\it Case 1. $\alpha\geq m$.}  When $0<\delta\ll 1$, we have
\begin{equation}\label{L*=infinity}
L > \int_0^\delta  \left(\frac{m+\alpha}{3m \lambda}\right)^{1/2} \eta^{- \frac{m+\alpha}{2m}} d\eta =\infty.
\end{equation}
Thus, the solution $U(x)>0$ in $\R$, it is a Type I ground state solution, as that in the bistable RDEs.

{\it Case 2. $0<\alpha<m$.} When $0<\delta\ll 1$, we have
\begin{equation}\label{L*=finite}
L < \int_\delta^{\theta_1^m} \frac{d\eta}{p(\eta)} + \int_0^\delta  \left(\frac{m+\alpha}{m \lambda}\right)^{1/2} \eta^{- \frac{m+\alpha}{2m}} d\eta <\infty.
\end{equation}
By \eqref{rate-p-eta} and \eqref{def-p-eta} we have
$$
\frac{d\eta}{dx} \sim \left(\frac{2m \lambda}{m+\alpha}\right)^{1/2} \eta^{ \frac{m+\alpha}{2m}}\mbox{\ \ as\ \ } x\to -L+0,
$$
and so
$$
\frac{2m}{m-\alpha} \Big( \frac{m+\alpha}{2m\lambda}\Big)^{1/2} \Big( \eta^{\frac{m-\alpha}{2m}} \Big)' \sim 1,\mbox{\ \ as\ \ } x\to -L+0.
$$
Hence,
\begin{equation}\label{ground-state-rate}
U(x) = \eta^{1/m} (x)\sim A_1(\alpha,\lambda) (x+L)^{\frac{2}{m-\alpha}} \mbox{\ \ as\ \ }x\to -L + 0,
\end{equation}
with
$$
 A_1(\alpha,\lambda) :=
\left(\frac{m-\alpha}{2m}\right)^{\frac{2}{m-\alpha}}\left(\frac{2m \lambda}{m+\alpha}\right)^{\frac{1}{m-\alpha}}.
$$
From \eqref{ground-state-rate} we see that, at the boundaries $\pm L$, $U$ is continuous for all $m>\alpha$, Lipschitz when $\alpha < m\leq 2+\alpha$ and $C^1$ when $\alpha <m< 2+\alpha$.

If we denote by $V(x) := \frac{m}{m-1}U^{m-1}$ the corresponding pressure of $U$, then we have
\begin{equation}\label{ground-state-V}
\left\{
\begin{array}{l}
V\in C([-L,L]),\quad  V(0)=\Theta_1 := \frac{m}{m-1}\theta_1^{m-1},\quad V(\pm L) =0,\\
V(x)>0,\ \ V'(x)<0,\ \ V(x)=V(-x) \mbox{ for }x\in (0,L),\\
V(x)\sim A (x+L)^{\frac{2(m-1)}{m-\alpha}} \mbox{\ \ as\ \ }x\to -L +0,\quad  \mbox{with } A := \frac{m}{m-1}A_1^{m-1}.
\end{array}
\right.
\end{equation}
This proves the lemma.
\end{proof}

\begin{rem}\label{rem:combination-of-GS}
\rm In the special case when $f$ is a bistable nonlinearity with $f'(0)<0$ as in (f$_B$), we have $\alpha=1$ and $\lambda =-f'(0)$, and so $V$ in \eqref{ground-state-V} satisfies:
$$
V(x)\sim A (x+L)^2 \mbox{\ \ as\ \ }x\to -L +0,\quad  \mbox{with } A := \frac{-(m-1)f'(0)}{2(m+1)}.
$$
Thus $V'(-L+0)=0$, and so $V$ or $U= (\frac{m-1}{m}V)^{\frac{1}{m-1}}$ is a Type II ground state solution. By Lemma \ref{lem:ss-problem}, $V(x)$ is a stationary solution not only for the bistable RPME, but also for the Cauchy problem (pCP). This is different from the compactly supported solutions in Case D, where they are only stationary solutions of the RPME but not of (pCP). Furthermore, the combinations of such solutions, like
\begin{equation}\label{combination-GS}
\mathcal{V}(x):= V (x-z_1)+ V(x -z_2), \quad \mbox{with\ \ }z_2 -z_1 \geq 2L,
\end{equation}
are also Type II ground state stationary solutions of the problem (pCP).
\end{rem}

\subsection{Asymptotic behavior}

{\it Proof of Theorem \ref{thm:bi}}: We prove the trichotomy result in Theorem \ref{thm:bi}.
By $f'(0)<0$, it follows from Remark \ref{rem:combination-of-GS} that, the problem has only Type II ground state solution as in \eqref{ground-state-V} or \eqref{combination-GS}. By the general convergence result Theorem \ref{thm:general-conv} we know that $u$ converges as $t\to \infty$ to $0, \theta, 1$, or a Type II ground state solution.

The constant solution $\theta$ can be excluded easily from the candidates. In fact, for any positive stationary periodic solution $U_{per}(x)$ of the equation (see Case E in Subsection \ref{subsec:bi-ss}), by the classical zero number argument we see that the zero number of $u_\sigma(\cdot,t)-U_{per}(\cdot)$ is finite and decreasing in $t>0$. However, if $u_\sigma\to \theta$ for some $\sigma>0$, then the above zero number tends to that of $\theta -U_{per}(\cdot)$, which is infinite. So we derive a contradiction.

For the initial data $\phi_\sigma \in \mathfrak{X}$, when $\sigma>0$ is small we have by ($\Phi_3$) that $\phi_\sigma \leq \theta$. Hence, a similar argument as in Lemma \ref{lem:sufficient-vanishing} implies that vanishing happens for $u_\sigma$. Denote
$$
\Sigma_0:= \{\sigma>0 \mid \mbox{vanishing happens for } u_\sigma\}.
$$
Then $\Sigma_0$ is an open interval $(0,\sigma_*)$ for some $\sigma_*\in (0,\infty]$, as in the combustion equation.

In case $\sigma_*=\infty$ (this is the complete vanishing case, as stated in Theorem \ref{thm:complete-vanishing-main}), there is nothing left to prove.

In case $\sigma_*$ is a positive number, we can show that $u_{\sigma_*}$ is a transition solution, that is, it converges to some Type II ground state solution $\mathcal{U}$. Otherwise, it tends to $1$. Then the set
$$
\Sigma_1:= \{\sigma>0 \mid u_{\sigma}\to 1\mbox{ as }t\to \infty\}
$$
is not empty. Since this set is open as proved in the combustion case and since $\sigma_*\in \Sigma_1$ by our assumption, we see that $\sigma\in \Sigma_1$ for any $\sigma$ with $0<\sigma_* -\sigma\ll 1$. This contradicts the definition of $\sigma_*$.
In what follows we assume $u(x,t)$ converges to a Type II ground state solution $\mathcal{U}(x)$ as in \eqref{finite-ground-state-0}.

If $\Sigma_1$ is empty, then $u_{\sigma}$ is a transition solution for all $\sigma\geq \sigma_*$ (this is the case $\sigma_*\in (0,\infty)$ and $\sigma^*=\infty$ in Theorem \ref{thm:bi}), and the proof is also completed. In case $\Sigma_1$ is a non-empty open set, by the comparison principle, it is actually an interval $(\sigma^*, \infty)$ for some $\sigma^*\geq \sigma_*$. For each $\sigma\in [\sigma_*, \sigma^*]$, $u_\sigma$ is a transition solution.

This completes the proof for Theorem \ref{thm:bi}.
\hfill $\Box$

\medskip

On the sharpness of the transition solution: $\sigma_* =\sigma^*$, we give some sufficient conditions.

\begin{lem}\label{lem:sharp-bi}
Let $u_\sigma$ be the solution of {\rm (CP)} with bistable $f$ and with initial data $\phi_\sigma$ satisfying $(\Phi_1)$-$(\Phi_3)$. Assume $u_{\sigma_*}$ is a transition solution, then, for any $\sigma>\sigma_*$, $u_\sigma$ is no longer a transition one, if one of the following conditions holds:
\begin{enumerate}[{\rm (i).}]
\item $(\Phi_2)$ is strengthened as: $\phi_{\sigma_*}(x) \leq \phi_{\sigma}(x+\varepsilon)$ when $|\varepsilon|$ is small;

\item there is $i\in \{1,2,\cdots,k\}$ such that $r_i(t)-l_i(t)\not\equiv r_{*i}(t)-l_{*i}(t)$, where $l_i(t),r_i(t)$ are the left and right boundaries of the $i$-th connected component of the the support ${\rm spt}[u_\sigma(\cdot,t)]$, $l_{*i}(t), r_{*i}(t)$ are the analogues of ${\rm spt}[u_{\sigma_*}(\cdot,t)]$;

\item $l_i(t)\equiv l_{*i}(t),\ r_i(t)\equiv r_{*i}(t)$ for all $i=\{1,2,\cdots,k\}$, and one of the waiting times of these boundaries is finite.
\end{enumerate}
\end{lem}

\begin{proof}
In a similar way as proving Lemma \ref{lem:separate} we see that any of the three conditions (i)-(iii) implies that, for some $T\geq 0,\ \tau\geq 0$ and $\varepsilon_0>0$, there holds
$$
u_{\sigma_*}(x,t)\leq u_{\sigma}(x+\varepsilon, t+\tau),\quad x\in [l_{*i}(t),r_{*i}(t)],\ t\geq T,\ \varepsilon\in [0,\varepsilon_0] \mbox{ or }\varepsilon\in [-\varepsilon_0,0].
$$
In case $u_\sigma$ is also a transition solution, by taking limit as $t\to \infty$ in this inequality we conclude that
$$
U(x-z_*)\leq U(x+\varepsilon-z_\sigma),\quad x\in [l^\infty_*, r^\infty_*],\ \varepsilon\in [0,\varepsilon_0] \mbox{ or }\varepsilon\in [-\varepsilon_0,0],
$$
for some $z_*, z_\sigma\in [-b,b]$, where $l^\infty_*=\lim\limits_{t\to \infty}l_{*i}(t)$,  $r^\infty_*=\lim\limits_{t\to \infty}r_{*i}(t)$. This is impossible since $U$ is a Type II ground state solution.
\end{proof}

\begin{rem}\label{rem:sharp-bi-equiv}\rm
The only situation not included in {\rm (ii)} and {\rm (iii)} in the previous lemma is
\begin{equation}\label{left-condition}
\left\{\begin{array}{l}
l_i(t)\equiv l_{*i}(t)\equiv l_i(0),\ r_i(t)\equiv r_{*i}(t)\equiv r_i(0) \mbox{\ \ for all } i=\{1,2,\cdots,k\},\\
\mbox{and all the waiting times of these boundaries are infinite}.
\end{array}
\right.
\end{equation}
We guess that, even in this case, the transition solution is also unique. The instability of Type II ground state solution should be the essential reason.
\end{rem}

\begin{rem}\rm
In case the initial data $\phi_\sigma$ has a unique maximum point, so does $u(\cdot,t)$. Then the limit of the transition solution is a Type II ground state solution $U(x-z)$ for some $z\in [-b,b]$.
Even in this case, the transition solution may have a very wide support but just converges to a positive limit in a bounded interval. For example,
assume $u(x,t;\sigma(b^2 -x^2))$ is a transition solution for some $\sigma>0$, and $b>L$, then
$$
u(x,t)>0 \mbox{ for } x\in (l(t),r(t)),\ t>0,
$$
$$
u(x,t)\leq u(\pm L,t), \quad x\in (l(t), -L]\cup [L, r(t)),\ t>0,
$$
and
$$
u(\pm L, t)\to 0,\quad \|u(\cdot,t)-U(\cdot)\|_{L^\infty([l(t),r(t)])} \to 0 \mbox{\ \ as\ \ }t\to \infty.
$$
\end{rem}

\section{Complete Vanishing Phenomena}

Note that, Theorem \ref{thm:mono}  gives the hair-trigger effect for (CP) with monostable reactions (see also Remark \ref{rem:hair-Garriz} and \ref{rem:hair-weak}). In the bistable and combustion cases, however, Theorems \ref{thm:com} and \ref{thm:bi} show that vanishing happens for the solutions $u_\sigma$  when $\sigma<\sigma_*$, where $\sigma_*\in (0,\infty]$. A natural question is that: is it really possible that $\sigma_*=\infty$?
In other words, is there $b>0$ such that, for any initial data $u_0$ with support in $[-b,b]$,
vanishing always happens for $u$, no matter how large $\|u_0\|_{L^\infty}$ is? We call such a result as a {\it complete vanishing phenomena}. Of course, this phenomena seems difficult in most cases. However, it does happen
for some problems. Recently, Li and Lou \cite{LiLou} proved that such a phenomena really happens in bistable and combustion RDEs provided the nonlinearity $f(u)$ decreases sufficiently fast as $u\to \infty$. Also, it was shown in \cite{DL} that this phenomena occurs in the free boundary problems for RDEs with bistable, combustion or even monostable reactions. We now show that this phenomena happens for (CP), provided $f(u)$ decreases fast as $u\to \infty$.

\begin{thm}\label{thm:complete-vanishing-main}
Assume $f\in C^2$ is a bistable or combustion reaction term. Assume also that
\begin{equation*}\label{f-decrease-fast}
\mbox{\rm (F)$_\infty$}\hskip 60mm
\liminf\limits_{u\to +\infty} \frac{-f(u)}{u^{p_0}} >0, \hskip 65mm
\end{equation*}
for some $p_0>m$. Then there exists a small $b>0$ such that vanishing always happens for the solution of {\rm (CP)} provided ${\rm spt}[u_0]\subset [-b,b]$.
\end{thm}

\begin{proof}

Recall that, when $f$ is a bistable or combustion nonlinearity, there exists $K>0$ such that
$$
f(u)\leq Ku,\quad u\geq 0.
$$
Using the pressure notation $v$ and
$$
g(v) := m\left(\frac{(m-1)v}{m}\right)^{\frac{m-2}{m-1}} f\left( \Big(\frac{(m-1)v}{m}\Big)^{\frac{1}{m-1}}\right),
$$
we see that
\begin{equation}\label{cond-f-to-F}
 g(v)\leq K(m-1)v \mbox{ for } v\geq 0.
\end{equation}
Denote
$$
p_1 := \frac{m-2+p_0}{m-1},\quad p:= \frac{p_0 + 3m -4}{2(m-1)},\quad \alpha:= \frac{2}{p-2}.
$$
Then $p_1 - p = p-2 = \frac{p_0 -m }{2(m-1)} = \frac{2}{\alpha}$. By (F)$_\infty$ we have
$$
\liminf\limits_{v\to +\infty} \frac{-g(v)}{v^{p_1}} >0.
$$
So, there exists $L_0>0$ and $M>1$ such that
\begin{equation}\label{cond-f-to-F-2}
g(v)\leq -L_0 v^{p_1-p}v^{p} \leq - L v^p, \qquad v\geq M,
\end{equation}
where $L:= L_0 M^{p_1 -p}$.

Recall that $\theta$ is a positive zero of $f$ in the bistable and combustion equations. Choose $M>1$ larger if necessary (to be determined below) and denote $M_0 := \Big( \frac{(m-1)M}{m}\Big)^{\frac{1}{m-1}}$, then we can choose $s_0 >0$, with $0< 1 -K(m-1)s_0 \ll 1$ when $M_0$ is sufficiently large, such that
\begin{equation}\label{xx-1}
\theta \frac{\left( \frac{1-K(m-1)s_0}{2K}\right)^{\frac{1}{m+1}}}{\left( 1+ \frac{(m-1)-K(m^2 -1)s_0}{2}\right)^{\frac{1}{m-1}}} = 2M_0 s_0^{\frac{1}{m+1}}>1.
\end{equation}
Define $b_1 = C_1(m) s_0^{\frac{1}{m+1}}$ by
\begin{equation}\label{xx-2}
\frac{(m-1)b^2_1}{2m(m+1) s_0^{\frac{2}{m+1}}} = 1- \frac{1}{2^{m-1}}.
\end{equation}
Then, we can choose $M>1$ sufficiently large such that
\begin{equation}\label{xx-4}
b := M^{-\frac{1}{\alpha}} \leq \min\left\{ \frac{b_1}{3}, 1 \right\},
\end{equation}
\begin{equation}\label{xx-5}
L := L_0 M^{p_1 - p}\geq \frac{3c}{b_1 M} + 2^p [(m-1)\alpha (\alpha+1) +\alpha^2],
\end{equation}
where $c>0$ is defined by
$$
c:= 2^{\frac{1}{\alpha}} b^{-1-\alpha} \alpha^{-1} [2(m-1)\alpha(\alpha+1)+2\alpha^2 +K(m-1)].
$$
Denote $C:= (2M_0)^{m-1} s_0^{\frac{m-1}{m+1}}>1$, then
\begin{equation}\label{xx-6}
\left( C - \frac{(m-1)b^2_1}{2m(m+1) s_0^{\frac{2}{m+1}}} \right)^{\frac{1}{m-1}} \geq \frac12 C^{\frac{1}{m-1}} = M_0 s_0^{\frac{1}{m+1}}.
\end{equation}
We will show that the complete vanishing phenomena happens if ${\rm spt}[u(x,0)]\subset [-b,b]$.

\medskip
{\it Step 1}. We first consider the stage when the solution is suppressed by a spatial homogeneous upper solution down to the range $v\leq M$. Consider the equation
$$
v_{1t} = -L v_1^2.
$$
Then $v_1 = (Lt)^{-1}$ is monotonically decreasing from $+\infty$, and till it reaches $M$ at time $t= t_1 := (LM)^{-1}$, it is a supersolution of (pCP) due to \eqref{cond-f-to-F-2} and $p>2$, no matter how large $\|u_0\|_{L^\infty}$ is. So,
$$
v(x,t)\leq v_1(t_1)=M,\quad t\geq t_1.
$$

This inequality gives an estimate for $v$ from above. We also need estimates on the right and left boundaries. We state only the case on the right side since the left side is studied similarly. For $c,b$ given above, we define
$$
v_2 (x,t):= (x-ct-b)^{-\alpha}-b^{-\alpha},\quad b+ct<x\leq 2b+ct,\ t\geq 0.
$$
We now show that $v_2$ is also a supersolution of (pCP) in the domain
$$
E:= \{(x,t)\mid b+ct<x\leq 2b+ct,\ t\geq 0\}.
$$

1). First, on the interval
$J_1:= [b+ct + 2^{-\frac{1}{\alpha}}b, 2b+ct]$, we have
$$
2^{-\frac{1}{\alpha}}b \leq H:= x-ct-b \leq b,\quad 0\leq v_2 \leq b^{-\alpha},
$$
and so by the first inequality in \eqref{cond-f-to-F} and the choice of $c$ we have
\begin{eqnarray*}
\mathcal{N}v_2 & := & v_{2t} - (m-1)v_2 v_{2xx} -v_{2x}^2 - g(v_2)\\
& \geq & \alpha c H^{-\alpha-1} -(m-1)\alpha (\alpha+1) v_2 H^{-\alpha-2} -\alpha^2 H^{-2\alpha -2} - K(m-1)v_2\\
& \geq & H^{-\alpha-2} \left[ \alpha c H -(m-1)\alpha (\alpha+1) H^{-\alpha} -\alpha^2 H^{-\alpha } \right] - K(m-1) b^{-\alpha}\\
& \geq & b^{-\alpha-2} \left[ \alpha c 2^{-\frac{1}{\alpha}} b -2(m-1)\alpha (\alpha+1) b^{-\alpha} - 2 \alpha^2 b^{-\alpha} \right] - K(m-1) b^{-\alpha}\\
& \geq & b^{-2\alpha-2} \left[ \alpha c 2^{-\frac{1}{\alpha}} b^{1+\alpha} -2(m-1)\alpha (\alpha+1) - 2 \alpha^2 - K(m-1) \right] \geq 0.
\end{eqnarray*}

2). Next, in the interval $J_2:= (b+ct, b+ct+2^{-\frac{1}{\alpha}}b]$ we have
$$
0<H:= x-ct-b \leq 2^{-\frac{1}{\alpha}} b,\quad v_2 \geq b^{-\alpha} \geq M,
$$
and by the choice of $L$ we have
\begin{eqnarray*}
\mathcal{N}v_2 & \geq & H^{-\alpha-2} \left[ \alpha c H -(m-1)\alpha (\alpha+1)H^{-\alpha} -\alpha^2 H^{-\alpha} \right] + Lv_2^{p}\\
& \geq & L (H^{-\alpha}-b^{-\alpha})^{p} - [(m-1)\alpha (\alpha+1) + \alpha^2 ] H^{-2\alpha-2}  \\
& \geq & 2^{-p} L H^{-\alpha p} - [(m-1)\alpha (\alpha+1) + \alpha^2 ] H^{-2\alpha-2} \\
& \geq & H^{-2\alpha -2} \left[2^{-p} L - (m-1)\alpha (\alpha+1) -\alpha^2 \right]\geq 0.
\end{eqnarray*}

3). On the free boundary $r_2(t) := 2b+ct$ of $v_2$, by the choice of $c$, we easily obtain
$$
-v_{2x}(2b+ct,t) = \alpha b^{-\alpha-1} \leq c,\quad t\geq 0.
$$
Therefore, $v_2$ is a supersolution of (pCP) in the domain $E$.
As a consequence, at $t_1$, the right free boundary $r(t)$ of $v$ satisfies
$$
r(t_1) \leq r_2(t_1) =2b + ct_1 \leq 2b + \frac{c}{LM} \leq b_1,
$$
by the choice of $b,c$ and $L$. Similarly, the left free boundary $l(t)$ of $v$ satisfies $l(t_1)\geq -b_1$.

Using the notation of $u$ we conclude from above that
\begin{equation}\label{stage-1}
u(x,t_1)\leq \bar{u} (x):= M_0  \cdot \chi_{[-b_1, b_1]}.
\end{equation}

\medskip

{\it Step 2}. Next we use the ZKB solution as another upper solution to suppress $u$ further such that it goes down below $\theta$ after some time. More precisely, consider
$$
\left\{
 \begin{array}{ll}
 u_{3t} = (u^m_3)_{xx} + Ku_3 , & x\in \R,\ t>0,\\
 u_3 (x,0)= \bar{u} (x), & x\in \R.
 \end{array}
 \right.
$$
Then $u_3$ is a upper solution of (CP) and
$$
u(x,t+t_1) \leq u_3(x,t), \quad x\in \R,\ t>0.
$$

Change the time variable from $t$ to $s$ by
$$
s:= \frac{e^{K(m-1)t} -1}{K(m-1)} \Leftrightarrow t = T(s):= \frac{\ln (1+K(m-1)s)}{K(m-1)},
$$
and define
$$
w(x,s):= u_3(x, T(s)) [1+K(m-1)s]^{-\frac{1}{m-1}},
$$
then we have
\begin{equation}\label{problem-of-w-u3}
\left\{
 \begin{array}{ll}
  w_s = (w^m)_{xx},& x\in \R, \ s>0,\\
  w(x,0)= \bar{u}(x), & x\in \R.
  \end{array}
  \right.
\end{equation}
To estimate $w(x,s)$, we consider the ZKB solution
$$
w_1(x,s) := s^{-\frac{1}{m+1}} \left( C - \frac{(m-1)x^2}{2m(m+1) s^{\frac{2}{m+1}}}  \right)^{\frac{1}{m-1}}_+,\quad x\in \R,\ s\geq s_0,
$$
with $s_0$ and $C$ given above. By the choice of $b_1, C$ and $s_0$, for $x\in [-b_1, b_1]$ we have
\begin{eqnarray*}
w_1(x,s_0) & \geq & w_1(\pm b_1, s_0) =  s_0^{-\frac{1}{m+1}} \left( C - \frac{(m-1)b^2_1}{2m(m+1) s_0^{\frac{2}{m+1}}}  \right)^{\frac{1}{m-1}} \\
& \geq & \frac12 s_0^{-\frac{1}{m+1}} C^{\frac{1}{m-1}} > M_0 \geq \bar{u}(x).
\end{eqnarray*}
Therefore, $w_1$ is a upper solution of \eqref{problem-of-w-u3}, and so
$$
w(x,s)\leq w_1(x,s+s_0),\quad x\in \R,\ s>0.
$$
In particular, at $x=0$ and $s=s_1:= \frac{1-K(m+1)s_0}{2K}$ we have
$$
w(0,s_1)  \leq  w_1(0,s_1+s_0) = (s_1+s_0)^{-\frac{1}{m+1}} C^{\frac{1}{m-1}}.
$$
Using the notation $u$ we have
\begin{eqnarray*}
u(x,T(s_1)+t_1) & \leq & u_3(x,T(s_1)) = w(x,s_1)[1+K(m-1)s_1]^{\frac{1}{m-1}}\\
& \leq &  (s_1+s_0)^{-\frac{1}{m+1}} C^{\frac{1}{m-1}} [1+K(m-1)s_1]^{\frac{1}{m-1}}\leq \theta,
\end{eqnarray*}
by \eqref{xx-1} and \eqref{xx-6}.

\medskip
{\it Step 3.} By the sufficient condition for vanishing in bistable and combustion RPMEs (cf. Lemma \ref{lem:sufficient-vanishing}), we conclude that $u\to 0$ as $t\to \infty$.

This proves the complete vanishing phenomena in Theorem \ref{thm:complete-vanishing-main}.
\end{proof}

\section{Appendix: Well-posedness and {\it a priori} Estimates}\label{sec:well}

For our RPME with general reaction terms, it seems that the a priori estimates, well-posedness and other important properties as that in PME are not well collected in literature, though they are not entirely new. We present some details here for the convenience of the readers.

\subsection{Well-posedness}
The existence of very weak solution of (CP) has been well studied by many authors (cf. \cite{ACP, PV, Sacks, Vaz-book, Wu-book}).

\begin{prop}\label{prop:well}
Assume {\rm (I)} and $f$ is Lipschitz continuous  with $f(0)=0$. Then the problem {\rm (CP)} has a unique very weak solution $u(x,t)\in C(Q_T)\cap L^\infty (Q_T)$ for any $T>0$, and
$$
0\leq u(x,t)\leq M'_0:= \max\{1,\|u_0\|_{L^\infty} \}\quad \mbox{ in } Q_T=\R\times (0,T).
$$
Moreover, there are two free boundaries $l(t)<r(t)$ such that ${\rm spt}(u(\cdot,t))\subset [l (t),r(t)]$ for all $t>0$.
\end{prop}

We can not obtain $u(x,t)>0$ as in RDEs due to the degeneracy of the equation at $u=0$. Nevertheless, the solution $u$ solves the equation in the classical sense by the standard parabolic theory in a neighborhood of every point $(x,t)$ at which the solution is positive (see also \cite{OKC}).

\subsection{Persistence of positivity and finite propagation speed}

We now show the following positivity persistence property as in PME.

\begin{prop}[Persistence of positivity]\label{prop:positive}
Assume $u(x_0,t_0)>0$ for some $x_0\in \R,\ t_0> 0$.
Then $u(x_0, t)>0$ for all $t\geq t_0$. $($Thus, $u$ is classical near the line $\{(x_0,t)\mid t >t_0\}).$
\end{prop}

\begin{proof}
Since $t_0>0$ we have $u(x,t_0)\in C(\R)$, and so, by our assumption there exist $\varepsilon>0,\ \delta>0$ such that
$$
u(x,t_0)\geq \varepsilon\mbox{ for } x\in J:= [x_0 -\delta, x_0+\delta].
$$
We use $K$ to denote the Lipschitz constant for $f$.
Denote
$$
\beta := \frac{1}{m-1}, \qquad \mu := K + \frac{\pi^2 }{4\delta^2}\varepsilon^{m-1},
$$
and define
\begin{equation}\label{def-underline-u}
\underline{u}(x,t):= \left\{
 \begin{array}{ll}
\displaystyle \varepsilon e^{- \mu t} \left(\sin \frac{\pi (x-x_0 +\delta)}{2\delta} \right)^{\frac1m}, & x\in J,\ t\geq 0,\\
 0, & x\in \R\backslash J,\ t\geq 0.
 \end{array}
 \right.
\end{equation}
It is easily seen that
$$
\underline{u}(x,0)\leq \varepsilon \leq u(x,t_0),\quad x\in J,
$$
and
$$
\underline{u}_t - (\underline{u}^m)_{xx} - f(\underline{u})
\leq \underline{u}_t - (\underline{u}^m)_{xx} + K\underline{u} \leq 0,\quad x\in \R,\ t\geq 0.
$$
By the comparison principle we have
$$
u(x,t+t_0) \geq \underline{u}(x,t)>0,\quad |x-x_0|< \delta,\ t>0.
$$
This proves the proposition.
\end{proof}

\begin{rem}\rm
A positive lower bound can also be given by a subsolution  with the form of ZKB solution:
\begin{equation}\label{ZKB-like subsol}
\underline{u}_1 (x,t) := \frac{1}{(t+1)^\beta e^{K(t+1)}} \left(C -\frac{(x-x_0)^2 e^{(m-1)K(t+1)}}{2m } \right)^\beta_+ ,\quad x\in \R,\ t\geq 0,
\end{equation}
with $w_+ := \max\{w(x,t),0\}$ and $C>0$ satisfying
$$
2m C e^{-(m-1)K}\leq \delta^2,\quad C^\beta \leq \varepsilon e^{K}.
$$
Using this subsolution, we can conclude that
$$
u(x,t)\geq \underline{u}_1 (x,t)>0 \mbox{ if } |x-x_0|< \underline{s}(t):= (2mC)^{1/2} e^{-\frac{(m-1)K(t+1)}{2}},\ t\geq 0.
$$
Note that, using this estimate we have $u>0$ in a shrinking domain $(x_0-\underline{s}(t),x_0 +\underline{s}(t))$, while using $\underline{u}$ we have $u>0$ in a fixed domain $(x_0-\delta, x_0+\delta)$.
\end{rem}

When $u_0\geq 0$ is compactly supported, besides the upper bound in Proposition \ref{prop:well}, we also have the {\it finite propagation speed} as in PME.

\begin{prop}\label{prop:finite speed}
Assume {\rm (I)} and $f$ is Lipschitz continuous  with $f(0)=0$. Then the support of $u(\cdot,t)$ lies in $[-\bar{s}(t),\bar{s}(t)]$ with
\begin{equation}\label{def-bar-s}
\bar{s}(t) := [C_1 (t+1)]^{1/2} e^{\frac{(m-1)K(t+1)}{2}}, \quad t>0,
\end{equation}
for some $C_1$ depending only on $m, K$ and $\|u_0\|_{L^\infty}$, and so $-l(t),r(t)\leq \bar{s}(t)$.
\end{prop}

\begin{proof}
We construct a supersolution of the ZKB form. Set $A := (4m\beta)^{-\beta}$, choose $C_1>0$ such that
$$
A e^K \left( C_1 -\frac{b^2}{e^{(m-1)K}} \right)^\beta \geq \|u_0\|_{L^\infty},
$$
where $b>0$ is the number in (I), and define
$$
\bar{u}(x,t):= A e^{K(t+1)} \left(C_1 -\frac{x^2}{(t+1)e^{(m-1)K(t+1)}} \right)^\beta_+, \quad x\in \R,\ t\geq 0.
$$
By the choose of $C_1$ we see that $\bar{u}(x,0)\geq u_0(x)$, and by a direct calculation we have
$$
\bar{u}_t \geq (\bar{u}^m)_{xx} +K\bar{u} \geq (\bar{u}^m)_{xx}+f(\bar{u}),\quad |x|\leq \bar{s}(t):= [C_1 (t+1)]^{1/2} e^{\frac{(m-1)K(t+1)}{2}},\ t\geq 0.
$$
Hence, $\bar{u}$ is a supersolution of (CP), and the conclusion follows from the comparison.
\end{proof}

\begin{rem}\rm
Our construction for the subsolution  in \eqref{ZKB-like subsol} and the supersolution in this proof is inspired by the ZKB solution, but with obvious difference due to the presence of the linear term $Ku$.
\end{rem}

\subsection{A priori estimate for the pressure}
As we have mentioned in Section 2, the pressure $v(x,t) := \frac{m}{m-1} [u(x,t)]^{m-1}$ solves
$$
{\rm (pCP)}\hskip 30mm
 \left\{
 \begin{array}{ll}
 v_t = (m-1) v v_{xx} + v_x^2 + g(v), & x\in \R,\ t>0, \\
 \displaystyle  v(x,0) = v_0(x) := \frac{m}{m-1} u_0^{m-1}(x), & x\in \R,
 \end{array}
 \right.
 \hskip 50mm
$$
with
$$
g(v) := m \Big( \frac{(m-1)v}{m}\Big)^{\frac{m-2}{m-1}} f
\left( \Big( \frac{(m-1)v}{m}\Big)^{\frac{1}{m-1}} \right)
$$
satisfying
\begin{equation}\label{bound-g-Lip}
|g(v)|\leq K (m-1) v, \quad v\geq 0,
\end{equation}
by (F). Since $u$ is bounded as in Proposition \ref{prop:well}, we see that $v$ is also bounded:
\begin{equation}\label{bound-v}
0\leq v(x,t) \leq M_0 := \max\left\{ \frac{m}{m-1}, \|v_0\|_{L^\infty} \right\},\quad x\in \R,\ t>0.
\end{equation}
Hence,
\begin{equation}\label{bound-g}
|g(v)|\leq G_0 (m,K,v_0):= K (m-1) M_0,\quad  0\leq v\leq M_0.
\end{equation}

By the above propositions we have the following result.

\begin{prop}\label{prop:well-v}
Assume {\rm (I)} and $f$ is Lipschitz continuous with $f(0)=0$. Then the problem {\rm (pCP)} has a unique nonnegative solution $v(x,t)\in C(Q_T)\cap L^\infty (Q_T)$ for any $T>0$.

Moreover, if $v(x_0, t_0)>0$ for some $x_0\in \R$ and $t_0>0$, then $v(x_0,t)>0$ for all $t\geq t_0$, and so $v$ is a classical solution near the line $\{(x_0, t) \mid t\geq t_0\}$.

The support of $v(\cdot,t)$ lies in $[-\bar{s}(t), \bar{s}(t)]$ with $\bar{s}(t)$ given by \eqref{def-bar-s}.
\end{prop}

In what follows, we continue to present the a priori estimates for $v_x,\ v_{xx}$ and $v_t$ under further restrictions on $f$: $f\in C^2$. In this case we have
\begin{equation}\label{bound-g'}
|g'(v)| \leq G_1 (m,K,v_0)\quad \mbox{and}\quad |g''(v)|\leq G_2 (m,K,v_0),\quad 0\leq v\leq M_0.
\end{equation}
In a similar way as Aronson \cite{A1969} (with obvious modification caused by the reaction term $g(v)$) we have

\begin{lem}[\cite{A1969}, uniform estimate for $v_x$]\label{lem:C1-est}
Assume {\rm (I)} and $f\in C^1([0,\infty))$ with $f(0)=0$. Let $v$ be a smooth positive classical solution of  {\rm (pCP)} in $R:= (a,b)\times (0,T]$ for some $a,b,T\in \R$ with $b>a,\ T>0$. Then for any $0< \delta <\frac{b-a}{2}, \ \tau < T$ there holds:
\begin{equation}\label{C1 bound}
|v_x(x,t)|\leq M_1 (m, M_0,G_0, G_1, \delta, \tau),\quad (x,t)\in [a+\delta, b-\delta]\times [\tau, T].
\end{equation}

Moreover, if $\tau =0$ and
$$
M_1^0 := \max\limits_{[a,b]} |v'_0(x)| <\infty,
$$
then \eqref{C1 bound} holds in $(a+\delta, b-\delta)\times (0,T]$ for $M_1 =M_1(m, M_0, G_0, G_1, \delta, M_1^0)$.
\end{lem}

Note that the a priori bound $M_1$ of $v_x$ obtained in this lemma is uniform in $T$ due to the fact that the bound of $v$ is so. The result holds for classical solutions.
For the very weak solution in $C(Q_T)\cap L^\infty(Q_T)$, however, one can show the following Lipschitz continuity for $v$ and H\"{o}lder continuity for $u$ (which can be regarded as the limit of a sequence of classical solutions).

\begin{prop}[\cite{A1969}, Lipschitz continuity for $v$ and H\"{o}lder continuity for $u$]\label{prop:regular weak sol}
Assume {\rm (I)} and $f\in C^1([0,\infty))$ with $f(0)=0$. If we further assume that $u^m_0$ is Lipschitz continuous, then for $T>\tau>0$,
\begin{equation}\label{Lip-p}
|v(x,t)-v(y,t)| \leq C_1(m,\tau,\|u_0\|_{L^\infty}) \cdot |x-y|,\quad x,y\in \R,\ t\in [\tau, T],
\end{equation}
\begin{equation}\label{Holder-u}
|u(x,t) -u(y,t)| \leq C_2 (m,\tau,\|u_0\|_{L^\infty}) \cdot |x-y|^\nu ,\quad x,y\in \R,\ t\in [\tau, T],\ \nu := \min\Big\{1, \frac{1}{m-1}\Big\}
\end{equation}
and $\frac{\partial u^m}{\partial x}(x,t)$ exists and is continuous in $x\in \R$, with $\frac{\partial u^m}{\partial x}(x,t)=0$ if $u(x,t)=0$;
\end{prop}

\begin{proof}
The proof follows from the previous lemma and the fact:  the very weak solution $u$ is the pointwise limit of a decreasing sequence of positive classical functions. (For the PME, this was shown in \cite{OKC}. For the current problem, the proof is similar).
\end{proof}

Next we consider the lower bound of $v_{xx}$. For the PME, Aronson and B\'{e}nilan \cite{AB} gave a lower bound for $v_{xx}$ in 1979:
\begin{equation}\label{AB-est}
v_{xx} \geq - \frac{1}{(m+1)t},\quad t>0.
\end{equation}
This inequality, usually known as the Aronson-B\'{e}nilan estimate, is understood in the sense of distributions. It is optimal in the sense that equality is actually attained by the source-type or ZKB solutions. It is a significant novelty of the Cauchy problem, and is used so often in the theory of nonnegative solutions in the whole space. For our problem (pCP), we will also give a lower bound for $v_{xx}$, but with quite different order from \eqref{AB-est}.

\begin{prop}\label{prop:2-est}
Assume {\rm (F)}, {\rm (I)}, and $v_0\in C^2$ a.e. in $\R$, with
$$
0\leq v_0(x) \leq M_0^0,\quad M_1^0 := \sup\limits_{\R} |v'_0(x)| <\infty,\quad v''_0(x)\geq -M_2^0 >-\infty,\quad a.e.\ x\in \R.
$$
Let $v\in C(Q_T)\cap L^\infty (Q_T)$ for any $T>0$ be the solution of  {\rm (pCP)}. Then
$$
v_{xx} (x,t)\geq - C(t+\tau),\quad x\in \R,
\ t\in (0,T],
$$
for some $\tau =\tau(m, M_0^0, M_1^0, M_2^0, G_0, G_1, G_2)$ and $C=C(m, M_0^0, M_1^0, G_0, G_1, G_2)$, here the inequality holds in the sense of distributions in $\R\times (0,T]$.
\end{prop}

\begin{proof}
We follow the idea in Aronson and B\'{e}nilan \cite{AB} (see also \cite[\S 9.3]{Vaz-book}).

1). First we assume $v>0$ in $Q_T :=\R\times (0,T]$. In this case, $v$ is a classical solution of (pCP). We write the equation satisfied by $\eta := v_{xx}$ by differentiating the equation of $v$ twice:
$$
\eta_t = \mathcal{L} (\eta) + g'' v_x^2,
$$
with
$$
\mathcal{L}(\eta) := (m-1)v\eta_{xx} + 2 m v_x \eta_x + (m+1)\eta^2 + g'(v)\eta.
$$
Note that here we use the $C^2$ smoothness assumption for $f$.
By Lemma \ref{lem:C1-est} we have
$$
|v_x(x,t)|\leq M_1 (m,M_0^0, M_1^0, G_0, G_1),\quad x\in \R,\ t>0.
$$
Hence,
$$
\eta_t \geq \mathcal{L}(\eta) - G_2 (M_1+1)^2.
$$

Now we construct a subsolution  of  this operator. Define
\begin{equation}\label{def-tau}
\tau_1 := \frac{G_1 + (m+1) M_2^0}{(m+1)G_2 (M_1+1)^2}  .
\end{equation}
Then $\zeta(x,t) := - G_2 (M_1 +1)^2 (t+\tau_1)$ satisfies
$$
\zeta_t - \mathcal{L}(\zeta) + G_2 (M_1 +1)^2 = G_2 (M_1 +1)^2 (t+ \tau_1 ) [G_1 - (m+1) G_2 (M_1 +1)^2 (t+\tau_1)]\leq 0.
$$
In addition,
$$
\zeta(x,0) = - G_2 (M_1 +1)^2 \tau_1 < -M_2^0 \leq v''_0(x)= \eta(x,0+),\quad x\in \R.
$$
By comparison we have
$$
v_{xx}(x,t) = \eta(x,t) \geq \zeta(x,t)= - G_2 (M_1 +1)^2 (t+\tau_1), \quad x\in \R,\ t>0.
$$

2). We now construct a family of approximate positive solutions to approach the nonnegative general solution. Set
$$
u_{0\varepsilon}(x) := u_0(x) + \varepsilon,\quad x\in \R, \ 0<\varepsilon \ll 1.
$$
Then the corresponding pressure is $v_{0\varepsilon} := \frac{m}{m-1} u_{0\varepsilon}^{m-1}$, which satisfies
$$
0<\varepsilon \leq v_{0\varepsilon} (x) \leq M_0^0 +1,
\quad
|v'_{0\varepsilon}(x)|\leq M_1^0,\quad v''_{0\varepsilon} (x)\geq -M_2^0,\quad x\in \R.
$$
According to the standard theory, the problem (CP) with initial data $u_{0\varepsilon}$ has a unique solution $u_\varepsilon (x,t)$. In a similar way as in the previous step, we see that the corresponding pressure $v_\varepsilon$ satisfies
\begin{equation}\label{est-v-ep-0}
0< v_\varepsilon(x,t) \leq \max\Big\{\frac{m}{m-1}, M_0^0 +1\Big\},\quad x\in \R,\ t>0,\ 0<\varepsilon\ll 1,
\end{equation}
$v_{\varepsilon x}$ satisfies
$$
|v_{\varepsilon x}(x,t)|\leq M'_1 (m,M_0^0, M_1^0, G_0, G_1),\quad x\in \R,\ t>0,\ 0<\varepsilon \ll 1,
$$
and $v_{\varepsilon xx}$ satisfies
\begin{equation}\label{est-v-ep-2}
v_{\varepsilon x x} (x,t)\geq -C (t+ \tau),\quad x\in \R,\ t\in (0,T],\ 0<\varepsilon\ll 1,
\end{equation}
for
$$
C= C(m,M_0^0, M_1^0, G_0, G_1, G_2) \ \ \mbox{ and } \ \
\tau = \tau(m,M_0^0, M_1^0, M_2^0, G_0, G_1, G_2).
$$
Since $v_{\varepsilon}$ is strictly decreasing in $\varepsilon$ and $v_\varepsilon \geq v$. There exists $\hat{v}\geq v\geq 0$ such that
$$
\lim\limits_{\varepsilon\to 0} v_\varepsilon (x,t) = \hat{v}(x,t) \mbox{ in the topology of } C_{loc}(Q_T).
$$
(In particular, $\hat{v}\in C(Q_T)\cap L^\infty(Q_T)$). Furthermore, in the area $D:=\{(x,t)\in Q_T \mid \hat{v}(x,t)>0\}$, by standard interior parabolic estimates we see that the convergence holds also in the topology $C^{2,1}_{loc}(D)$.
Using the Lebesgue theorem to take limit as $\varepsilon\to 0$ in
$$
\iint [v_\varepsilon \varphi_{xx} +C(t+\tau)\varphi] dxdt = \iint [v_{\varepsilon xx} +C(t+\tau)]\varphi dxdt \geq 0
$$
for every $\varphi \in C^\infty_c (Q_T),\ \varphi \geq 0$, we obtain
$$
\iint [\hat{v}\varphi_{xx} +C(t+\tau)\varphi] dxdt\geq 0.
$$
This means that \eqref{est-v-ep-2} holds for $\hat{v}$ in the distribution sense.

To complete our proof we only need to show that $\hat{v}\equiv v$.
Denote by
$$
\hat{u}=\Big( \frac{m-1}{m} \hat{v} \Big)^{\frac{1}{m-1}} \in C(Q_T)\cap L^\infty(Q_T)
$$
the original density variable corresponding to the pressure $\hat{v}$.
Recalling that $u_\varepsilon$ is a classical solution of  (CP) with initial data $u_{0\varepsilon}$ we have
$$
\int_{\R} u_\varepsilon (x,T)\varphi(x,T) dx  =  \int_{\R} u_{0\varepsilon} (x) \varphi(x,0)dx +
\iint_{Q_T} f(u_\varepsilon) \varphi  dx dt  + \iint_{Q_T} [u_\varepsilon\varphi_t + u^m_\varepsilon \varphi_{xx}] dx dt.
$$
Taking limit as $\varepsilon\to 0$ we have
$$
\int_{\R} \hat{u} (x,T)\varphi(x,T) dx  =  \int_{\R} u_{0} (x) \varphi(x,0)dx +
\iint_{Q_T} f(\hat{u}) \varphi  dx dt  + \iint_{Q_T} [\hat{u}\varphi_t + u^m  \varphi_{xx}] dx dt.
$$
This means that $\hat{u}$ is a very weak solution of  (CP) with $\hat{u}(x,0)=u_0$.
By the uniqueness, $\hat{u}$ must be $u$, and so $\hat{v} \equiv v$.
\end{proof}

Finally, we show that $v$ is also Lipschitz in time.

\begin{lem}\label{lem:bound-vt}
Assume the hypotheses in Proposition \ref{prop:2-est} hold.
Then for any $\tau>0$ there holds
\begin{equation}\label{upper bound vt}
v_t(x,t) \leq C_1 t +C_2,\quad x\in \R,\ t\geq \tau,
\end{equation}
where $C_1$ depends on $m, M^0_1, G_0$ and $G_1$, $C_2$ depends on $m, G_0, G_1, v_t(x,\tau)$, and
\begin{equation}\label{lower bound vt}
v_t \geq - C_3 t -C_4,\quad x\in \R,\ t>0,
\end{equation}
where $C_3$ and $C_4$ depend on $m, M_0^0, M_1^0, G_0, G_1, G_2$.
\end{lem}

\begin{proof}
By approximation, we may assume that $v$ is positive and smooth. In this case we have
$$
0<v\leq 1+\|v_0\|_{L^\infty},\quad |v_x(x,t)|\leq M_1(m,M_1^0, G_0, G_1),\quad x\in \R,\ t>0,
$$
as above. We consider the function
$$
P(x,t) := v_t + (m-1) v_x^2 = (m-1)vv_{xx} + mv_x^2 + g(v).
$$
Then
$$
\begin{array}{l}
 P_x = (m-1) vv_{xxx} + (3m-1)v_x v_{xx} + g'(v) v_x,\\
 P_{xx} = v_{txx} + 2(m-1) (v_{xx}^2 + v_x v_{xxx}),\\
 P_t = (m-1)v v_{txx} + (m-1)v_t v_{xx} + 2m(m+1)v_x^2 v_{xx} + 2m (m-1)vv_x v_{xx}\\
 \ \ \ \ \ \ \ + 2m v_x g'(v) v_x + g'(v) v_t.
\end{array}
$$
Hence we have
\begin{eqnarray*}
\mathcal{L}(P) & := & P_t - (m-1)v P_{xx} - 2v_x P_x \\
& = & \widetilde{g}(x,t,P) := \frac{1}{v}[-(P-g)^2 + A_1 (P-g) +A_2] + A_3,
\end{eqnarray*}
with
$$
\begin{array}{l}
A_1 := (4m-1)v_x^2 + vg',\qquad A_2 := -m(3m-1) v_x^4 ,\\
A_3 := 2m v_x g'(v) v_x + g'(v) (g-(m-1)v_x^2).
\end{array}
$$
By $f\in C^2$ and the bounds of $v,\ v_x$ we see that
$$
|A_i(x,t)| \leq \bar{A}_i , \quad x\in \R,\ t\geq \tau,\ i=1,2,3 .
$$
Choose
$$
C_1 := \bar{A}_3,\quad C_2 := G_0 + \frac{\bar{A}_1 +\sqrt{\bar{A}_1^2 +4 \bar{A}_2}}{2} + (m-1)M_1^2 + \sup\limits_{x\in \R} v_t (x,\tau),
$$
then
$$
-(C_1 t +C_2-g)^2 + A_1 (C_1 t + C_2 -g) +A_2 \leq 0
$$
and so
$$
\mathcal{L}(C_1 t + C_2) - \widetilde{g}(x,t,C_1 t + C_2) \geq 0
$$
and
$$
(C_1 t + C_2)|_{t=\tau} > C_2 > P(x,\tau).
$$
Hence $C_1 t + C_2$ is a supersolution of $\mathcal{L}(P)=\widetilde{g}(x,t,P)$ in $t\geq \tau$, and so the \eqref{upper bound vt} follows from the maximum principle.
Note that the auxiliary function $P$ was used for PME in \cite{Ben1} and \cite[Chapter 15]{Vaz-book}, but the supersolution they used is $-C/t$, different from ours.

The lower bound for $v_t$ is an immediate consequence of Proposition \ref{prop:2-est}:
$$
v_t = (m-1)vv_{xx} +v_x^2 +g(v) \geq -C (m-1)(t+\tau) (1+\|v_0\|_{L^\infty}) - G_0.
$$
This proves the Lipschitz continuity for $v$ in time.
\end{proof}

\end{document}